\documentclass[10pt]{amsart}
\usepackage{amssymb,mathrsfs}
\usepackage{color,umoline}
\usepackage[dvipsnames]{xcolor}
\usepackage{graphicx}
\usepackage{enumitem}
\usepackage[font=small,labelfont=bf]{caption}
\usepackage{tikz, caption, subcaption}
\usepackage{relsize}
\usepackage[mathscr]{eucal}
\usepackage{verbatim}
\usepackage[draft]{todonotes}
\usepackage[pdftex]{hyperref}
\usepackage[dvipsnames]{xcolor}
\usetikzlibrary{patterns}
\tikzstyle{EDR}=[draw=lightgray,line width=0pt,preaction={clip, postaction={pattern=north east lines, pattern color=gray}}]
\tikzstyle{EDR1}=[draw=lightgray,line width=0pt,preaction={clip, postaction={pattern=north west lines, pattern color=gray}}]

\def\normo#1{\left\|#1\right\|}
\def\wt#1{\widetilde{#1}}
\def\wh#1{\widehat{#1}}

\newcommand{\F}{\mathcal{F}}

\newcommand{\C}{\mathbb{C}}
\newcommand{\N}{\mathbb{N}}
\newcommand{\R}{\mathbb{R}}
\newcommand{\Z}{\mathbb{Z}}

\newcommand{\al}{\alpha}
\newcommand{\be}{\beta}
\newcommand{\ga}{\gamma}
\newcommand{\de}{\delta}
\newcommand{\e}{\varepsilon}
\newcommand{\la}{\lambda}
\newcommand{\s}{\sigma}
\newcommand{\ta}{\tau}

\newcommand{\x}{\xi}
\newcommand{\De}{\Delta}

\newcommand{\p}{\partial}
\newcommand{\re}{\mathop{\mathrm{Re}}}
\newcommand{\im}{\mathop{\mathrm{Im}}}
\newcommand{\supp}{\operatorname{supp}}
\newcommand{\dist}{\operatorname{dist}}
\newcommand{\rank}{\operatorname{rank}}

\newcommand{\zs}{\mathbb S^1\setminus \{1\}}
\newcommand{\mf}{\mathfrak}
\newcommand{\phim}{[\Phi, \mf a]}
\newcommand{\mI}{{\mathbf I}}
\newcommand{\mm}{\mathfrak M}
\newcommand{\pyd}{\Phi^{y_d}}
\newcommand{\Ell}{{\bf Ell}(N,\epsilon)}
\newcommand{\elne}{{{\bf Ell}(N,\epsilon)}}
\newcommand{\mnb}{{{\bf Mul}({ N}, b)}}

\numberwithin{equation}{section}

\newtheorem{thm}{Theorem}[section]
\newtheorem{cor}[thm]{Corollary}
\newtheorem{lem}[thm]{Lemma}
\newtheorem{defn}[thm]{Definition}
\newtheorem{prop}[thm]{Proposition}

\newtheorem{conj}{Conjecture}
\theoremstyle{remark}
\newtheorem{rem}{Remark}

\begin{document}

\title[Sharp resolvent estimates]{
Sharp resolvent estimates
\\
outside of the uniform boundedness range
}

\author[Y. Kwon]{Yehyun Kwon}
\author[S. Lee]{Sanghyuk Lee}

\address{School of Mathematics, Korea Institute for Advanced Study,  Seoul 02455, Republic of Korea}
\email{yhkwon@kias.re.kr}
\address{Department of Mathematical Sciences and RIM, Seoul National University, Seoul 08826, Republic of Korea}
\email{shklee@snu.ac.kr}


\begin{abstract}
In this paper we are concerned with resolvent estimates for the Laplacian $\Delta$ in Euclidean spaces. Uniform resolvent estimates for $\Delta$ were shown by Kenig, Ruiz and Sogge \cite{KRS} who established rather a complete description of the Lebesgue spaces allowing such estimates. However, the problem of obtaining sharp $L^p$--$L^q$ bounds depending on $z$ has not been considered in a general framework which admits all possible $p,q$.  In this paper, we present a complete picture of sharp $L^p$--$L^q$ resolvent estimates, which may depend on $z$.  We also obtain the sharp resolvent estimates for the fractional Laplacians and a new result for the Bochner--Riesz operators of negative index.
\end{abstract}
\maketitle

\section{Introduction and main results}\label{intro}
In this paper we are concerned with the  resolvent estimate for the Laplacian which is of the form 
\begin{equation}\label{resol0}
\|(-\Delta -z)^{-1}  f\|_{L^q(\R^d)}\le C\|f\|_{L^p(\R^d)}, \quad \forall z\in\C\setminus[0,\infty) .
\end{equation}
When $z=0$ the estimate is simply  the classical Hardy--Littlewood--Sobolev inequality.  If $z\in (0,\infty)$ the left hand side cannot be defined even as a distribution without additional assumption. Throughout this article we assume $z\in\C\setminus[0,\infty)$. The inequality  \eqref{resol0} and its variants  (especially, with $C$ independent of $z$)  have applications  to various related problems. Among them are uniform Sobolev estimates, unique continuation properties  \cite{KRS, JKL}, limiting absorption principles \cite{GS}, absolute continuity of the spectrum of periodic Schr\"odinger operators \cite{Shen}  and eigenvalue bounds for  Schr\"odinger operators with complex potentials \cite{Fr11, Fr18}.  As just mentioned, \eqref{resol0} has been usually considered with $C$ independent of $z$ but the sharp bounds which are allowed to be dependent on $z$ are not studied in a general framework. The primary purpose of this paper is to provide  complete characterization of the sharp  $L^p$--$L^q$ bounds for the resolvent operators up to  a multiplicative constant.  

\subsection*{Uniform resolvent estimate} In their celebrated work \cite{KRS}  Kenig, Ruiz and Sogge showed that,  for certain pairs of $p, q$, the constant $C$  in \eqref{resol0} can be chosen uniformly in $z\in\C\setminus[0,\infty)$. More precisely, for $d\ge 3$, it was shown  that there is a uniform constant $C=C(p,q,d)>0$ such that \eqref{resol0} holds  if and only if  $1/p-1/q=2/d$ and $\frac{2d}{d+3}<p<\frac{2d}{d+1}$, or equivalently $(1/p, 1/q)$ lies on the open line segment whose endpoints are  
\begin{equation}\label{AA'3}
	A=A(d):=\Big(\frac{d+1}{2d}, \frac{d-3}{2d} \Big), \quad A'=A'(d):=\Big(\frac{d+3}{2d}, \frac{d-1}{2d}\Big), \quad d\ge 3.
\end{equation} 
See Figure \ref{fig2}.  They used these estimates to show {\it uniform Sobolev estimates} for second order elliptic differential operators on the same range of $p, q$ (see \cite[Theorem 2.2]{KRS}).  When $(\frac1p,\frac1q)=(\frac{d+2}{2d}, \frac{d-2}{2d})$\footnote{The midpoint of the line segment $AA'$ in Figure \ref{fig2}.} the same estimate was also obtained by Kato and Yajima \cite[pp. 493--494]{KY} by a different approach.  

The result in \cite{KRS}  gives complete characterization of the range of $p, q$ which admits  the uniform resolvent estimate. However, it is not difficult to see that if  $C$ in \eqref{resol0} is allowed to be dependent on $z\in\C\setminus[0,\infty)$,  there is a larger set of $p,q$ for which the estimate \eqref{resol0} holds. To be precise, for  $z\in \C\setminus[0,\infty)$  let us set 
\[	\|(-\Delta-z)^{-1}\|_{p\to q}:=\inf\Big\{C_z:   \|(-\Delta -z) ^{-1} f \|_{L^q(\R^d)}\le C_z\| f \|_{L^p(\R^d)}, ~ \forall f\in  \mathcal S(\R^d)  \Big\},	\]
where $\mathcal S(\R^d)$ denotes the space of Schwartz functions on $\R^d$.

\begin{prop}\label{eresb}
Let $d\ge2$, $1\le p, q \le \infty$ and $z\in\C\setminus [0,\infty)$. Then $\|(-\Delta-z)^{-1}\|_{p\to q}<\infty$  if and only if  $(1/p, 1/q)\in \mathcal R_0$ which is given by 
\[	\mathcal R_0=\mathcal R_0(d):= 
	\begin{cases}
	\big\{ (x,y) : 0\le x, y\le 1, ~ 0\le x-y <1 \big\}
		   &\text{ if } ~  d=2,\\[5pt]
	\big\{(x,y) : 0\le x, y\le 1,~ 0\le x-y \le  \frac2d\big\} \setminus \big\{ \big(1, \frac{d-2}d \big) , \big(\frac 2d, 0\big) \big\}	
		   &\text{ if } ~  d\ge 3.
	\end{cases}	\]
\end{prop} 

In view of  Proposition \ref{eresb} it is natural to ask what is the sharp value of $\|(-\Delta-z)^{-1}\|_{p\to q}$ which depends on $z$.\footnote{Sharpness here refers to the optimal dependence of $\|(-\Delta-z)^{-1}\|_{p\to q}$ on the spectral parameter $z$.}  For some $p, q$ such estimate (modulo  a constant multiplication)  can be deduced by interpolation between estimates in \cite{KRS, Gu} and  the easy bound 
\begin{equation} 
	\|(-\Delta-z)^{-1} f \|_2 \le  \frac{ \| f \|_2}{\dist(z,[0,\infty))} ,
\end{equation}
which directly  follows from the Fourier transform and Plancherel's identity.  Some of related results can be found in  \cite{Fr18}. Moreover, these estimates turn out to be sharp (see \eqref{sc} and Proposition \ref{shp} below). But, the sharp bound  for $\|(-\Delta-z)^{-1}\|_{p\to q}$ with general $p, q$ cannot be deduced from interpolation between previously known estimates. For the purpose we need to make use of $L^p$ theory of oscillatory integral operators of Carleson--Sj\"olin type under the additional elliptic condition (\cite{CS, Ho72, St-book, Lee2, GHI}, also see Section \ref{sect2-1} below).

\subsection*{Boundedness of the associated multiplier operators} To obtain the sharp resolvent estimates, it is convenient to consider bounds for the associated multiplier operators. Clearly, 
\begin{equation}\label{mult}
\| (-\Delta - z)^{-1}\|_{p\to q}
 	=\sup_{\|f\|_p\le 1} \Big\| \F^{-1} \Big( \frac{\F f(\xi)}{|\xi|^2-z} \Big) \Big\|_{L^q(\R^d)}, \quad \forall z\in\C\setminus[0,\infty).
\end{equation}
Here $\F$ and $\F^{-1}$ denote the Fourier and inverse Fourier transforms on $\R^d$, respectively. Since the multiplier $(|\xi|^2-z)^{-1}$ becomes singular as $z$ approaches to the set $[0,\infty)$ it is reasonable to expect that the bound $\| (-\Delta - z)^{-1}\|_{p\to q}$ gets worse  as $\dist(z, [0,\infty))\to 0$.  Thanks  to homogeneity  and scaling, we have that 
\begin{equation}\label{sc}
\| (-\Delta - z)^{-1}\|_{p\to q} = |z|^{-1+\frac d2(\frac 1p-\frac 1q)} \Big\| \Big(-\Delta -\frac{z}{ |z|}\Big)^{-1} \Big\|_{p\to q}, \quad \forall z\in \C\setminus [0,\infty).
\end{equation}
Thus we may assume that  $|z|=1$, $z\neq 1$ to get the sharp bounds for $\| (-\Delta - z)^{-1}\|_{p\to q}$.  Indeed, when $d\ge 3$, it was shown in \cite{KRS} that there is a uniform constant $C$, independent of $z$, such that
\begin{equation}\label{unif}
\| (-\Delta - z)^{-1}\|_{p\to q} \le C, \quad  \forall z\in \zs \footnote{ $\mathbb S^1:=\{z\in\C: |z|=1\}$.}
\end{equation}
if $(1/p, 1/q)$ lies in either the open line segment of which endpoints are  $A$ and $A'$ (see \eqref{AA'3} and Figure \ref{fig2}), or the line of duality $1/p+1/q=1$ restricted to $\frac{d+3}{2d+2}\le \frac1p \le \frac{d+2}{2d}$ (see \cite[Lemma 2.2(b) and Theorem 2.3]{KRS}).  Later,  Guti\'{e}rrez (\cite[Theorem 6]{Gu}) extended \eqref{unif} to the optimal range of $p,q$. More precisely,  she  proved that the uniform bound \eqref{unif} is true if  $(1/p,1/q)$ lies in the set
\[	\mathcal R_1=\mathcal R_1(d):= \Big\{ (x,y)\in\mathcal R_0(d) :   \frac{2}{d+1}\le x-y\le \frac 2d, ~ x>\frac {d+1}{2d}, ~ y<\frac{d-1}{2d} \Big\}, \quad d\ge 3.	\]
This region is the closed trapezoid $ABB'A'$ from which the closed line segments joining $A,B$ and $A',B'$ are removed (see Figure \ref{fig2}).  She also established the $L^{p,1}$--$L^{q,\infty}$ (restricted weak type) analogues of \eqref{unif} when $(1/p, 1/q)$ is either $B$ or $B'$, where
\begin{equation}\label{BBprime}
	B=B(d):=\Big(\frac{d+1}{2d}, \frac{(d-1)^2}{2d(d+1)} \Big),  \quad B'=B'(d):=\Big(\frac{d^2+4d-1}{2d(d+1)}, \frac{d-1}{2d} \Big).
\end{equation} 

Failure of \eqref{unif} for $(1/p, 1/q)\notin \mathcal R_1$ has been actually known before in the studies of the Bochner--Riesz operators of negative orders (see Section \ref{REBR}). In fact, the necessity of  the conditions $\frac1p>\frac{d+1}{2d}$ and $\frac1q<\frac{d-1}{2d}$ follow since  \eqref{unif} combined with \eqref{mult} implies $L^p$--$L^q$ boundedness of the restriction-extension operator on the sphere (see Theorem \ref{rest_ext_sphere} and  \cite[pp. 341--342]{KRS}), which is a constant multiple of the Bochner--Riesz operator \eqref{nbr} of order $-1$. The other two conditions $ \frac{2}{d+1}\le \frac1p-\frac1q$ and $ \frac1p-\frac1q \le \frac 2d$ can be obtained by the Knapp type example (see B\"orjeson \cite{Bo}) and a simple argument involved with the Littlewood--Paley projection (see Proof of Proposition \ref{eresb}), respectively. 

When $d=2$, as far as the authors are aware,  the corresponding results regarding the uniform resolvent estimate \eqref{unif} are not explicitly stated anywhere else before,  although the $L^p$--$L^q$ mapping properties of the closely related Bochner--Riesz operaters of negative order are well known (see e.g., \cite{bak, CKLS} and references therein). However, the method in \cite{Gu} can be applied to obtain \eqref{unif}  provided that $(1/p,1/q)$ is contained in the pentagon 
\[	\mathcal R_1(2):= \{(x,y): 2/3\le x-y <1, \, 3/4<x\le1, \, 0\le y<1/4 \}.	\] 
See Figure \ref{fig1} and Remark \ref{rmk1}.

\begin{figure}
\captionsetup{type=figure,font=footnotesize}
\begin{minipage}[b]{0.45\textwidth}
\centering
\begin{tikzpicture} [scale=0.6]\scriptsize
	\draw [<->] (0,10.7)node[above]{$y$}--(0,0) node[below]{$(0,0)$}--(10.7,0) node[right]{$x$};
	\draw (0,10) --(10,10)--(10,0) node[below]{$(1,0)$};
	\draw (2.5,2.5)node[left]{$D$}--(30/4,10/12)node[above]{$B$};
	\draw (7.5,7.5)node[above]{$D'$}--(110/12,10/4)node[left]{$B'$};
	\draw (0,0)--(10,10);
	\draw (30/4,10/12)--(110/12,10/4);
	\draw (19/3,11/3) node{$\mathcal R_2$};
	\draw (30/8, 1) node{$\mathcal R_3$};
	\draw (9, 50/8) node{$\mathcal R_3'$};
	\draw (9, 1) node{$\mathcal R_1$};
	\draw [dash pattern={on 2pt off 1pt}]  (0,5)node[left]{$\frac12$}--(5,5)node[above]{$H$}--(5,0)node[below]{$\frac12$}; 
	\draw [dash pattern={on 2pt off 1pt}] (30/4,10/12)--(30/4,0)node[below]{$E$};
	\draw [dash pattern={on 2pt off 1pt}] (110/12,10/4)--(10,10/4)node[right]{$E'$};
\end{tikzpicture}\caption{The case $d=2$.}\label{fig1}
\end{minipage}\hfill
\begin{minipage}[b]{0.46\textwidth}
\centering
\begin{tikzpicture} [scale=0.6] \scriptsize
	\draw [<->] (0,10.7)node[above]{$y$}--(0,0) node[below]{$(0,0)$}--(10.7,0) node[right]{$x$};
	\draw (0,10) --(10,10)--(10,0) node[below]{$(1,0)$};
	\draw (4,4)node[left]{$D$}--(6,8/3)node[above]{$B$}--(10-8/3,4)node[left]{$B'$}--(6,6)node[above]{$D'$};
	\draw (0,0)--(10,10);
	\draw (4,0)node[below]{$\frac 2d$}--(10,6);
	\draw [dash pattern={on 2pt off 1pt}]  (0,5)node[left]{$\frac12$}--(5,5)node[above]{$H$}--(5,0)node[below]{$\frac12$}; 
	\draw [dash pattern={on 2pt off 1pt}] (6, 8/3)--(6, 2)node[below]{$A$}--(6,0)node[below]{$E$};;
	\draw [dash pattern={on 2pt off 1pt}] (10-8/3, 4)--(8, 4)node[right]{$A'$}--(10,4)node[right]{$E'$};;
	\draw (5.6, 4.4) node{$\mathcal R_2$};
	\draw (30/8, 2) node{$\mathcal R_3$};
	\draw (8, 50/8) node{$\mathcal R_3'$};
	\draw (6.9, 3.2) node{$\mathcal R_1$};
\end{tikzpicture}\caption{The case $d\ge3$.}\label{fig2}
\end{minipage}
\end{figure}

\subsection*{Conjecture regarding $L^p$--$L^q$ resolvent estimate with  $(1/p, 1/q)\in\mathcal R_0\setminus \mathcal R_1 $}  
Having seen that we have the uniform bound \eqref{unif} on the optimal range $\mathcal R_1$, we now proceed to investigate the (non-uniform) sharp bounds with $p,q$ which lie outside of the uniform boundedness range.  As becomes clear later, the problem is closely related to sharp $L^p$--$L^q$ boundedness of  the Bochner-Riesz operators of negative orders (see Section \ref{REBR}).  The non-uniform bounds on the resolvents have been used to study eigenvalues of the Schr\"odinger operators with complex potentials (for example, see \cite{Fr18, Cu15}).

In order to state our results we introduce some notations which denote points and regions in the closed unit square $I^2:=\{(x,y)\in\R^2: 0\le x,y\le 1\}$.  For each $(x,y) \in I^2$  we set 
\[	(x,y)' :=(1-y, 1-x).	\]
Similarly, for every subset $\mathcal R$ of $I^2$ we define $\mathcal R'\subset I^2$ by 
\[	\mathcal R':=\{ (x,y)\in I^2: (x,y)' \in \mathcal R \}.\]

\begin{defn}  For   $X_1, \cdots, X_\ell\in I^2$, we denote by $[X_1, \cdots, X_\ell]$ the convex hull of  the points $X_1, \cdots, X_\ell$. In particular, if  $X, Y\in I^2$,  $[X,Y]$  denotes  the closed line segment  connecting $X$ and $Y$ in $I^2$. We also denote by $(X,Y)$ and $[X,Y)$  the open interval $[X,Y]\setminus\{X,Y\}$ and the half-open interval $[X,Y]\setminus\{Y\}$, respectively. 
\end{defn}

For every $d\ge2$ and every $(1/p, 1/q)\in I^2$, define a nonnegative number
\begin{equation}\label{def-gamma}
	\gamma_{p,q} =\gamma_{p,q}(d) : =\max \Big\{ 0, ~ 1-\frac{d+1}2\Big(\frac 1p-\frac 1q \Big),~ \frac{d+1}2-\frac dp, ~ \frac dq-\frac{d-1}{2} \Big\}.
\end{equation}
The definition of $\ga_{p,q}$ naturally leads to division of $\{(x,y)\in I^2: y\le x\}$ into the four regions
\begin{align}
\label{pentagon}\mathcal P=\mathcal P(d)
	&:=\Big\{(x,y)\in I^2: x-y\ge \frac{2}{d+1}, ~ x>\frac{d+1}{2d}, ~y<\frac{d-1}{2d} \Big\}, \\
\label{trapezoid}\mathcal T=\mathcal T(d)
	&:=\Big\{(x,y)\in I^2: 0\le x-y<\frac{2}{d+1},~\frac{d-1}{d+1}(1-x) \le y \le \frac{d+1}{d-1} (1-x) \Big\}, \\
\label{quadrangle}\mathcal Q=\mathcal Q (d)
	&:=\Big\{(x,y)\in I^2: y< \frac{d-1}{d+1}(1-x),~ y\le x<\frac{d+1}{2d}\Big\}, 
\end{align}
and $\mathcal Q'$.\footnote{$\mathcal P=[(1,0),  E, B, B', E']\setminus ([E, B]\cup [E', B'] )$,  $\mathcal T=[B, D, D', B']\setminus [B, B']$,  and  $\mathcal Q=[(0,0), D, B, E]\setminus([D, B]\cup [B, E])$. See Figure \ref{fig1} and Figure \ref{fig2}.}  We  now observe that $\mathcal R_1(d)=\mathcal P(d)\cap\mathcal R_0(d)$. Setting $H:=(\frac12, \frac12)$ and  $D=D(d):=(\frac{d-1}{2d}, \frac{d-1}{2d})$ we also define $\mathcal R_2=\mathcal R_2(d)$ and $\mathcal R_3=\mathcal R_3(d)$  by
\[
\mathcal R_2 
	:= \mathcal T(d) \setminus \big([D,H)\cup [D',H) \big), \qquad
\mathcal R_3 
	:= \mathcal Q(d) \cap \mathcal R_0(d).
\]
See Figure \ref{fig1} and Figure \ref{fig2}. Observe that the sets $\mathcal R_i$ $(i=1,2,3)$ and $\mathcal R_3'$ are mutually disjoint. Setting $E=E(d):=(\frac{d+1}{2d},0)$ we have that
\[	\Big(\bigcup_{i=1}^3\mathcal R_i \Big) \cup \mathcal R_3'=\mathcal R_0 \setminus \big( [B,E] \cup [B', E'] \cup [D,H) \cup [D',H) \big),\]
and we also see  that
\begin{equation}\label{gamma} 
\gamma_{p,q}
=\left\{\begin{matrix}
0 & \quad \text{if} \quad (\frac 1p,\frac1q)\in \mathcal R_1, \\[4pt] 
1-\frac{d+1}2 (\frac 1p-\frac 1q ) & \quad \text{if} \quad (\frac 1p,\frac1q)\in\mathcal R_2, \\[4pt] 
\frac{d+1}2-\frac dp & \quad \text{if} \quad (\frac 1p,\frac1q)\in\mathcal R_3, \\[4pt] 
\frac dq-\frac{d-1}{2} & \quad \text{if}  \quad (\frac 1p,\frac1q)\in\mathcal R_3'.
\end{matrix}\right.
\end{equation} 

In Section \ref{shp1} we obtain the following lower bounds for $\|(-\Delta-z)^{-1}\|_{p\to q}$. 
\begin{prop}\label{shp} 
Let $d\ge 2$. Suppose that $(1/p, 1/q)\in  \big(\bigcup_{i=1}^3 \mathcal R_i\big) \cup \mathcal R_3'$. Then, for $ z\in \zs$,
\begin{equation}\label{shp2}
\|(-\Delta-z)^{-1}\|_{p\to q} \gtrsim \dist(z,[0,\infty))^{-\gamma_{p,q}},
\end{equation}
where the implicit constant is independent of $z\in\zs$.
\end{prop}

As mentioned in the above, when $(1/p,1/q)\in[B,E]\cup[B',E']$, $\sup_{z\in\zs} \|(-\Delta-z)^{-1}\|_{p\to q} = \infty$. For $(1/p,1/q)\in[D, H) \cup [D', H')$,  it is likely that  by adapting Fefferman's  disproof of  disk multiplier  conjecture \cite{F71} one can show $\sup_{z\in\zs}  \dist(z,[0,\infty))^{\gamma_{p,q}} \|(-\Delta-z)^{-1}\|_{p\to q} = \infty$.  However, for the other $p,q$ with  $(1/p, 1/q)\in \big(\bigcup_{i=2}^3 \mathcal R_i\big)\cup\mathcal R_3'$ it seems natural to expect  that  the lower bound in \eqref{shp2} is also an upper bound.  

For $p,q$ with $(1/p, 1/q)\in \big(\bigcup_{i=2}^3 \mathcal R_i\big)\cup\mathcal R_3'$ and $z\in \C\setminus [0,\infty)$, let us set  
\[  { \mathlarger \kappa}_{p,q}(z) ={ \mathlarger \kappa}_{p,q,d}(z):= |z|^{-1+\frac d2(\frac 1p-\frac 1q)+\gamma_{p,q}}  \dist(z,[0,\infty))^{-\gamma_{p,q}}.	\]
Since $\dist(|z|^{-1}z, [0,\infty))=|z|^{-1}\dist(z, [0,\infty))$,  from  Proposition  \ref{shp} and  \eqref{sc} we  conjecture the following which  completely characterizes the  resolvent estimates outside of the uniform boundedness range.

\begin{conj}\label{main_conj} Let $d\ge 2$ and $(1/p, 1/q)\in \mathcal R_2 \cup \mathcal R_3 \cup \mathcal R_3'$.  There exists an absolute constant $C$, depending only on $p$, $q$ and $d$, such that, for $z\in \C\setminus [0,\infty)$, 
\begin{equation}\label{con1}	
C^{-1}{\mathlarger \kappa}_{p,q}(z) \le \|(-\Delta-z)^{-1}\|_{p\to q} \le C  {\mathlarger \kappa}_{p,q}(z).
\end{equation}
\end{conj}

\subsection*{Sharp $L^p$--$L^q$ resolvent estimate with  $(1/p, 1/q)\notin \mathcal R_1$}
\begin{figure}
\captionsetup{type=figure,font=footnotesize}
\begin{minipage}[b]{0.45\textwidth}
\centering
\begin{tikzpicture}[scale=0.6] \scriptsize
	\path [fill=lightgray] (0,0)--(10,10)--(10,0)--(0,0);
	\draw [<->] (0,10.7)node[above]{$y$}--(0,0) node[below]{$(0,0)$}--(10.7,0) node[right]{$x$};
	\draw (0,10) --(10,10)--(10,0) node[below]{$(1,0)$};
	\draw (2.5,2.5)node[left]{$D$}--(30/4,10/12)node[above]{$B$};
	\draw (7.5,7.5)node[above]{$D'$}--(110/12,10/4)node[left]{$B'$};
	\draw (0,0)--(2.5,2.5);
	\draw (7.5,7.5)--(10,10);
	\draw (30/4,10/12)--(110/12,10/4);
	\draw (19/3,11/3) node{$\wt{\mathcal R}_2$};
	\draw (30/8, 1) node{$\wt{\mathcal R}_3$};
	\draw (9, 50/8) node{$\wt{\mathcal R}_3'$};
	\draw (9, 1) node{$\mathcal R_1$};
	\draw [dash pattern={on 2pt off 1pt}] (2.5,2.5)--(7.5,7.5);
	\draw [dash pattern={on 2pt off 1pt}]  (0,5)node[left]{$\frac12$}--(5,5)node[above]{$H$}--(5,0)node[below]{$\frac12$}; 
	\draw [dash pattern={on 2pt off 1pt}] (30/4,10/12)--(30/4,0)node[below]{$E$};
	\draw [dash pattern={on 2pt off 1pt}] (110/12,10/4)--(10,10/4)node[right]{$E'$};
\end{tikzpicture}\caption{Theorem \ref{thm} when $d=2$.}\label{figthm2}
\end{minipage}\hfill
\begin{minipage}[b]{0.5\textwidth}
\centering
\begin{tikzpicture} [scale=0.6]\scriptsize
	\path [fill=lightgray] (0,0)--(15/4,15/4)--(50/11, 40/11)--(5,5)--(10-40/11,10-50/11)--(10-15/4,10-15/4)--(10,10)--(10,6)--(4,0)--(0,0);
	\draw [<->] (0,10.7)node[above]{$y$}--(0,0) node[below]{$(0,0)$}--(10.7,0) node[right]{$x$};
	\draw (0,10) --(10,10)--(10,0) node[below]{$(1,0)$};
	\draw (4,4)node[above]{$D$}--(6,8/3)node[above]{$B$}--(10-8/3,4)node[left]{$B'$}--(6,6)node[above]{$D'$};
	\draw (0,0)--(4,4);
	\draw [dash pattern={on 2pt off 1pt}] (4,4)--(6,6);
	\draw (6,6)--(10,10);
	\draw (4,0)node[below]{$\frac 2d$}--(10,6);
	\draw [dash pattern={on 2pt off 1pt}]  (0,5)node[left]{$\frac12$}--(5,5)node[above]{$H$}--(5,0)node[below]{$\frac12$}; 
	\draw [dash pattern={on 2pt off 1pt}] (6, 8/3)--(6, 2)node[below]{$A$}--(6,0)node[below]{$E$};
	\draw [dash pattern={on 2pt off 1pt}] (10-8/3, 4)--(8, 4)node[right]{$A'$}--(10,4)node[right]{$E'$};
	\draw [dash pattern={on 2pt off 1pt}] (0,25/7)node[left]{$\frac{d}{2(d+2)}$}--(5, 25/7);
	\draw [dash pattern={on 2pt off 1pt}] (50/11, 40/11)--(5,5)node[above]{$H$}--(10-40/11,10-50/11)--(10-15/4,10-15/4);
	\draw (5.6, 4.4) node{$\wt{\mathcal R}_2$};
	\draw (30/8, 2) node{$\wt{\mathcal R}_3$};
	\draw (8, 50/8) node{$\wt{\mathcal R}_3'$};
	\draw (6.9, 3.2) node{$\mathcal R_1$};
	\draw [fill, red] (5, 25/7) circle [radius=0.03];
	\draw [fill, red] (15/4,15/4) circle [radius=0.03];
	\draw [fill, red] (50/11, 40/11) circle [radius=0.03] node[below] {$P_\circ$};
	\draw [dash pattern={on 2pt off 1pt}, red] (15/4,15/4)node[left]{$P_*$}--(5, 25/7);
\end{tikzpicture}\caption{Theorem \ref{thm} when $d\ge3$.}\label{figthm}
\end{minipage}
\end{figure}
Our main result is that the estimate \eqref{con1} is true for most of cases of $p$, $q$.   For the statement of the result we introduce additional notations. Let  $p_\circ$, $q_\circ$ and $p_\ast$ be defined by 
\begin{align}
\label{pstar}
\frac1{p_*} := 
	\begin{cases}
		\frac{3(d-1)}{2(3d+1)}
			&\!\!\!\text{ if } d \text{ is odd } \\[4pt]
		 \frac{3d-2}{2(3d+2)}
			&\!\!\!\text{ if } d \text{ is even}
	\end{cases} ,   
	\qquad  \Big(\frac1{p_\circ}, \frac1{q_\circ}\Big) :=
	\begin{cases}
		\big( \frac{(d+5)(d-1)}{2(d^2+4d-1)},\, \frac{(d-1)(d+3)}{2(d^2+4d-1)} \big) 
			&\!\!\!\text{ if } d \text{ is odd } \\[4pt]
		\big( \frac{d^2+3d-6}{2(d^2+3d-2)}, \, \frac{(d-1)(d+2)}{2(d^2+3d-2)} \big)    
			&\!\!\!\text{ if } d \text{ is even }
	\end{cases}.
\end{align} 
The number $p_\ast$ is related to Theorem \ref{osc-est1} and the numbers $p_\circ$, $q_\circ$ are determined by \eqref{below_the_line} and  $\frac1{q}=\frac{d-1}{d+1}(1-\frac1{p})$. We also set  $P_*=P_*(d):=(1/p_*, 1/p_*)$  and $P_\circ :=P_\circ(d)=(1/p_\circ,1/{q_\circ})$.  See Figure \ref{figthm} and Section \ref{prelim} (Corollary \ref{tcsbil}).  
When $d\ge 2$ we define $\wt{\mathcal R}_2=\wt {\mathcal R}_2(d)$ and $\wt{\mathcal R}_3=\wt{\mathcal R}_3(d)$ by
\[	\wt{\mathcal R}_2:=[B,B',P_\circ ', H, P_\circ]\setminus\big( [P_\circ, H) \cup [P_\circ ', H)\cup [B, B'] \big), \qquad \wt{\mathcal R}_3:=\mathcal R_3 \setminus [D, P_\circ, P_*].	\]
If $d=2$, note that $P_\circ=P_*=D=(1/4,1/4)$ and $\wt{\mathcal R}_i = \mathcal R_i$, $i=2,3$. See Figure \ref{figthm2}.
\begin{thm}\label{thm}  Let $z\in \C\setminus [0,\infty)$. 
If $d=2$, Conjecture \ref{main_conj} is true.  If $d\ge3$,  the conjectured estimate \eqref{con1} is true whenever $(1/p,1/q)\in \wt{\mathcal R}_2 \cup \wt{\mathcal R}_3 \cup\wt{\mathcal R}_3'$. Furthermore, when $d\ge2$,  for $(1/p,1/q)\in\{B,B'\}$  the estimate $	\|(-\Delta-z)^{-1}f\|_{q,\infty} \le C |z|^{-1+\frac{d}{d+1}} \|f\|_{p,1}$  
holds, and for  $(\frac1p, \frac{d-1}{2d})\in (B',E']\cap \mathcal R_0$  the  estimate 
$\|(-\Delta-z)^{-1}f\|_{\frac{2d}{d-1},\infty} \le C |z|^{-1+\frac d2(\frac1p-\frac1q)} \|f\|_p$ holds. 
\end{thm}

It is also possible to obtain similar results regarding the Laplace--Beltrami operator on compact manifolds (\cite{KL}). To prove the sharp resolvent estimates \eqref{con1} we dyadically decompose the multipliers $(|\xi|^2-z)^{-1}$  by taking into account the region of $\xi$ where the multiplier gets singular as $\im z\to 0$. Such idea is now classical in the context of the Bochner--Riesz conjecture (e.g. \cite{C85, Lee1}).  It is important to obtain the optimal  $L^p$--$L^q$ bounds for each of the operators which are given by the dyadic decomposition. For the purpose  we use the Carleson--Sj\"olin reduction (\cite{CS, St-book}), and combine this with Theorem \ref{osc-est1} in Section \ref{sect2-1} (\cite{GHI}) and  bilinear estimate for the extension operator associated to the hypersurfaces of elliptic type (\cite{T}). For more details, see Section \ref{prelim} (Corollary \ref{tcsbil}).

\begin{rem}\label{rmk1}
As mentioned in the above, the restricted weak type $(p,q)$ estimates with $(1/p,1/q)=B, B'$ when $d\ge 3$ were shown in \cite{Gu}. In Section \ref{proof_of_main_thm} we provide a different proof of those restricted weak type estimates  for $d\ge 2$, together with the weak type $(p,q)$ estimates when $(1/p, 1/q)$ is in the half open line segment $[E', B')\cap \mathcal R_0$ (see Figure \ref{figthm2}, Figure \ref{figthm} and Remark \ref{diff_Gu}).  This upgrades the endpoint case of uniform Sobolev estimate in \cite{rxz} from the restricted weak type $(p,q)$ to the weak type $(p,q)$ for $(1/p,1/q)=A'=\big(\frac{d+3}{2d}, \frac{d-1}{2d}\big)$ when $d\ge 4$.  Also, for $p,q$ satisfying $(1/p,1/q)\in\mathcal R_1(d)$, the uniform resolvent estimate \eqref{unif} follows by duality and interpolation. (For $d=2$ an additional simple argument involving frequency localization and Young's inequality is necessary to cover the case $(1/p,1/q)\in\mathcal R_1(2)$.)
\end{rem}

\begin{rem}
When $d=1$ it is also possible and much simpler to obtain the sharp resolvent estimates. For $z\in\C\setminus[0,\infty)$ we write $(- d^2/dx^2-z )^{-1}f(x)=G_z*f(x)$, where $G_z(x)=\frac{i}{2\sqrt{z}}e^{i\sqrt{z}|x|}$ (see \cite[p. 203]{Tes}). Since the kernel is bounded and integrable, Young's inequality and  \eqref{sc} yield
\[	\|(-{d^2}/{dx^2}-z)^{-1}\|_{p\to q}\lesssim |z|^{-\frac12(\frac1p-\frac1q)} \dist(z,[0,\infty))^{-1+\frac1p-\frac1q}, \quad \forall z\in\C\setminus [0,\infty)	\]
for all $p,q$ such that $1\le p\le q\le \infty$. Following the argument in Section \ref{knapp} one can easily  check that the estimates are sharp.
\end{rem}

\subsection*{Resolvent estimates on compact Riemannian manifolds}
Let $(M,g)$ be a  $d$-dimensional compact Riemannian manifold without boundary. When $d \ge 3$ Dos Santos Ferreira, Kenig, and Salo proved in \cite{DKS} that for any fixed $\de>0$ the uniform estimate
\begin{equation}\label{resol_g}
\|(-\Delta_g-z)^{-1}f\|_{L^{\frac{2d}{d-2}}(M)} \le C\|f\|_{L^{\frac{2d}{d+2}}(M)}
\end{equation}
holds for all $z\in \Xi_\de :=\{z\in\C\setminus[0,\infty): \im\sqrt{z}\ge \de\}$.\footnote{Here we choose the branch of $\sqrt z$, $z\in\mathbb C\setminus[0,\infty)$, such that the imaginary part is positive. Note that $\Xi_\de=\{z\in\C\setminus[0,\infty): (\im z)^2\ge4\de^2(\re z+\de^2) \}$. In the complex plane this region excludes a  neighborhood of the origin and a parabolic region opening to the right.} 
Shortly afterwards, Bourgain, Shao, Sogge and Yao \cite{BSSY}  proved that if $M$ is Zoll, then the region $\Xi_\de$ cannot be significantly improved by showing that 
\begin{equation}\label{shp_region}
\lim_{\la\to+\infty} \sup_{\ta\in[1,\la]} \|(-\Delta_g-(\ta^2+ i\e (\ta)\ta))^{-1}\|_{L^{\frac{2d}{d+2}}(M) \to L^{\frac{2d}{d-2}}(M)} = +\infty
\end{equation}
whenever $\e(\ta)>0$ for all $\ta$, and $\e(\ta)\to 0$ as $\ta\to +\infty$. However, in some cases where the manifold has favorable geometry such as the flat torus or Riemannian manifolds with nonpositive sectional curvature, the range of $z$ for \eqref{resol_g} can be extended (see \cite{BSSY}).  Shao and Yao  \cite{SY} proved the off-diagonal $L^p(M)$--$L^q(M)$ estimate of \eqref{resol_g}  for $p,q$ satisfying $1/p-1/q = 2/d$, $p\le\frac{2(d+1)}{d+3}$ and $q\ge \frac{2(d+1)}{(d-1)}$, but it is not known whether this range of $p,q$  is optimal even for $p,q$ which satisfy $1/p-1/q = 2/d$.  In \cite{FS} Frank and Schimmer observed that  the argument in \cite{DKS} can be applied to establish $L^p(M)$--$L^{p'}(M)$ analogue of \eqref{resol_g} when $ \frac{2d}{d+2}<p< \frac{2(d+1)}{d+3}$ and $d\ge 2$. They also obtained the estimate
\[
\|(-\Delta_g-z)^{-1}f\|_{L^{\frac{2(d+1)}{d-1}}(M)} \le C|z|^{-\frac1{d+1}}\|f\|_{L^{\frac{2(d+1)}{d+3}}(M)}
\]
with $C$ independent of $z\in\Xi_\de$ by proving an off-diagonal restricted weak type bound for the parametrix constructed in \cite{DKS}.

\begin{figure}
\captionsetup{type=figure,font=footnotesize}
\centering
	\begin{subfigure}[b]{0.38\textwidth}
	\centering
	\includegraphics[width=\textwidth]{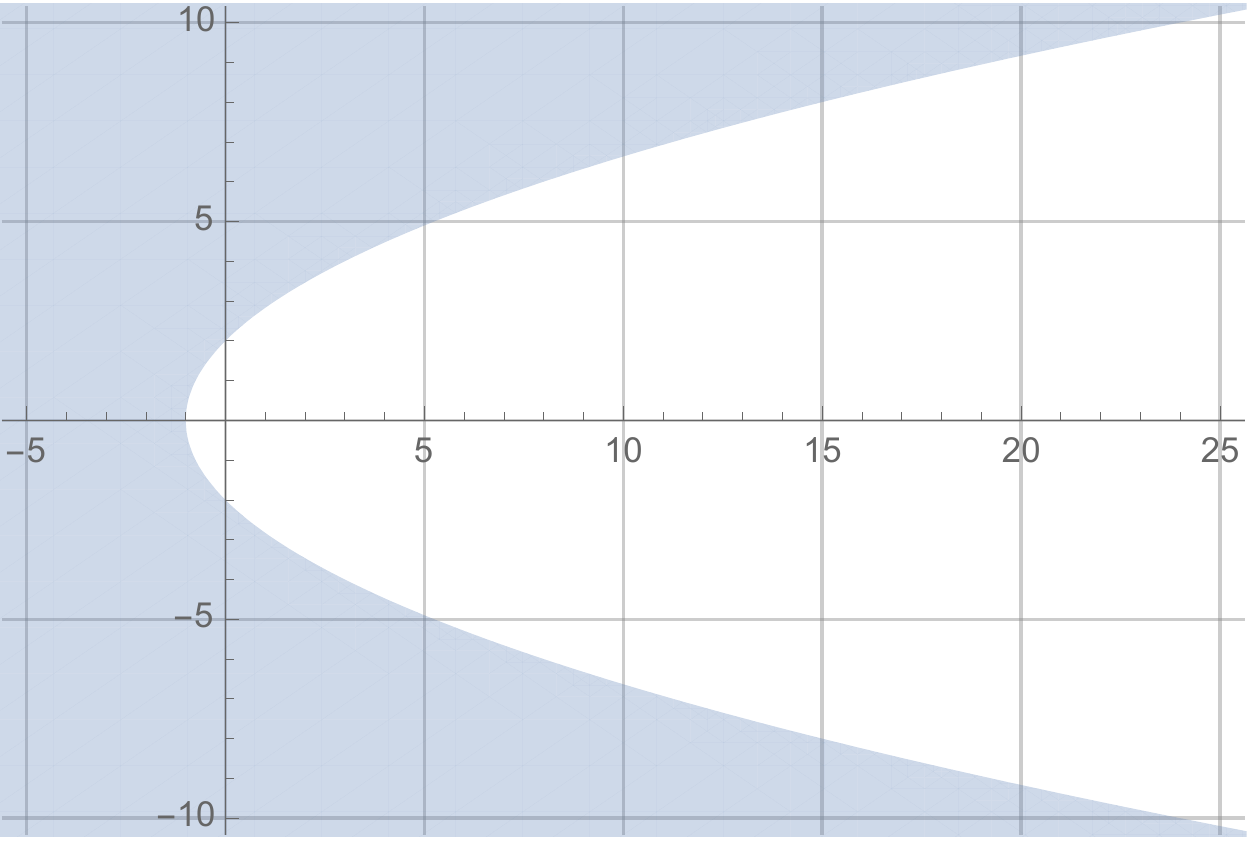}
	\caption{\footnotesize Compact manifold case: $\Xi_1$}
	\end{subfigure}
	\qquad \qquad
	\begin{subfigure}[b]{0.38\textwidth}
	\centering
	\includegraphics[width=\textwidth]{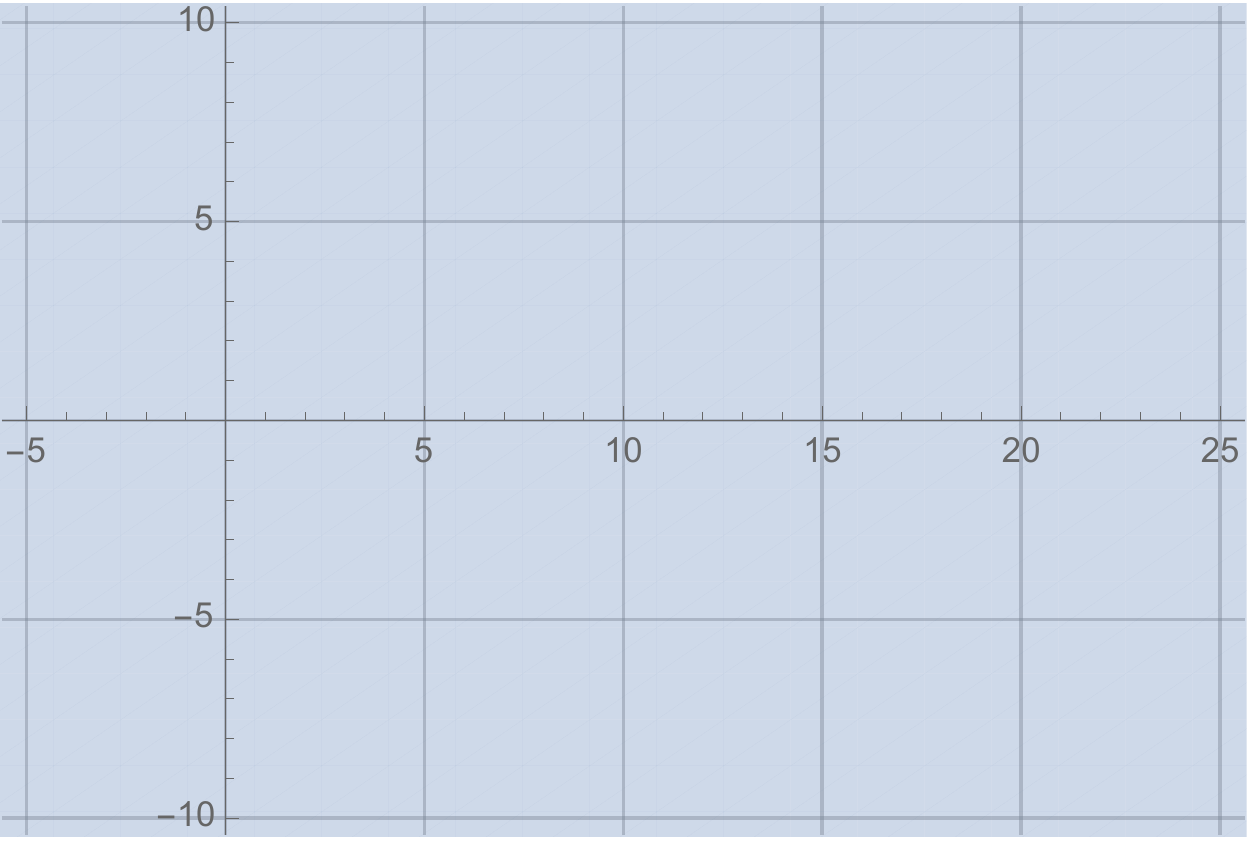}
	\caption{\footnotesize Euclidean case: $\C\setminus [0,\infty)$}
	\label{KRS_fig}
	\end{subfigure}
\caption{Spectral regions for which the uniform $L^{\frac{2d}{d+2}}$--$L^\frac{2d}{d-2}$ resolvent estimate holds.}
\label{dks_krs}
\end{figure}

\subsection*{Regions of spectral parameters where uniform resolvent estimate is allowed}
Since we now have sharp resolvent estimates which depend on the spectral parameter $z$,  it is possible, for each given $p,q$, to describe the region of $z$ for which the resolvent estimates are uniform. 

The $L^\frac{2d}{d+2}(\R^d)$--$L^\frac{2d}{d-2}(\R^d)$ bound for $(-\Delta-z)^{-1}$ is uniform in $z\in  \C\setminus[0, \infty)$ while the uniform estimate \eqref{resol_g} on compact manifold holds only for $z\in\Xi_\de$ (see Figure \ref{dks_krs}). Thus, we may reasonably  expect that the bound   $\|(-\Delta-z)^{-1}\|_{p\to q}$ behaves better on $\mathbb R^d$ than on compact manifolds. However, as is to be seen below, it is rather surprising that, for certain $p,q$,  the bound for $\|(-\Delta-z)^{-1}\|_{p\to q}$  has a similar  behavior with those on compact manifolds and the profile of the $z$-region where  $\|(-\Delta-z)^{-1}\|_{p\to q}$  is uniformly bounded  changes dramatically depending on the values of $p,q$.

For $p,q$ which satisfy  $(1/p,1/q)\in \mathcal R_1\cup \big( \bigcup_{i=2}^3\wt{\mathcal R}_i\big) \cup\wt{\mathcal R}_3'$, and $\ell>0$ we define the region $\mathcal Z_{p,q}(\ell)$ of spectral parameters by
\[	\mathcal Z_{p,q}(\ell):= \{z\in\C\setminus [0,\infty):  {\mathlarger \kappa}_{p,q}(z)\le \ell\}.	\] 
For simplicity,  let us focus on the case $\ell =1$, and  describe  roughly the typical shapes of $\mathcal Z_{p,q}(1)$. See Section \ref{drawing_figures} (and Figure \ref{fig_line_of_duality} and Figure \ref{fig_shapes_p2}) for detailed description of $\mathcal Z_{p,q}(\ell)$ in terms of $p,q,d,$ and $\ell$.

\vspace{-10pt}

\begin{itemize}
\item  If $d\ge3$ and $(1/p, 1/q)\in (A,A')$, then $\mathcal Z_{p,q}(1)=\C\setminus[0,\infty)$ (see Figure \ref{KRS_fig}).  
\item  If $ (1/p,1/q)\in\mathcal R_1$, and $1/p-1/q<2/d$,  then $\mathcal Z_{p,q}(1)$ is given by  removing the unit disk centered at zero from $\C\setminus [0,\infty)$ (see Figure \ref{disk_removed}).  
\item  If $(1/p, 1/q)\in \wt{\mathcal R}_2$, then $\mathcal Z_{p,q}(1)$ basically have  two different types.   
When $(p,q)\neq(2,2)$, $\mathcal Z_{p,q}(1)$ is the complex  plane  minus a neighborhood of $[0,\infty)$ which shrinks along the positive real line as $\re z\to \infty$ (see Figure \ref{fig_shrink}).  When $p=q=2$, $\mathcal Z_{2,2}(1)$ is the complex plane from which the $1$-neighborhood of $[0,\infty)$ is removed (see Figure \ref{1nbd_removed}).
\end{itemize}

\vspace{-10pt}

A remarkably interesting phenomenon occurs  when $(1/p,1/q)\in\wt{\mathcal R}_3\cup\wt{\mathcal R}_3'$. To describe this let us divide  $\wt{\mathcal R}_3$ into the three subsets $\wt{\mathcal R}_{3,+}$, $\wt{\mathcal R}_{3,0}$, and $\wt{\mathcal R}_{3,-}$, given by
\[	\wt{\mathcal R}_{3,\pm} :=\Big\{(x,y)\in \wt{\mathcal R}_3: \pm\Big(x+y-\frac{d-1}d\Big)>0 \Big\}, \quad
	\wt{\mathcal R}_{3,0} :=\Big\{(x,y)\in \wt{\mathcal R}_3:x+y-\frac{d-1}d=0 \Big\}.	\]

\vspace{-12pt}
	
\begin{itemize}
\item  If  $(1/p,1/q)\in \wt{\mathcal R}_{3,+}\cup\wt{\mathcal R}_{3,+}'$,    $\mathcal Z_{p,q}(1)$ is similar type as in the case $(1/p,1/q)\in\wt{\mathcal R}_2\setminus\{H\}$ (see Figure \ref{fig_shrink}). 
\item  If $(1/p,1/q)\in \wt{\mathcal R}_{3,0}\cup\wt{\mathcal R}_{3,0}'$,   we have  $\mathcal Z_{p,q}(1)=\mathcal Z_{2,2}(1)$ (see Figure \ref{1nbd_removed}). 
\item  Let $(1/p,1/q)\in\wt{\mathcal R}_{3,-}\cup \wt{\mathcal R}_{3,-}'$. If $1/p-1/q<2/d$,  $\mathcal Z_{p,q}(1)$ is the complement (in $\C$) of a neighborhood of $[0,\infty)$ whose boundary becomes wider as $\re z$ gets large  (see Figure \ref{fig_z6}). If $1/p-1/q=2/d$, then $\mathcal Z_{p,q}(1)=\{z\in\C\setminus \{0\} : \re z\le 0\}$. 
\end{itemize}
 
%%%%%%%%%%%%%%%%%%%%%%%%%%%%%%%%%%%%%%%%%%%%%%%%%%%%
\begin{figure}%[htp]
\captionsetup{type=figure,font=footnotesize}
\centering
	\begin{subfigure}[b]{0.38\textwidth}
		\centering
		\includegraphics[width=\textwidth]{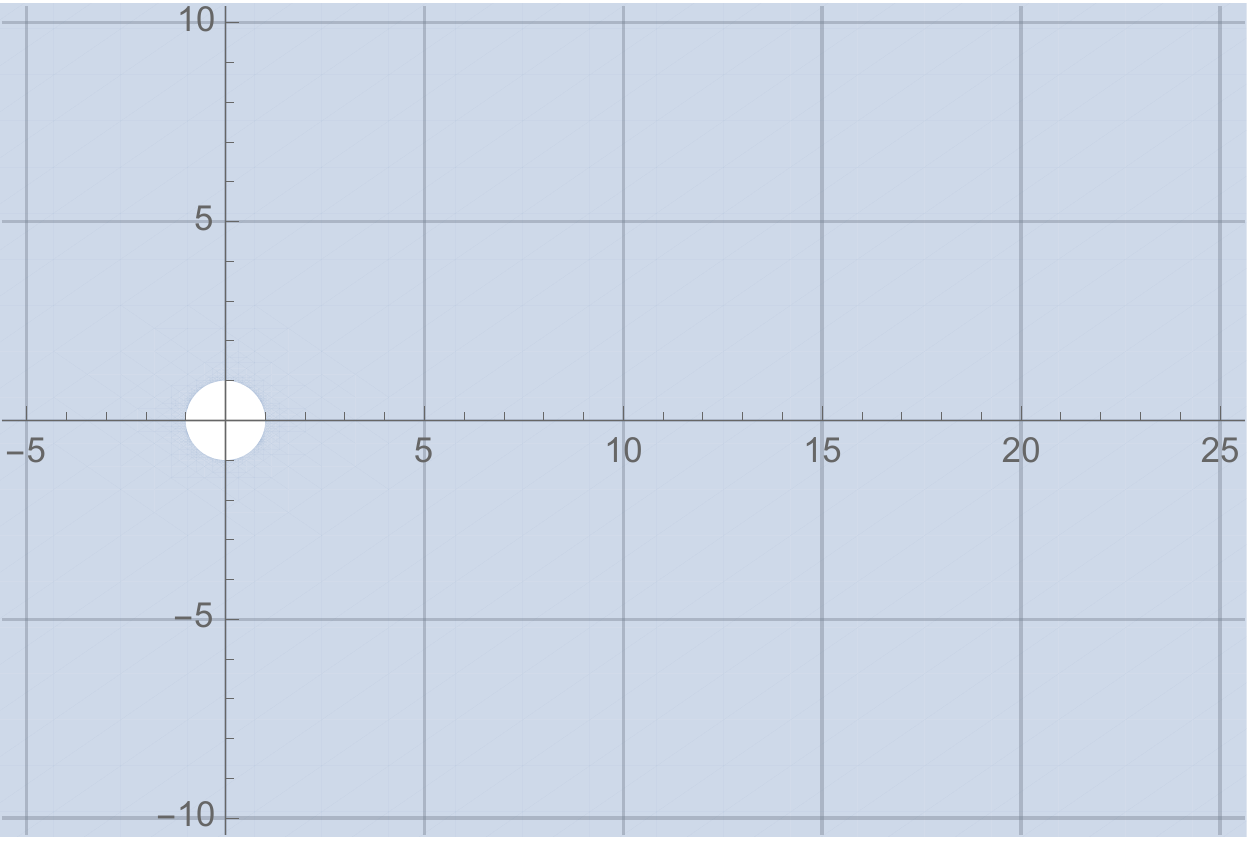}
		\caption{\scriptsize $(\frac1p,\frac1q)\in\mathcal R_1\setminus (A,A')$ }
		\label{disk_removed}
	\end{subfigure}
	\qquad\qquad
	\begin{subfigure}[b]{0.38\textwidth}
		\centering
		\includegraphics[width=\textwidth]{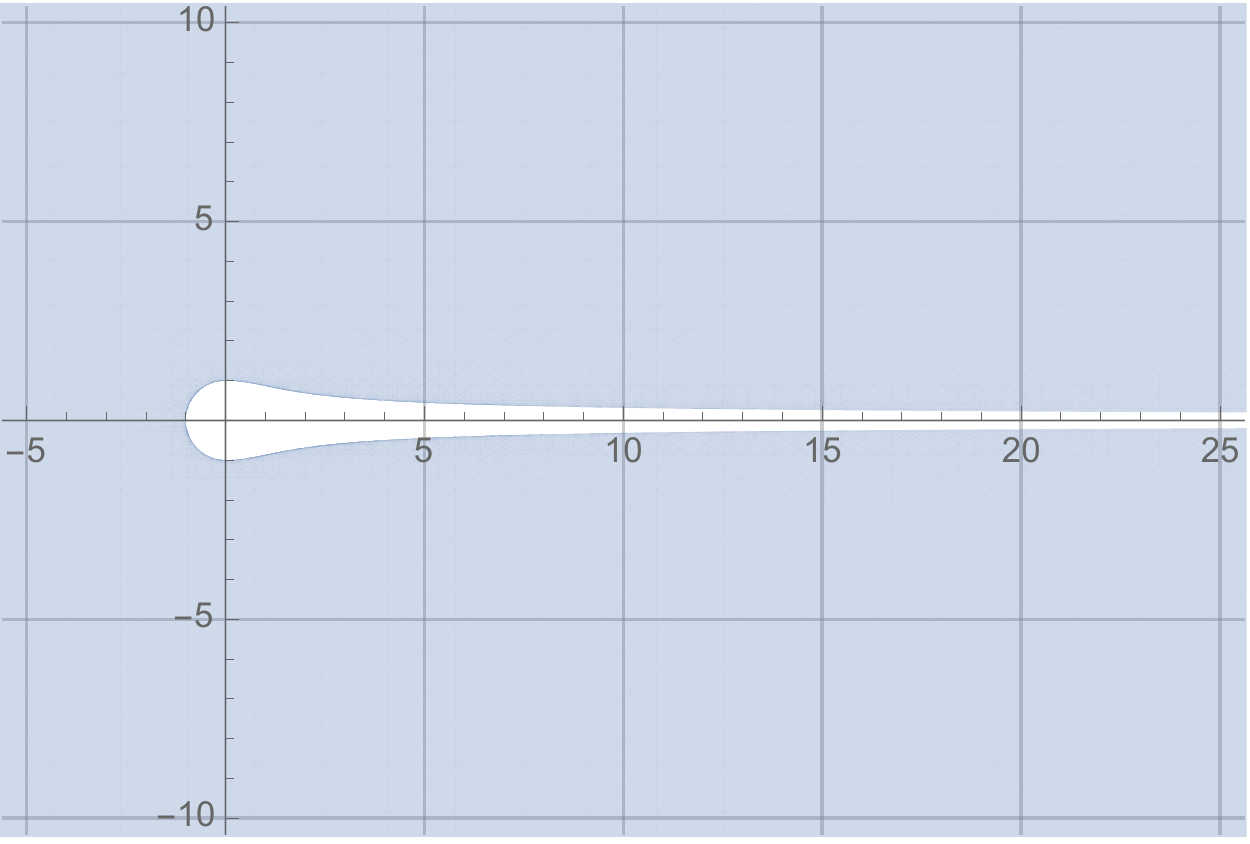}
		\caption{\scriptsize $(\frac1p,\frac1q)\in (\wt{\mathcal R}_2\setminus \{H\})\cup \wt{\mathcal R}_{3,+} \cup \wt{\mathcal R}_{3,+}'$}
		\label{fig_shrink}
	\end{subfigure}
	
	\bigskip
	
	\begin{subfigure}[b]{0.38\textwidth}
		\centering
		\includegraphics[width=\textwidth]{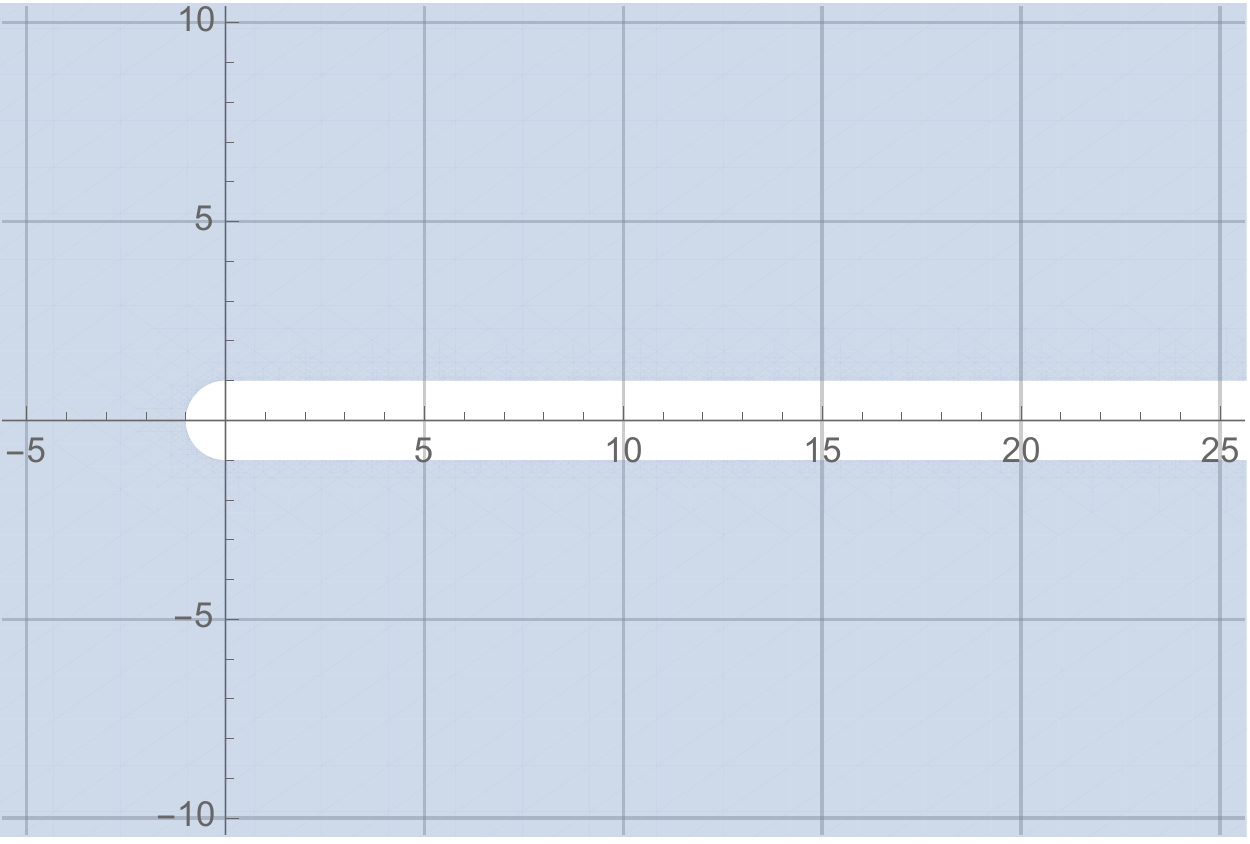}
		\caption{\scriptsize $(\frac1p,\frac1q)\in \{H\}\cup \wt{\mathcal R}_{3,0} \cup \wt{\mathcal R}_{3,0}'$}
		\label{1nbd_removed}
	\end{subfigure}
	\qquad\qquad
	\begin{subfigure}[b]{0.38\textwidth}
		\centering
		\includegraphics[width=\textwidth]{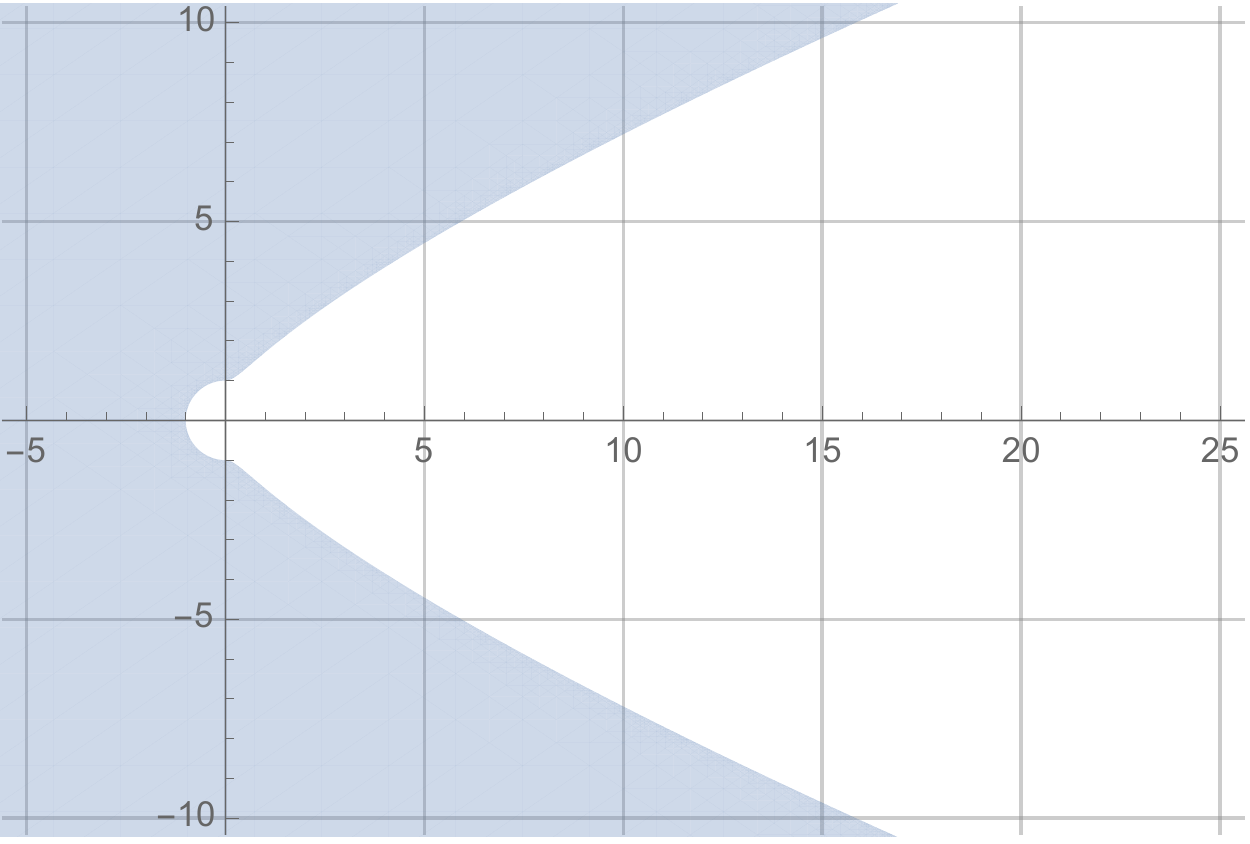}
		\caption{\scriptsize $(\frac1p,\frac1q)\in \wt{\mathcal R}_{3,-}  \cup \wt{\mathcal R}_{3,-}'$}
		\label{fig_z6}
	\end{subfigure}		
\caption{Some typical appearances of the spectral region $\mathcal Z_{p,q}(1)$ when $d\ge3$ and $\frac1p-\frac1q<\frac2d$ .}\label{fig_shapes}
\end{figure}

\subsection*{Location of the eigenvalues of $-\Delta+V$}
The sharp resolvent estimates (Theorem \ref{thm}) can be  used  to  specify  the location of eigenvalues of non-self-adjoint Schr\"odinger operators $-\Delta+V$ acting in $L^q(\R^d)$, $1\le q\le \infty$.  As was shown in \cite{Fr11, Fr18}, if $-\Delta+V$ acts in $L^2(\R^d)$ one can use the Birman--Schwinger principle, but this  is not the case when  $-\Delta+V$ acts in $L^q(\R^d)$, $q\neq2$. 
  
\begin{cor} \label{location_eigen}
Let $(1/p,1/q)\in\mathcal R_1\cup \big(\bigcup_{i=2}^3 \wt{\mathcal R}_i \big)\cup \wt{\mathcal R}_3'$ and let $C>0$ be the constant which appears in \eqref{con1}. Fix a positive number $\ell>0$ (we choose $\ell\ge1$ if $1/p-1/q=2/d$). Suppose that,  for some $t\in(0,1)$,
\begin{equation}\label{small_potential} 
	\|V\|_{L^\frac{pq}{q-p}(\R^d)} \le t (C\ell )^{-1}.
\end{equation}
If $E\in\C\setminus[0,\infty)$ is an eigenvalue of $-\Delta+V$ acting in $L^q(\R^d)$, then $E$ must lie in $\C\setminus \mathcal Z_{p,q}(\ell)$.
\end{cor}

This is rather a direct consequence of  Theorem \ref{thm}.  Let $u\in L^q(\R^d)$ be an eigenfunction of $-\Delta+V$ with eigenvalue $E\in\C\setminus[0,\infty)$. If $E$ were  contained in $\mathcal Z_{p,q}(\ell)$, Theorem \ref{thm} gives $\|(-\De-E)^{-1}\|_{p\to q}\le C{\mathlarger \kappa}_{p,q}(E)\le C\ell$.  By Minkowski's and H\"older's inequalities, and \eqref{small_potential} we have
\[	\|u\|_{q} \le C\ell \big(\|(-\De+V-E)u\|_{p} +\|Vu\|_{p}\big) \le C\ell \|V\|_{\frac{pq}{q-p}}\|u\|_q \le t\|u\|_q,	\]
which implies $u= 0$ since $t<1$. This is contradiction, hence $E$ must be in $\C\setminus\mathcal Z_{p,q}(\ell)$.

\begin{rem}
It is possible to formulate a statement which is analogous to the observation in \cite[p. 220, Remark (1)]{Fr18}. For example, if $(1/p,1/q)\in \wt{\mathcal R}_2\cup \wt{\mathcal R}_{3,+}\cup \wt{\mathcal R}_{3,+}'$, then for a sequence of eigenvalues $\{E_j\}$ of $-\Delta+V$ acting in $L^q(\R^d)$ such that $\re E_j\to\infty$ we have $\im E_j\to 0$ provided that $\|V\|_{\frac{pq}{q-p}}$ is small enough.  However, it does not seem to be likely that this phenomenon continues to be true for $p,q$ satisfying  $(1/p,1/q)\in \big(\wt{\mathcal R}_{3,0}\cup \wt{\mathcal R}_{3,- }\big)\cup \big(\wt{\mathcal R}_{3,0}\cup \wt{\mathcal R}_{3,- }\big)'$  and it would be interesting to ask whether there is a potential $V\in L^{\frac{pq}{q-p}}$ for which this kind of phenomenon fails.
\end{rem}
\begin{rem}
If $1/p-1/q=2/d$, and \eqref{small_potential} is satisfied with some $\ell\ge1$ and $t\in(0,1)$, then it follows from Corollary \ref{location_eigen} that the Schr\"odinger operator $-\De+V$ acting in $L^q(\R^d)$ does not have any eigenvalue of which real part is negative.
\end{rem}

\subsection*{Sharp resolvent estimate for the fractional Laplacian}
We also consider the sharp bound  on $\|((-\Delta)^{\frac s2}-z)^{-1}\|_{p\to q}$,  that is to say,  the $L^p$--$L^q$ resolvent estimate for the fractional Laplacian $(-\Delta)^{\frac s2}$ which is defined by
\[	(-\Delta)^{\frac s2}f(x)=\frac{1}{(2\pi)^d} \int_{\R^d} e^{ix\cdot\xi}|\xi|^{s} \wh f(\xi)d\xi.	\]
Uniform bounds on $\|((-\Delta)^{\frac s2}-z)^{-1}\|_{p\to q}$ for $p,q$ on certain range  were obtained in  Cuenin \cite{Cu15} and  these bounds were used to study eigenvalues of the fractional Schr\"odinger operators with complex potentials. Later, uniform bounds up to the optimal range of $p,q$ were obtained  by Huang, Yao, and Zheng \cite{HYZ}.  We also obtain the sharp bounds on $\|((-\Delta)^{\frac s2}-z)^{-1}\|_{p\to q}$ for $p,q$ which are not contained in the uniform boundedness range. See Theorem \ref{fractional}.  
 
Our method here is flexible and robust enough so that  it is  rather straightforward  to extend our argument from the Laplacian to  the fractional  Laplacian. This allows us to obtain the sharp bounds on $((-\Delta)^{\frac s2}-z)^{-1}$ for $s\in (0,d)$, which include the results for the resolvent of the Laplacian. Furthermore, Proposition \ref{shp}, Theorem \ref{thm}, and Corollary \ref{location_eigen}, can also be generalized in the context of  the fractional Laplacian $(-\Delta)^{\frac s2}$, $s>0$. There are also some new phenomena which do not  appear in the study of the resolvent of the Laplacian. For example, if $s$ is small,  the profile of the spectral parameter region where uniform bound is  allowed  never takes the form such as in Figure \ref{fig_shrink} (see Section \ref{fract} for details). However, we postpone discussions regrading  the resolvent of the fractional Laplacian until the last section to keep the presentation simpler. 
 
\subsubsection*{Organization of this paper} In Section \ref{prelim}, we review some properties of hypersurfaces of elliptic type, and the $L^p$--$L^q$ estimate for the Carleson--Sj\"olin type oscillatory integral operators. Then we obtain sharp estimates for the related multiplier operators of which frequency is localized.  In Section \ref{proof_bil},  based on the results obtained in Section \ref{prelim}, we establish Proposition \ref{dethick} which is the main ingredient for the proof of Theorem \ref{thm}.   In Section \ref{proof_of_main_thm} we prove Theorem \ref{thm}, and give descriptions in detail for various regions of spectral parameters $\mathcal Z_{p,q}(\ell)$ depending on $p$, $q$, $d$, and $\ell$. In Section \ref{shp1} the proof of  Proposition \ref{eresb} and Proposition \ref{shp} is given.  In Section \ref{fract} we obtain the sharp resolvent estimates for the fractional Laplacian $(-\De)^{\frac s2}$, $0<s< d$.

\subsubsection*{Notations}  For positive numbers $A$ and $B$, $A\lesssim B$ means that there is a constant $C$ such that $A\le CB$. We write $A\approx B$ if $A\lesssim B$ and $B\lesssim A$.  Both $x\cdot y$ and $\langle x, y \rangle$ denote the Euclidean inner product of $x$ and $y$. For a function $f$ on $\R^d$
\[
\F f(\xi)=\wh f(\xi) =\int_{\R^d} f(x)e^{-ix\cdot \xi}dx, \qquad  \F^{-1}f(\x)= f^\vee (\x)=(2\pi)^{-d}\wh f(-\x)
\]
denote the Fourier and inverse Fourier transforms, respectively. We set $D=-i({\p}/{\p x_1},\cdots , {\p}/{\p x_d})$. For a bounded measurable function $m$,  $m(D)f$ denotes the Fourier multiplier operator $(m\wh{f}\, )^\vee$. For $p,q\in [1,\infty]$ we define $\|m(D)\|_{p\to q}:=\sup_{\|f\|_p\le1}\|m(D)f\|_q$. For any pair of subsets $A$, $B$ of the Euclidean spaces or the complex plane, we write $\dist(A,B):=\inf\{|x-y|: x\in A,\, y\in B\}$. For any rectangle $Q$ and a positive number $a$, $a Q$ is the rectangle whose side length is $a$ times that of $Q$ with same center as $Q$. $B_{d}(c,r)$ is the open ball in $\R^{d}$ centered at $c$ with radius $r$. If $A$ is a set $\chi_A$ is the characteristic function of $A$.  We denote by $C_0^\infty(X)$ the class of smooth functions which are compactly supported in the set $X$. Throughout this paper, we fix an even function $\be\in C^\infty_0(\R)$ which is supported in the interval $[-9/8, -3/8]\cup[3/8, 9/8]$ and satisfies $\sum_{j=-\infty}^\infty \be(2^{-j}t)=1$ whenever $t\neq 0$. We also set $\be_0 =1-\sum_{j\ge 0}\be(2^{-j}\,\cdot\,)\in C^\infty_0((-3/4, 3/4))$. For a variable $x\in \R^d$ and a multi-index $\al\in\N_0^d$ we sometimes write $x=(x',x_d)\in\R^{d-1}\times\R$ and $\al=(\al',\al_d)\in\N_0^{d-1}\times\N_0$.

\subsubsection*{Acknowledgement} The authors were supported by NRF-2018R1A2B2006298. We  would like to  thank Ihyeok Seo and Younghun Hong for discussions on related issues, and the anonymous referees for various helpful comments.

\section{Estimates for localized frequency}\label{prelim}

In this section  we prove basic estimates which play important roles in obtaining  our main result. 

\subsection{Oscillatory integral operator of Carleson--Sj\"olin type} \label{sect2-1}
Let $\la\ge1$, $\mf a\in C^\infty_0(\R^d\times \R^{d-1})$, $\Phi\in C^\infty(\R^d\times\R^{d-1})$, and let $T_\la [\Phi, \mf a]$ be the operator defined by
\[
T_\la [\Phi, \mf a]f (x)= \int_{\R^{d-1}}e^{i\la\Phi(x,u)}\mf a(x,u) f(u) du, \quad (x,u)\in\R^d\times\R^{d-1}.
\]
Suppose that, for every $(x, u) \in \supp \mf a$, 
\begin{equation}\label{mixhess}
\rank \big(\p_x\p_u \Phi (x, u) \big) = d-1.
\end{equation}
We also assume that, for every $(x_\circ,u_\circ)\in \supp\mf a$, if $v\in\mathbb S^{d-1}$ is the (unique up to sign) direction such that the function $ u\to \langle v, \p_x\Phi (x_\circ, u)\rangle$ has a critical point at $u=u_\circ$, then
\begin{equation}\label{curv}
\rank \left( \p_u^2 \langle v, \p_x\Phi (x_\circ, u)\rangle \vert_{u=u_\circ}\right)=d-1.
\end{equation}
The operator $T_\la[\Phi, \mf a]$ with $\Phi$  satisfying \eqref{mixhess}, \eqref{curv} on $\supp\mf a$ is called the {\it Carleson--Sj\"olin type} oscillatory integral operator which originated from the work of Carleson and Sj\"olin \cite{CS} for the study of the two dimensional Bochner--Riesz problem (also, see \cite[pp. 60--70]{sogge_book}, \cite{MSS}). H\"ormander \cite{Ho72}  proved 
\begin{equation}\label{oscineq2}
\|T_\lambda  \phim  f\|_{L^q(\R^2)} \lesssim  \lambda^{-2/q} \|f\|_{L^p(\R)}
\end{equation}
if $4<q\le \infty$ and $3/q\le 1-1/p$, and the range of $p,q$ for \eqref{oscineq2} is optimal. The following higher dimensional extension is due to Stein \cite{St-beijing} (also, see \cite[Chapter 9]{St-book}).
\begin{thm}  \label{osc-est} Suppose $\Phi$ satisfies \eqref{mixhess} and \eqref{curv} on $\supp \mf a$. Then, for $1\le p, q\le \infty$ satisfying  $q\ge \frac{2(d+1)}{d-1}$ and $\frac{d+1}q\le (d-1) (1-\frac1p )$, the following  estimate holds:
\begin{equation}\label{oscineq}
\|T_\lambda  \phim f\|_{L^q(\R^d)} \lesssim  \lambda^{-d/q} \|f\|_{L^p(\R^{d-1})}.
\end{equation}
\end{thm} 
Bourgain  \cite{Bo-osc} showed that the estimate \eqref{oscineq} under the conditions \eqref{mixhess} and \eqref{curv}  generally fails if $q< \frac{2(d+1)}{d-1}$ when $d\ge3$ is odd. However, in \cite{Lee2} one of the authors observed that in addition to  \eqref{mixhess}, \eqref{curv}, if we assume that
\begin{equation}
\label{samesign}
\text{the surface $u\to \p_x \Phi(x,u)$ has $d-1$ nonzero principal curvatures of the same sign,}\footnote{This is equivalent to saying that the matrix $\p_u^2\langle v, \p_x\Phi \rangle$ is   either positive or negative definite.}
\end{equation} 
then the range of $p,q$ for which \eqref{oscineq} holds can be enlarged to $q>\frac{2(d+2)}{d}$. For  most recent developments  see  Bourgain and Guth \cite{BG} and  Guth, Hickman and Iliopoulou \cite{GHI}. These results are   based on multilinear estimates due to Bennett, Carbery and Tao \cite{BCT} and the method of polynomial partitioning due to Guth \cite{G1, G2}. We record here the recent sharp result due to Guth, Hickman and Iliopoulou \cite{GHI}.
\begin{thm} \label{osc-est1} 
Let $d\ge2$ and suppose $\Phi$ satisfies \eqref{mixhess}, \eqref{curv} and \eqref{samesign} on $\supp \mf a$. Then, the  estimate  \eqref{oscineq}  holds whenever $p=q>p_*$, for $p_*$ given in \eqref{pstar}.  This is sharp (up to endpoint) in the sense that there are examples of  Carleson--Sj\"olin type operators $T_\la\phim$ with  phase functions satisfying all of \eqref{mixhess}, \eqref{curv} and \eqref{samesign} for which the estimate \eqref{oscineq}  with $p=q$ fails whenever $p<p_*$.
\end{thm} 

\begin{rem} \label{stability} The estimate  \eqref{oscineq} with $p=q$   in  Theorem \ref{osc-est1}   is uniform under small smooth  perturbation of the phase $\Phi$ and the amplitude $\mf a$.  In fact,  the estimate  \eqref{oscineq}   in \cite{GHI} was obtained by running induction argument over a class of  operators while the  phase functions  are properly normalized. See \cite[Lemma 4.1, Definition 11.3]{GHI}. Since small smooth perturbation of the phase functions are  allowed within the class of operators, stability of the estimates follows. 
\end{rem}

\subsection{Functions of elliptic type}
Let $N\in\N$ and $\epsilon>0$.  Let us set  $\mI =[-1,1]$.  Following \cite{TVV} and \cite{LRS}, we define ${\bf Ell}(N,\epsilon)$ as the class of $C^N$-functions $\psi : \mI^{d-1}\to\R$ satisfying
\vspace{-10pt}  
\begin{itemize}
	\item $\psi(0)=0$ and $\nabla\psi(0)=0$;
	\item Let $w(\xi')=\psi(\xi')-|\xi'|^2/2 $. Then 
\begin{equation}\label{elliptic_error}
 \sup_{\xi'\in \mI^{d-1}} \max_{0\le |\al|\le N} |\p^\al w (\xi') | \le \epsilon .
\end{equation}
\end{itemize}
\vspace{-10pt}  
Typically $N$  is chosen to be large and $\epsilon$ to be small.  As was pointed out in \cite{TVV}, every convex smooth hypersurface  with nonvanishing Gaussian curvature can be locally parametrized as  graph of a function of elliptic type after a proper affine transformation. 

For later use, we record here an approximate property of functions of elliptic type, which is an easy consequence of Taylor's theorem. Let $H\psi$ denote the Hessian matrix of $\psi$.
\begin{lem}\label{approximation}
Let $N$, $\epsilon$ be as above and $0<\rho\le 2^{-1}$.   For $\psi\in{\bf Ell}(N,\epsilon)$ and $c\in ( 2^{-1} \mI)^{d-1}$ set 
\begin{equation}\label{para_scale}
\psi_{c,\rho}(\xi') = \rho^{-2}\big( \psi(\rho\xi'+c) -\psi(c)-\rho\nabla\psi(c)\cdot\xi'\big). 
\end{equation}
Then, we have 
\begin{equation}\label{para_scale_approx}
\sup_{(\xi',c)\in  \mI^{d-1}\times ( 2^{-1} \mI)^{d-1} }\Big | \partial_{\xi'}^\alpha \Big( \psi_{c,\rho}(\xi')- \frac12\langle H\psi(c)\xi',\xi'\rangle \Big) \Big| 
\le	\begin{cases}
	 \, \frac{(d-1)^{3-|\al|}}{(3-|\al|)!} \rho \epsilon &  \text{if} \quad 0\le |\alpha|\le 2, \\[4pt] 
	 \, \rho^{|\al|-2} \epsilon &  \text{if} \quad 3\le |\alpha|\le N .
	\end{cases}
\end{equation}
Moreover,  there is a constant $\mf c$, depending only on $d$,  such that if $\psi\in{\bf Ell}(N,\epsilon)$, then, for all $c\in( 2^{-1} \mI)^{d-1}$ and $0<\rho\le 2^{-1}$,  $\psi_{c,\rho} \in {\bf Ell}(N,  \mf c\epsilon)$.
\end{lem}
\begin{proof}
Clearly $\psi_{c,\rho}(0)=0$ and $\nabla\psi_{c,\rho}(0)=0$. Let $|\xi'|_1:=\sum_{k=1}^{d-1} |\xi_k|$.  By Taylor's theorem we have for $(\xi',c)\in  \mI^{d-1}\times ( 2^{-1} \mI)^{d-1} $ that
\[
\Big| \psi_{c,\rho}(\xi')-\frac12\langle H\psi(c)\xi',\xi'\rangle \Big| \le \frac{|\xi'|_1^3}{3!} \sup_{|\al|=3, \, 0\le t\le 1} \left| \rho(\p^{\al}\psi)(\rho t\xi'+c)\right| \le \frac{  |\xi'|_1^3 }{3!}\rho\epsilon.
\]
The second inequality follows from  \eqref{elliptic_error} since $\rho t\xi' +c\in \mI^{d-1}$ whenever $(\xi',c)\in  \mI^{d-1}\times ( 2^{-1} \mI)^{d-1}$, $0<\rho\le 2^{-1}$ and $0\le t\le 1$. Similarly, using Taylor's theorem and \eqref{elliptic_error} we also have
\begin{align*}
\Big | \partial_{\xi'}^\alpha \Big( \psi_{c,\rho}(\xi')- \frac12\langle H\psi(c)\xi',\xi'\rangle \Big) \Big|  \le
\begin{cases}
	\frac{\rho\epsilon}{2!} |\xi'|_1^2 &  \text{if} \quad  |\alpha|=1,\\[4pt] 
	\rho\epsilon |\xi'|_1 &  \text{if} \quad  |\alpha|=2.
\end{cases}
\end{align*}
If $|\alpha|\ge 3$, then $ \p_{\xi'}^\al\big(\psi_{c,\rho}(\xi')-\frac12\langle H\psi(c)\xi',\xi'\rangle \big) = \rho^{|\al|-2}(\p^\alpha\psi)(\rho\xi'+c)$. Hence we have \eqref{para_scale_approx}.  The second assertion follows immediately from \eqref{elliptic_error}, \eqref{para_scale_approx} and comparing $\langle H\psi(c)\xi',\xi'\rangle$ with $|\xi'|^2$.
\end{proof}

\subsection{Estimates for the operator with localized frequency}
To obtain the sharp  bound  \eqref{con1},  the  case in which  $|z|\approx 1$, $\re z>0$, and $|\im z|\ll 1$ is most important  (see  Section \ref{reduction} below).  In this case,  the corresponding Fourier multiplier  carries  most of its mass  near the sphere $S_z:=\{\xi\in\R^d : |\xi| = \sqrt{\re z}\}$, where $\sqrt{\re z}\approx1$.  Since $S_z$ is compact and convex with non-vanishing curvature, using finite decomposition and affine transformations, we can regard $S_z$  as   a finite union of graphs of functions of elliptic type.  Such operations do not have significant effect on the estimate \eqref{con1} except for a minor change of the multiplicative constant $C$. 

Now, by a dyadic decomposition (away from the graph of a function $\psi(\xi')$ of elliptic type) in the Fourier side,  we need to  obtain the sharp bounds for the multiplier operators of which Fourier transform is supported in a $\delta$-neighborhood of the surface $\xi_d=\psi(\xi')$.  

For $b>0$, let us set 
\[	\mnb=\{ m\in C^N(\mathbb R^d):      2^{-1}\le |m| \le 2, \   | \partial^\alpha m|\le b,  \  1\le |\alpha|\le { N} \}.	\]
Let $0<\de\le 1$, $\la\ge1$ and $\de\la\le 1/{10}$, and let $\varphi\in C^{\infty}(\R)$. For $\psi\in \elne$ and $m\in \mnb$  we  set  
\begin{align}
\label{mult_de}
	&\mm_{\de}(\xi) :=\varphi \Big(\frac{m(\xi)(\xi_d-\psi(\xi'))}\de \Big) \be_0 \Big(\frac{m(\xi)(\xi_d-\psi(\xi'))}{\de} \Big)   \chi_0(\xi), \\
\label{mult_la}
	&\mm_{\de,\la}(\xi) :=\varphi \Big(\frac{m(\xi)(\xi_d-\psi(\xi'))}\de \Big) \be \Big(\frac{m(\xi)(\xi_d-\psi(\xi'))}{\de\la} \Big) \chi_0(\xi),
\end{align}
where $\chi_0\in C^\infty_0 (2^{-1}\mI^{d})$. A particular example of $\varphi$  is the function $\varphi(t)=(2t\pm i)^{-1}$. In the proof of Theorem \ref{thm} (Section \ref{pf_of_main_prop}) $\psi$ parametrizes the surface given by a part of the sphere which is deformed under parabolic rescaling. By the additional  $m$ we may perturb the multipliers $\mf M_{\de}$ and $\mf M_{\de,\la}$, so this allows us to handle  other classes of operators which are given by multipliers with similar structure. 

The following  provide sharp estimates for $\mf M_{\de}(D)$ and $\mf M_{\de,\la}(D)$ and these are most  important ingredients in proving Theorem \ref{thm}. For its application, see Section \ref{pf_of_main_prop} (especially \eqref{break_dyadic}) and the proof of Theorem \ref{boc_rie_neg}.
\begin{prop}\label{dethick} Let $b>0$ and suppose that, for $ k\ge 0$, 
\begin{equation}\label{diff}
\Big| \Big(\frac{d}{dt} \Big)^k  \varphi(t) \Big | \le C_k t^{-k-1}, \quad |t|\ge 1.
\end{equation}
Then,  for $p,q$ satisfying  $\frac 1q=\frac{d-1}{d+1}(1-\frac 1p)$ and $q_\circ< q\le \frac{2(d+1)}{d-1}$,  there exist $N$ and $\epsilon>0$ such that the following hold uniformly provided that $\psi\in {\bf Ell}(N,\epsilon)$ and $m\in \mnb$: 
\begin{align}
\label{dela_0} 
	&\normo{\mm_{\de}(D)f}_{L^q(\R^d)}  \le C \de^{\frac{1-d}2+\frac dp} \|f\|_{L^p(\R^d)}, \\
\label{dela} 
	&\normo{\mm_{\de, \la}(D)f}_{L^q(\R^d)}  \le C \la^{-1} (\de\la)^{\frac{1-d}2+\frac dp} \|f\|_{L^p(\R^d)}.
\end{align}
Here the constant $C$ may depend on $b$, $d$, $p$, $q$, $N$, $\epsilon$, $\varphi$ and $\chi_0$, but is independent of $\de$, $\la$, $m$, $\psi$ and $f$. 
\end{prop} 

\begin{rem} 
Similar estimates  were obtained in \cite[Proposition 2.4]{CKLS}. However, there are differences  which need to be mentioned. Firstly, the function $\phi\in\mathcal S(\R)$ in \cite[Proposition 2.4]{CKLS} is assumed to have the special cancellation property $\supp\wh\phi =\{t\in\R: |t|\approx 1\}$ which  was crucial in obtaining the sharp estimate,  whereas we do not need such extra assumption in Proposition \ref{dethick}.  This is necessary for the proof of Theorem \ref{thm}. Unlike \cite{CKLS}  the associated multipliers $(|\xi|^2-z)^{-1}$ are not homogeneous, so we cannot decompose them in such a nice way as in  \cite[Lemma 2.1]{CKLS} (also, see \cite[Section 2.1]{JKL}).  Secondly, we allow smooth perturbation of $m$ in \eqref{mult_de} and \eqref{mult_la} with $m$ satisfying $|m|\approx 1$. Lastly, the estimates \eqref{dela_0} and \eqref{dela} hold on a wider  range of $p,q$ than that of the estimate in \cite[Proposition 2.4]{CKLS}. 
\end{rem}

We postpone the proof of Proposition \ref{dethick} until the next section. For the rest of this section we  present results which will be used for the proof of Proposition \ref{dethick}.

\subsection{$L^p$ boundedness of multiplier operators}\label{lp_bdd_mult}
In this section, we obtain sharp $L^p$ estimates for the multiplier operators $\mm_\de(D)$ and $\mm_{\de,\la}(D)$ which are consequences  of Theorem \ref{osc-est1}. We work with $\mm_{\de,\la}(D)$ only, since the same argument also works for $\mm_\de(D)$.  In what  follows  all of the $L^p$ estimates are uniform in $\psi\in {\bf Ell}(N,\epsilon)$, $m\in \mnb$, provided that $N$ is sufficiently large and $\epsilon$ is sufficiently small. 

\begin{prop}\label{cstomul}
Let $b$, $\chi_0$, $\de$, $\la$ and $\varphi$  be as in Proposition \ref{dethick}, and suppose that $p_*<p\le \infty$. Then there exist a large $N>0$, a small $\epsilon>0$ and a constant $C>0$ such that
\begin{align}
	\nonumber	 \|\mm_\de(D)f\|_{L^p(\R^d)}	&\le 	C \de^{\frac dp - \frac{d-1}2}\|f\|_{L^p(\R^d)},\\
	\label{CStoMUL} 	\|\mm_{\de,\la}(D)f\|_{L^p(\R^d)}	&\le 	C \la^{-1} (\de\la)^{\frac dp - \frac{d-1}2}\|f\|_{L^p(\R^d)},
\end{align}
where the constants $C$ are independent of $\de$, $\la$,  $\psi\in{\bf Ell}(N,\epsilon)$ and $m\in \mnb$. 
\end{prop}

We will achieve this by making use of  Theorem \ref{osc-est1}.  For this purpose we need to compute  the kernel $K_{\de,\la}$ of the operator $\mm_{\de, \la}(D)$. 

\begin{rem} \label{simple_cutoff}
To begin with, we readjust the cutoff functions $\chi_0$  in \eqref{mult_la} of which role is not so significant for the overall estimates.  We may regard $\chi_0 \widehat f $  as if it is $\widehat f$ (note that $\|\mathcal F^{-1}(\chi_0 \widehat f\,)\|_p\lesssim \|f\|_p$ if $\chi_0\in C_0^\infty$).  We may also introduce a new cutoff function $\chi'$ whenever $\chi_0 \chi'=\chi_0$ and replace $\widehat f$ with $\chi' \widehat f$.   By decomposing (with a suitable partition of unity) $\chi_0$ into finitely many cutoff functions with smaller support (of diameter $\lesssim \epsilon_0$)  we may  assume $\chi_0$ is supported in a small neighborhood near the surface $\xi_d=\psi(\xi')$. Otherwise, the contribution is negligible.  In fact, the associated kernel has a bounded $L^1$-norm as can be seen easily by a straightforward kernel estimate.  Let $\xi_0=(c, \psi(c))$ and suppose $\chi_0$ is supported in  $B_d(\xi_0, \epsilon_0)$ for a fixed $\epsilon_0$. Then, for $0<\rho\le 2^{-1}$, we may use the harmless affine transform   
\begin{equation}\label{affine}
	\xi\to L_{c,\rho}(\xi)=\big(\rho \xi'+c, \,  \rho^2\xi_d+ \psi(c)+ \rho\nabla\psi(c)\cdot \xi'\big)
\end{equation}
to  write  
\begin{equation}\label{multi-mod}
	\mm_{\de, \la}(L_{c,\rho}(\xi))= \varphi \Big(\frac{m( L_{c,\rho} \xi)(\xi_d-\psi_{c,\rho}(\xi'))}{\rho^{-2}\de }\Big) \be \Big(\frac{m( L_{c,\rho} \xi)(\xi_d-\psi_{c,\rho}(\xi'))}{\rho^{-2}\de \la} \Big) \chi_0(L_{c,\rho}\xi).
\end{equation}
By Lemma \ref{approximation} $\psi_{c,\rho}\in {\bf Ell}(N,\mf c\epsilon)$. Also,  $m\circ L_{c,\rho} \in {{\bf Mul}(N, Cb)}$ for some $C>0$.  Thus we may regard this as the same multiplier given by \eqref{mult_la} by simply replacing $\rho^{-2}\de$, $m\circ L_{c,\rho}$ and $\psi_{c,\rho}$ with $\de$, $m$ and $\psi$, respectively.  Hence taking $\epsilon_0$ small enough and  $\rho=\epsilon_0 2^{7}$, after a simple manipulation (discarding the part of multiplier which is away from the surface) we may assume  the cutoff function takes the form 
\[ 	\chi_0(\xi)=\chi(\xi') \chi_1 (\xi_d-\psi(\xi'))	\] 
and $\chi\in C_0^\infty(B_{d-1}(0, 2^{-7}))$  and $\chi_1\in C^\infty_0 ( (-2^{-7}, 2^{-7}))$.
\end{rem}

Before we prove \eqref{CStoMUL} rigorously, we explain the idea behind our proof with a simplified multiplier. This will help the reader to understand the detailed (but technical) proof below in this section. Let us assume that $m=1$, $\varphi(t)=1/t$, $\psi(\xi')=\frac12 |\xi'|^2$ and $\chi_0(\xi)=\chi(\xi')\chi_1(\xi_d-\psi(\xi'))$. Then the multiplier takes the simpler form
\[\mathfrak M_{\de,\la}(\xi) = \frac1\la b\Big(\frac{\xi_d-\psi(\xi')}{\de\la} \Big) \chi(\xi')\chi_1(\xi_d-\psi(\xi')),\]
where $b(t)=|t|^{-1}\beta(t)$. By the stationary phase method $K_{\de,\la}(x):=\F^{-1}(\mathfrak M_{\de, \la})(x)$ ``approximately" equals
\[ \lambda^{-1}|x_d|^{-\frac{d-1}2} e^{-i\frac{|x'|^2}{2x_d}} \int e^{ix_d\xi_d} b(\xi_d/\de\la)\chi_1(\xi_d) d\xi_d \]
for $|x_d|\gtrsim 1$ and $|x'|\lesssim |x_d|$. We dyadically decompose this kernel along $x_d$:
\[\sum_{l=1}^\infty\lambda^{-1}\beta(2^{-l}x_d)|x_d|^{-\frac{d-1}2} e^{-i\frac{|x'|^2}{2x_d}} \int e^{ix_d\xi_d} b(\xi_d/\de\la)\chi_1(\xi_d) d\xi_d=:\sum_{l=1}^\infty K_{\de,\la,l}(x). \]
Now the matter reduces to obtaining
\begin{equation}\label{psuedo}
\|2^{dl}K_{\de,\la,l}(2^l\,\cdot\,)*f\|_p\lesssim 2^{(\frac{d+1}2-\frac dp)l}\de(1+2^l\de\la)^{-M}\|f\|_p,
\end{equation}
because summation over $l$ gives the desired bound. Note that 
\[2^{dl}K_{\de,\la,l}(2^l\,\cdot\,)*f(x)= \lambda^{-1}2^{\frac{(d+1)l}2}\int e^{-i2^l\frac{|x'-y'|^2}{2|x_d-y_d|}} a_{\de,\la, l}(x_d-y_d) f(y)dy,\]
where $a_{\de,\la, l}(x_d-y_d)=\beta(x_d-y_d)|x_d-y_d|^{-\frac{d-1}2}\int e^{i2^l(x_d-y_d)\xi_d}b(\xi_d/\de\la)\chi_1(\xi_d) d\xi_d$.  Application of the oscillatory integral estimate (Theorem \ref{osc-est1}) gives \eqref{psuedo}. Of course this is an oversimplification. We provide detailed argument in what follows.

By Remark \ref{simple_cutoff} and change of variables $(\xi',\xi_d)\to(\xi',\xi_d+\psi(\xi'))$ we may write
\begin{equation} \label{kernel}
K_{\de,\la}(x):= \mathcal F^{-1}(\mm_{\de, \la})(x)
	= \frac1{2\pi}	\int e^{ix_d\xi_d}  \chi_1(\xi_d)	I_\psi (x;\xi_d)  \, d\xi_d,
\end{equation}
where 
\[ I_\psi (x;\xi_d) =   \frac1{(2\pi)^{d-1}}\int e^{i(x'\cdot\xi'+x_d\psi(\xi'))}     \varphi \Big(\frac{\wt m(\xi)\xi_d}\de \Big) \be \Big(\frac{\wt m(\xi)\xi_d}{\de\la} \Big) \chi(\xi') d\xi' \]
and $\wt m (\xi)=m(\xi',\xi_d+\psi(\xi'))$. $\wt m$ still enjoys the same property as $m$ in Proposition \ref{dethick}, that is to say, $\wt m\in {\bf Mul}(N, Cb)$ for some $C>0$.  For simplicity  we put
\begin{equation} \label{Adl}
 A_{\de, \la} (\xi) :=	\varphi \Big(\frac{\wt m(\xi)\xi_d}\de \Big) \be \Big(\frac{\wt m(\xi)\xi_d}{\de\la} \Big) \chi(\xi').
\end{equation}

Let us collect some bounds for the functions $A_{\de,\la}$ and their differentials which will be useful later when we show Proposition \ref{cstomul}.

\begin{lem}\label{diff_uniform}
Let $0<\de\le\de\la\le 1$, $b>0$ and let $\psi\in{\bf Ell}(N,\epsilon)$, $m\in \mnb$.  Then, for every $(d-1)$-dimensional multi-index $\vartheta \in\N_0^{d-1}$ with $|\vartheta|\le N$, we have 
\begin{equation}\label{diff_uniform1}
	\sup_{\xi} \big| \p_{\xi'}^\vartheta A_{\de, \la}(\xi) \big| \le C_\vartheta \la^{-1}
\end{equation}
uniformly in $\de$, $\la$, $\psi\in \Ell$ and $m\in \mnb$.  More generally, for every $d$-dimensional multi-index $\al=(\al',\al_d)\in \N_0^{d-1}\times \N$ and every $\vartheta\in \N_0^{d-1}$  such that $|\al|+|\vartheta|\le N$, we have
\begin{equation}\label{diff_uniform2}
	\sup_{\xi} \big| \p_{\xi'}^{\al'} \p_{\xi_d}^{\al_d} \big( (\xi_d)^\ell \p_{\xi'}^\vartheta A_{\de,\la}(\xi)\big)\big| \le C_{\al,\vartheta} \la^{-1} (\de\la)^{-\al_d+\ell}
\end{equation}
with $C_{\al,\vartheta}$ independent of $\de$, $\la$, $m$ and $\psi$. 
\end{lem}
\begin{proof}
For every $k\in\N_0$ note that $\be^{(k)} \big( \frac{\wt m(\xi) \xi_d}{\de\la} \big) \neq 0$ only if $|\xi_d|\approx\de\la$ since $|\wt m|\approx 1$. Hence for every $\vartheta\in\N_0^{d-1}$ with $0\le |\vartheta|\le N$  it is easy to see that $\supp \p_{\xi'}^\vartheta A_{\de,\la}$ is contained in  the set $\{\xi : |\xi_d| \approx \de\la \}$ and that 
\begin{equation}\label{be_diff_1}
	\sup_{\xi} \Big| \p_{\xi'}^\vartheta \Big( \be \Big(\frac{\wt m(\xi)\xi_d}{\de\la} \Big) \Big) \Big| \le C_\vartheta
\end{equation}
with $C_\vartheta$ independent of $\de$, $\la$, $m$ and $\psi$.  Also, for  $0\le |\vartheta|\le N$,  
\[	\sup_{\xi} \Big| \p_{\xi'}^\vartheta \Big( \varphi \Big( \frac{\wt m (\xi)\xi_d}{\de}\Big) \Big) \Big| 
	\lesssim \sum_{k=0}^{|\vartheta|} \sup_{\xi } \Big| \varphi^{(k)} \Big(\frac{\wt m(\xi)\xi_d}{\de}\Big) \Big(\frac{\xi_d}{\de} \Big)^k \Big|,	\]
where the implicit constant is independent of $\de$, $\la$, $m$ and $\psi$.  Since $|\wt m| \approx 1$ and $|\xi_d|\approx\la\de$ on $\supp A_{\de,\la}$, by \eqref{diff}  we see that 
\begin{equation}\label{phi_diff_1}
\sup_{\xi'\in\supp\chi, \, |\xi_d|\approx\de\la} \Big| \p_{\xi'}^\vartheta \Big( \varphi \Big( \frac{\wt m (\xi)\xi_d}{\de}\Big) \Big) \Big| 
	\le C_\vartheta \la^{-1} .
\end{equation}
By combining \eqref{be_diff_1} and \eqref{phi_diff_1}  it is easy to see  \eqref{diff_uniform1}. 

For the proof of \eqref{diff_uniform2} we first consider the case $\alpha'=\vartheta=0$. Note that  $\p_{\xi_d}^{\al_d} \big((\xi_d)^\ell  A_{\de,\la} \big)$ is given by a linear combination of  $(\xi_d)^{\ell-n}  \p_{\xi_d}^{\alpha_d-n}  A_{\de,\la}$  and $\p_{\xi_d}^{\alpha_d-n}  A_{\de,\la}$ is also a linear combination of 
\[  \delta^{-\nu}(\lambda\delta)^{-\mu} \chi_{\mu,\nu} (\xi) \varphi^{(\nu)}\Big(\frac{\wt m(\xi)\xi_d}{\de} \Big)  \beta^{(\mu)}\Big(\frac{\wt m(\xi)\xi_d}{\de\la} \Big),  \quad  \mu+\nu\le  \alpha_d-n.	\] 
Here $\chi_{\mu,\nu}$ is a smooth function with bounded derivatives.  Since $|\xi_d|\approx\la\de$ on $\supp A_{\de,\la}$, we deduce the desired bound \eqref{diff_uniform2} by \eqref{diff}.  For the general cases one can routinely repeat the same argument keeping in mind that $\p_{\xi'}^{\al'}$ or $\p_{\xi'}^{\vartheta}$ behaves almost similarly as in \eqref{diff_uniform1} on $\supp A_{\de,\la}$. So, we omit the detail. 
\end{proof}

We now obtain the asymptotic for the function $I_\psi(\,\cdot \, , \xi_d)$. Since $\psi\in \Ell$, $\nabla\psi(0)=0$ and $\nabla \psi (\xi')= \xi'+O(\epsilon)$ for $\xi'\in\mI^{d-1}$. Thus,  by the inverse function theorem we see that  there exist neighborhoods $U$, $V$ of the origin and a unique diffeomorphism $g: U \to V$ such that $g(0)=0$ and 
\begin{equation}\label{gauss_map}
	t'+\nabla\psi(g(t'))=0.
\end{equation}
If we take $\epsilon$ sufficiently small, we may assume that $U\supset  B_{d-1}(0, 1/2)$.  In fact, $(g(t'), \psi \circ g(t'))$ is the unique point on the graph $\mathcal G (\psi):=\{(\xi',\psi(\xi')): \xi'\in\supp\chi\}$ at which the normal vector is parallel to $(t',1)$. We denote by ${\bf K}(\xi)$ the Gaussian curvature of the surface $\mathcal G (\psi)$ at point $\xi=(\xi',\psi(\xi'))$ and by ${\bf J}g$ the Jacobian matrix of the diffeomorphism $g$.  Direct differentiation of the equation \eqref{gauss_map} gives  
\begin{equation}\label{jacobian}
	((H\psi)\circ g) \cdot {\bf J}g = -I_{d-1}.
\end{equation}

\begin{lem}\label{stnry} 
Let $0< \de \le \de\la \lesssim 1$. Suppose that $N$ (resp., $\epsilon$) is large (resp., small) enough so that for every $\psi\in{\bf Ell}(N, \epsilon)$, the aforementioned diffeomorphism $g: U\supset B_{d-1}(0,1/2 )\to V$ exists.  Then the following hold.  

\vspace{-8pt}  
$(\mathrm I)$ If $|x_d|\ge 1/2$ and $2^5|x'|\le|x_d|$,  then for every $M\in \N$ satisfying $2M\le N$  we have 
\begin{equation}\label{stationary} 
I_\psi (x;\tau)
	=\frac{c_d}{\sqrt{|{\bf K}|}} e^{i ( x' \cdot g (\frac{ x'}{x_d} ) +x_d\psi\circ g  ( \frac{x'}{x_d}) )} \sum_{j=0}^{M-1} \mathcal{D}_{j}A_{\de,\la}(\xi',\ta)\Big|_{\xi'=g(\frac{x'}{x_d})} |x_d|^{-\frac{d-1}{2}-j}  +\mathcal {E}_{\de, \la, M}(x; \tau),
\end{equation}
where $c_d$ is a constant depending only on $d$,  $|{\bf K}|=\big|{\bf K}\big(g ( \frac{x'}{x_d} ), \, \psi\circ g(\frac{x'}{x_d} )\big)\big| \approx 1$,  $\mathcal{D}_0 A_{\de,\la}= A_{\de, \la}$ and, for each $j\ge1$, $\mathcal{D}_{j}$ is a differential operator in $\xi'$ of order $2j$ whose coefficients vary smoothly depending on  $\big(\p_{\xi'}^{\al}\psi\big)\circ g\big( \frac{x'}{x_d}\big),$  $2\le |\al| \le 2j+2$.  For  $\mathcal {E}_{\de, \la, M}(x;\tau)$ we have the  estimate
\begin{equation}\label{asymp_error}
	\big| \mathcal {E}_{\de, \la, M}(x; \tau) \big| \le C |x_d|^{-M} \sum_{|\al|\le 2M} \sup_{(\xi',\tau)} \big | \p_{\xi'}^\al A_{\de, \la} (\xi', \tau) \big| \le C' |x_d|^{-M} \la^{-1}
\end{equation}
with $C'$ independent of $\de$, $\la$, $m$ and $\psi$. 

\vspace{-8pt}
$(\mathrm I\!\mathrm  I)$ On the other hand,   if $2^6|x'|\ge|x_d|$ or $|x_d|\le 2$, then for every $0\le M\le N$ there exists a constant $C_{M}$, independent of $\de$,  $\la$, $m$ and  $\psi$, such that
\begin{equation}\label{nonstationary}
	|I_{\psi}(x; \tau)|\le C_{M} \la^{-1} (1+|x|)^{-M}.
\end{equation}
\end{lem}
\begin{proof}
The asymptotic expansion \eqref{stationary} in $(\mathrm I)$  is a consequence of the stationary phase method. For its  proof we refer the reader to \cite[Theorem 7.7.5 and Theorem 7.7.6]{Ho}.  In \eqref{asymp_error} the uniformity of $C'$ in $\de$, $\la$, $m$ and $\psi$ follows from Lemma \ref{diff_uniform}.

For the second statement  $(\mathrm I\!\mathrm  I)$ we use integration by parts.  Since $\supp\chi\subset B_{d-1}(0, 2^{-7})$ and $|\nabla\psi(\xi')| \le (1+c\epsilon) |\xi'|$, it is easy to observe that, if  $\epsilon$ is sufficiently small and $ 2^6 |x'|\ge|x_d|$, 
\[	|\nabla_{\xi'} (x' \cdot \xi' + x_d \psi(\xi') )| \ge |x'|(1- 2^{6}|\nabla\psi(\xi')|) \gtrsim |x'|.	\]
If $|x_d|\le 2$ the same estimate also holds with $|x'|\ge 1$.  Hence \eqref{nonstationary}  follows from integration by parts in $\xi'$ together with \eqref{diff_uniform1} in Lemma \ref{diff_uniform}. 
\end{proof}

Now we prove \eqref{CStoMUL} by combining Theorem \ref{osc-est1} and Lemma \ref{stnry}.

\begin{proof}[Proof of \eqref{CStoMUL}]
Let $\wt{\chi}$  be a smooth function on $\R$ supported in the interval $(-2^{-5},2^{-5})$ and equal to $1$ on $(-2^{-6},2^{-6})$.  We break the kernel $K_{\de,\la}$ as follows: 
\[	K_{\de,\la} (x)=K_{\de, \la, 0}(x) +\sum_{l=1}^\infty K_{\de,\la, l}(x),	\]
where 
\[	K_{\de,\la,l}(x)= \wt\chi \Big(\frac{|x'|}{x_d} \Big)\be(2^{-l} x_d ) K_{\de, \la}(x) 	\]
for $l\in\N $.  So, the  function $K_{\de,\la, 0}$ is supported on the set $R := \{x: 2^6 |x'|\ge|x_d|\}\cup \{x:|x_d|\le 2\}$, and it follows from \eqref{nonstationary} ($(\mathrm I\!\mathrm  I)$ in Lemma \ref{stnry}) that
\[	|K_{\de,\la,0}(x)| \le C_M\la^{-1}(1+|x|)^{-M} 	\]
for any $0\le M\le N$, uniformly in $\de$, $\la$.  Since $\|K_{\de,\la, 0}\|_1 \lesssim \lambda^{-1}$, the  operator $f\to K_{\de,\la,0}*f$ admits much better estimate than \eqref{CStoMUL} since $p>p_*\ge\frac{2d}{d-1}$ and $\de\la\lesssim 1$.  Therefore it suffices to prove that
\begin{equation}\label{cstomul1}
	\sum_{l=1}^\infty \| K_{\de,\la,l}*f \|_p\le C \la^{-1} (\de\la)^{\frac dp - \frac{d-1}2}\|f\|_p.
\end{equation}
To show this we need the asymptotic \eqref{stationary} for $I_\psi(\, \cdot\, ; \xi_d)$, which is to be combined with \eqref{kernel}. Fixing $M\le 2N$ large enough, it is enough to handle the finite summation in \eqref{stationary}  since  the contribution from  the error term  $\mathcal {E}_{\de, \la, M}$ in \eqref{asymp_error} is at most $\lambda^{-1}$. In fact, since $(K_{\de, \la}-K_{\de,\la,0})(x)\neq 0$ only if $|x'|\lesssim |x_d|$, if we set $K_{err}(x)=\frac1{2\pi} \int e^{ix_d\xi_d}  \chi_1(\xi_d) \mathcal {E}_{\de, \la, M}(x,\xi_d)  d\xi_d$, it follows from \eqref{asymp_error} that  $\| K_{err}\|_1\lesssim \lambda^{-1}$. Thus, the contribution  from the first term in \eqref{stationary} is most significant  and it suffices to prove \eqref{cstomul1} by replacing $K_{\de,\la,l}$ with 
\begin{align} \label{wkernel}
\wt K_{\de,\la,l}(x) 
	&= \wt\chi \Big(\frac{|x'|}{x_d}\Big)\be(2^{-l} x_d )  |x_d|^{-\frac{d-1}2}  e^{i ( x' \cdot g (\frac{ x'}{x_d} ) +x_d\psi\circ g ( \frac{x'}{x_d}) )}
	\Big |{\bf K} \Big( g \Big( \frac{x'}{x_d}\Big), \, \psi\circ g \Big(\frac{x'}{x_d} \Big)  \Big) \Big|^{-\frac12} \\
	&	\qquad	 \times \int e^{ix_d\xi_d} \chi_1(\xi_d)	 A_{\de,\la} \Big( g\Big( \frac{x'}{x_d} \Big),\xi_d \Big) d\xi_d, \nonumber 
\end{align}
for $l\in\N$.  The contributions from the other terms given by replacing $A_{\de,\la}$ with $\mathcal D_jA_{\de,\la}$ can be handled similarly.  In fact, since $\mathcal D_j$ are only involved with derivatives in $\xi'$,   by making use of  Lemma \ref{diff_uniform}  it is easy to see that $\mathcal D_jA_{\de,\la}$ satisfies  the same bounds  \eqref{diff_uniform1} and \eqref{diff_uniform2}. See Remark \ref{DjA} below.  Thus, we may repeat the same argument for those terms but they give even better bounds because of the additional decay factor $|x_d|^{-j}$.   Therefore, for \eqref{cstomul1} we need only show that 
\begin{equation}\label{cstomul2}
	\sum_{l=1}^\infty \| \wt K_{\de,\la,l}*f \|_p\le C \la^{-1} (\de\la)^{\frac dp - \frac{d-1}2}\|f\|_p.
\end{equation}
By scaling, the $L^p$--$L^p$ norm of the convolution operator $f\to K*f$ is equal to that of $f\to L^d K(L \, \cdot\,)*f$ for any $L>0$. Thus for \eqref{cstomul2} we are reduced to showing that for a large enough $M>0$
\begin{equation}\label{lpiece}
	\big\| 2^{dl}\wt K_{\de,\la, l}(2^l \,\cdot\,)* f \big\|_p \lesssim_M 2^{(\frac{d+1}{2} -\frac{d}{p})l} \de (1+2^l \de\la)^{-M}  \|f\|_p.
\end{equation}
We sum \eqref{lpiece} over $l$ considering separately the cases $l\ge\log_2(\frac1{\de\la})$ and $l< \log_2(\frac1{\de\la})$ to get \eqref{cstomul2}. Indeed,
\[	\sum_{2^{l}\ge (\de\la)^{-1}} \|\wt K_{\de,\la,l}*f\|_p \lesssim \de (\de\la)^{-M} \|f\|_p \sum_{2^{l}\ge(\de\la)^{-1}} 2^{(\frac{d+1}2-\frac dp-M)l} \lesssim \la^{-1} (\de\la)^{\frac dp-\frac{d-1}{2}} \|f\|_p	\]
by choosing $M>\frac{d+1}2-\frac dp$. On the other hand, since $\frac{d+1}2> \frac dp$, 
\[	\sum_{2\le 2^{l}< (\de\la)^{-1}} \|\wt K_{\de,\la, l}*f\|_p \lesssim \de  \|f\|_p \sum_{2\le 2^{l}< (\de\la)^{-1}} 2^{(\frac{d+1}{2} -\frac{d}{p})l} \lesssim \la^{-1} (\de\la)^{\frac dp-\frac{d-1}{2}} \|f\|_p.	\]
Combining these two estimates we get \eqref{cstomul2}.

We now turn to the proof of  \eqref{lpiece}.  Since the kernel $\wt K_{\de,\la, l}(2^l x)$ is supported in the set $\{ x: |x'|<2^{-4},\,  3/8\le |x_d|\le 9/8 \}$, it is possible to  show the local estimate
\begin{equation}\label{local_operator}
	\big\| 2^{dl} \wt K_{\de,\la,l} (2^l\,\cdot\,)*f \big\|_{L^p(B_d(x_\circ,1))} \le C_M 2^{(\frac{d+1}{2}-\frac{d}p)l} \de  (1+2^l \de\la)^{-M} \|f\|_{L^p(B_d(x_\circ, 4))}
\end{equation}
with $C_M$ independent of  $l$, $\de$, $\la$ and $x_\circ\in\R^d$. Estimate \eqref{lpiece} follows directly from \eqref{local_operator} by integrating with respect to the $x_\circ$-variable and using Fubini's theorem.   The rest of this section is devoted to proof of \eqref{local_operator}.   Clearly, we may assume that $x_\circ=0$ by translation. 

Let us set  $\wt\be(t)=|t|^{-\frac{d-1}{2}}\be(t)$ and fix a function $\be_\circ\in C_0^\infty \big( (-2,-2^{-2})\cup (2^{-2},2)\big)$ such that $\wt{\be}\be_\circ=\wt{\be}$.  We also set 
\begin{equation}\label{amplitude-eq}
\mathfrak a_{\de,\la, l} (x) :=\wt\chi \Big(\frac{|x'|}{x_d} \Big)  \be_\circ(x_d)
	\Big|{\bf K}  \Big(g \Big( \frac{x'}{x_d}\Big), \, \psi\circ g \Big(\frac{x'}{x_d}\Big)  \Big) \Big|^{-\frac12}
		\int e^{i 2^l x_d \ta} \chi_1 (\ta)	 A_{\de,\la} \Big( g\Big( \frac{x'}{x_d} \Big), \ta \Big) d\ta
\end{equation}
and 
\[	\Phi (x,y) := (x'-y')\cdot g \Big( \frac{x'-y'}{x_d-y_d} \Big) +(x_d-y_d) (\psi\circ g) \Big( \frac{x'-y'}{x_d-y_d}  \Big).	\] 
Let $\eta$ be a nonnegative smooth function $\eta\in C_0^\infty(B_d(0,2))$ whose value is equal to $1$ on the unit ball $B_d(0,1)$.  Freezing $y_d$ we put
\begin{equation}\label{phase}
	\Phi^{y_d} (x,y'):=\Phi (x,y',y_d), \quad   \mathfrak a_{\de,\la, l}^{y_d} (x,y'):=   \eta(x) \mathfrak a_{\de,\la, l} (x-(y',y_d)). 
\end{equation}
Then from \eqref{wkernel} and the choice of $\wt \beta$ it  is clear that,  for $x\in B_d(0,1)$, 
\begin{equation}\label{conv}
\big( 2^{dl}\wt{K}_{\de,\la,l} (2^l \,\cdot\,)* f \big)(x) 
	= 2^{\frac{(d+1)l}{2}} \int \wt\be(x_d-y_d) \big( T_{2^l} [\Phi^{y_d}, \mathfrak a_{\de,\la, l}^{y_d} ] f(\,\cdot\, , y_d) \big) (x) dy_d .
\end{equation}

Next, we show that the phase $\Phi^{y_d}$ in \eqref{phase} satisfies the Carleson--Sj\"olin condition (\eqref{mixhess}, \eqref{curv}) and the elliptic condition \eqref{samesign} uniformly in  $\psi\in{\bf Ell}(N,\epsilon)$ and $y_d\in[-4,4]$ on  the set 
\[   \mathbf S_{y_d}:= \{ (x,y')\in \R^d\times\R^{d-1}: |x|\le 2, ~  |x'-y'|\le  2^{-3}, ~   2^{-2}\le|x_d-y_d|\le 2  \}.\] 
Let us write $g = (g_1, \cdots , g_{d-1})$.  Differentiating \eqref{phase} directly and then using \eqref{gauss_map} it is easy to see that
\[	\p_{x'}  \pyd (x,y')=  g\Big( \frac{x'-y'}{x_d-y_d} \Big), \quad  {\p_{x_d} \Phi^{y_d} (x,y')}= \psi\circ g \Big( \frac{x'-y'}{x_d-y_d} \Big).	\]
Differentiating these equations with respect to $y'$ the rank condition \eqref{mixhess} can be easily verified by \eqref{jacobian}.  For $v=(v',v_d)=(v_1,\cdots , v_d)\in \R^d$ we see that
\begin{equation}\label{first_diff}
	\p_{y'} \langle v, \p_x \pyd(x,y') \rangle = \frac{-1}{x_d-y_d}  \Big( v'+ v_d (\nabla \psi)\circ g \Big( \frac{x'-y'}{x_d-y_d} \Big) \Big) {\bf J}g \Big( \frac{x'-y'}{x_d-y_d} \Big).
\end{equation}
Hence, for fixed  $y_d\in[-4,4]$ and $(x, y'_\circ)\in \mathbf S_{y_d}$,  the unique (up to sign) direction $v$ in \eqref{curv} can be chosen as
\[	v=\frac{w}{|w|},\quad w= \Big(- (\nabla\psi)\circ g \Big( \frac{x'-y'_\circ}{x_d-y_d} \Big) , 1 \Big) = \Big( \frac{x'-y'_\circ}{x_d-y_d} ,1 \Big),	\]
where the second equality holds because of \eqref{gauss_map}.  By a straightforward computation we see that 
\[	\p_{y'}^2 \langle v, \p_x \pyd(x,y') \rangle\vert_{y'=y_\circ}=
	\frac{1}{(x_d-y_d)^2|w|}  (H \psi)\circ g \Big( \frac{x'-y'_\circ}{x_d-y_d} \Big) ({\bf J} g)^2 \Big( \frac{x'-y'_\circ}{x_d-y_d} \Big).	\] 
Since $-{\bf J} g=((H\psi)\circ g)^{-1}$ is  close to the identity matrix $I_{d-1}$ (see \eqref{jacobian} and \eqref{elliptic_error}), we see that the nondegeneracy \eqref{curv} and the ellipticity \eqref{samesign} hold whenever  $(x,y',y_d)\in \bigcup_{|y_d|\le 4} \mathbf S_{y_d}\times\{y_d\}$. 

For the moment, let $\mathfrak a\in C_0^\infty(\mathbf S_{y_d})$ for $y_d\in[-4,4]$.  We apply Theorem \ref{osc-est1} to the operator $T_\rho [\pyd, \mf a]$, which gives for $p>p_*$ 
\begin{equation}\label{wtT}  
	 \| T_\rho [\pyd, \mf a] h\|_{L^p(\mathbb R^d)} \lesssim   \rho^{-d/p} \|h\|_{L^p(\R^{d-1})}.
\end{equation} 
The bound is uniform not only for $y_d\in[-4,4]$ but also for $\psi\in{\bf Ell}(N,\epsilon)$ (see Remark \ref{stability}). 

To get estimate for  $T_{2^l} [\Phi^{x_d}, \mathfrak a_{\de,\la, l}^{y_d}]$ in \eqref{conv} we need to replace $\mathfrak a$ in  \eqref{wtT} with $\mathfrak a_{\de,\la, l}^{y_d}$.  For the purpose  we need the following which is a modification of \cite[Lemma 2.1]{tao-1999}.  In fact,  for the proof  of Lemma \ref{amplitude} one only need to expand $\mathfrak a_1$ into the Fourier series. 

\begin{lem}[\cite{tao-1999}]\label{amplitude}  
Let  $\mathfrak a_0, \mathfrak a_1 \in C_0^\infty(\R^d \times \R^{d-1})$ such that $ \mathfrak a_0\mathfrak a_1=\mathfrak a_1$. Suppose $\|T_\rho [\Phi,\mathfrak a_0]h\|_q \le L\|h\|_p$ and suppose $|\partial^\alpha \mathfrak a_1| \le B$ for $|\alpha|\le 2d$. Then, there is a constant $C$ independent of $\Phi$ and $\mathfrak a_1$ such that $\|T_\rho  [\Phi,\mathfrak a_1]h\|_{q}\le CB L\|h\|_p$. 
\end{lem}

\begin{lem}\label{symbol}
Let $0<\de\le \de\la \le 1$, $b>0$ and let $\epsilon>0$ be small enough. Then for every $M\ge 0$ and every muti-index $\al$ such that $|\al|\le \frac{N}2-M$ there exists a constant  $C_{\al, M}$, independent of $\de$, $\la$, $l$ and $(\psi, m)\in {\bf Ell}(N,\epsilon)\times{\bf Mul}(N,b)$, such that
\begin{equation}\label{symbol_decay}
	|\p^\al_x \mathfrak a_{\de, \la , l} (x) | \le C_{\al, M} \de (1+2^l \de\la)^{-M}.
\end{equation}
\end{lem}

Now,  by combining the estimate \eqref{wtT}, Lemma \ref{amplitude} and Lemma \ref{symbol} we obtain the  estimate for  $T_{2^l} [\Phi^{y_d}, \mathfrak a_{\de,\la, l}^{y_d}]$. Indeed, observe that in \eqref{amplitude-eq} and \eqref{phase} the amplitude $\eta(x)\mathfrak a_{\de,\la, l} (x-(y',y_d))$ is nonzero only if $|x|\le 2$, $3/8\le|x_d-y_d|\le 9/8$ and $|x'-y'|\le   2^{-4}$. Taking a smooth function $\mathfrak a$ such that $\supp\mf a\subset\bigcup_{|y_d|\le4}\mathbf S_{y_d}$ and $\mathfrak a\,\mathfrak a_{\de,\la, l}^{y_d}=\mathfrak a_{\de,\la, l}^{y_d}$ for all $y_d\in[-4,4]$, we may apply Lemma  \ref{amplitude}  and Lemma \ref{symbol} to get 
\begin{equation}\label{estT} 
	\| T_{2^l} [\Phi^{y_d}, \mathfrak a_{\de,\la, l}^{y_d}] h\|_{L^p(\mathbb R^d)} \le  C_M 2^{-dl/p} \de  (1+2^l \de\la)^{-M} \|h\|_{L^p(\R^{d-1})}. 
\end{equation} 
We now recall  \eqref{conv}  and  use   Minkowski's inequality to obtain
\begin{align*}
\|  2^{dl}\wt{K}_{\de,\la,l} (2^l \,\cdot\,)* f  \|_{L^p(B_d(0,1))} 
	&\le 2^{\frac{(d+1)l}{2}} \Big(\! \int_{|x|\le 2} \!\!\Big(\int_{-4}^4 \big|   \wt\be(x_d-y_d) \big( T_{2^l} [\Phi^{y_d}, \mathfrak a_{\de,\la, l}^{y_d}] f(\cdot, y_d) \big)(x)\big| dy_d \Big)^p  \! dx \Big)^{\frac1p} \\
	&\le 2^{\frac{(d+1)l}{2}} \int_{-4}^4  \Big( \int \big|\big( T_{2^l} [\Phi^{y_d}, \mathfrak a_{\de,\la, l}^{y_d}]f(\cdot, y_d) \big)(x)\big|^p dx \Big)^\frac1p dy_d . 
\end{align*}
Finally using \eqref{estT} which is followed by integration in $y_d$ gives the desired estimate \eqref{local_operator}.  To complete proof of Proposition \ref{cstomul} it remains to show Lemma   \ref{symbol}. 
\end{proof}

\begin{proof}[Proof of Lemma \ref{symbol}] 
Let us set 
\[	\mathcal  I_{\de,\la, l}(x) := \int e^{i 2^l x_d \ta} \chi_1 (\ta)	 A_{\de,\la} \Big( g\Big( \frac{x'}{x_d} \Big), \ta \Big) d\ta.	\] 
Since the term  $\wt\chi \big(\frac{|x'|}{x_d}\big)  \be_\circ(x_d) \big|{\bf K}  \big(g \big( \frac{x'}{x_d}\big), \psi\circ g \big(\frac{x'}{x_d}\big)  \big) \big|^{-1/2}$ in \eqref{amplitude-eq} has bounded derivatives of any order  it is sufficient to show that  for $2^5|x'|\le |x_d|\approx 1$ 
\begin{equation}\label{lovely_diff}
	|\p_x^\al \mathcal  I_{\de,\la, l}(x) |    \le C_{\al, M} \de (1+2^l \de\la)^{-M}.
\end{equation}
Let us first consider the case $|\al|=0$.  By integration by parts  
\[	\mathcal  I_{\de,\la, l}(x) = \Big(\frac{-1}{i2^l x_d}\Big)^M \int e^{i 2^l x_d \ta} \Big(\frac{d}{d\ta}\Big)^M \Big( \chi_1 (\ta)	 A_{\de,\la} \Big( g\Big( \frac{x'}{x_d} \Big), \ta \Big) \Big) d\ta .	\]
Since $0<\de\la\le 1$,  recalling \eqref{Adl}, \eqref{diff} and using Lemma \ref{diff_uniform} (\eqref{diff_uniform2} with $|\al'|=|\vartheta|=\ell=0$), we get  
\[	\Big| \Big( \frac{d}{d\ta} \Big)^M \Big( \chi_1 (\ta) A_{\de,\la} \Big( g\Big( \frac{x'}{x_d} \Big), \ta \Big) \Big)  \Big| \le C_M \la^{-1} (\de\la)^{-M}	\]
for $M\le N$.  Thus we obtain the desired bound \eqref{lovely_diff} when $|\al|=0$. 

Next we turn to proof of \eqref{lovely_diff} for the case $|\al|\ge 1$. We observe that the case $\al_d=0$ can be handled similarly as before in the case $|\al|=0$  by making use of  Lemma  \ref{diff_uniform} (\eqref{diff_uniform2} with $\ell=0$) since the derivative $\p_{x}^{\al}=\p_{x'}^{\al'} $ produces additional terms given by $( \partial_{\xi'}^{\vartheta'} A_{\de,\la}) \big( g\big( \frac{x'}{x_d}, \tau\big) \big)$, $|\vartheta'|\le |\al'|$.  However, if $\p_{ x_d}$ is involved we need to be additionally careful.  Note that
\begin{align*}
\p_{ x_d} \mathcal I_{\de,\la, l}(x) 
	&=  i2^l  \int  e^{i 2^l x_d \ta} \chi_1 (\ta) \,  \ta A_{\de,\la} \Big( g\Big( \frac{x'}{x_d} \Big), \ta \Big) d\ta  \\
	&\qquad \qquad -\frac{1}{x_d^2} \int e^{i 2^l x_d \ta} \chi_1 (\ta)	(\nabla_{\xi'} A_{\de,\la} )^{t}\Big( g\Big( \frac{x'}{x_d} \Big), \ta  \Big) \cdot  {\mathbf J} g\Big( \frac{x'}{x_d}\Big)\cdot x' d\ta .
\end{align*}
For the first term,  using Lemma  \ref{diff_uniform}  (\eqref{diff_uniform2} in with $\ell=1$) and repeating the same argument as before in the case $|\al|=0$,  we see that it is bounded by  $C_{\al, M} \de (1+2^l \de\la)^{-M+1}$. For the second term we use  \eqref{diff_uniform1} to see that  this is bounded by $C_{\al, M} \de (1+2^l \de\la)^{-M}$. Then we may repeat the same argument for general $\p_x^\al$ to get  
\[	| \p_{ x_d}^{\al_d} \p_{x'}^{\al'} \mathcal I_{\de,\la, l}(x)|\le C_{\al, M} \de (1+2^l \de\la)^{-M+\alpha_d} 	\] 
for any  $M+|\alpha|\le N$.  
\end{proof}

\begin{rem}\label{DjA}   It is not difficult to see that the same estimate for ${\mf a}_{\de,\la, l}$ remains valid  even if we replace $A_{\de,\la}$  in \eqref{amplitude-eq} with $\mathcal{D}_j A_{\de,\la}$ which appears in \eqref{stationary}.  This is due to Lemma \ref{diff_uniform} and the fact that $\mathcal{D}_j $ is given by derivatives in $\xi'$, thus  the above argument also works.  
\end{rem}

\subsection{Bilinear estimates for multiplier operators}
In this section we obtain bilinear $L^2\times L^2\to L^{q/2}$ estimates for the multiplier operators $\mm_\de(D)$ and $\mm_{\de, \la}(D)$ when $q>2(d+2)/d$. For this let us first recall the bilinear estimate for the extension operators given by elliptic surfaces which is due to Tao \cite{T}. 
\begin{thm}[\cite{T}]\label{bi_tao}
Let $q >\frac{2(d+2)}{d}$, $a_\circ \in (2^{-5}, 1/2]$. Then there exist $N$, $\epsilon$ and $C=C(d,q, a_\circ)$ such that
\begin{equation}\label{birest}
\Big\| \prod_{k=1,2} \int_{\mI^{d-1}} h_k(\xi') e^{i( x' \cdot \xi' +x_d\psi(\xi'))} d\xi' \Big\|_{L^{q/2}(\R^d)} \le C \prod_{k=1,2}\|h_k\|_{L^2([-1, 1]^{d-1})}
\end{equation}
for all $\psi\in{\bf Ell}(N,\epsilon)$ and all $h_1,\, h_2\in L^2(\mI^{d-1})$ satisfying $\dist(\supp h_1, \supp h_2)\ge a_\circ$.
\end{thm}

From Theorem \ref{bi_tao} we deduce the following bilinear estimate. We follow the proof of \cite[Lemma 2.4]{Lee1} (also, see \cite[Lemma 3.1]{LRS}).
\begin{cor}\label{L2bi}
Let $q$, $a_\circ$, $N$, $\epsilon$ and $\psi$ be as in Theorem \ref{bi_tao} and let $\de$, $\la$, $b$, $m$ and $\mm_{\de, \la}$ be given as in Proposition \ref{dethick}.  Suppose that 
\begin{equation}\label{trsv}
(\xi', \xi_d)\in\supp \wh{f_1}, \quad (\zeta ', \zeta_d)\in\supp \wh{f_2} \quad \Longrightarrow \quad |\xi ' - \zeta '| \ge a_\circ.
\end{equation}
Then there is a constant $C$, independent of $\de,$ $\la$, $\psi$ and $m$, such that, for $f_1, f_2\in L^2(\R^d)$ satisfying \eqref{trsv}, 
\begin{equation}\label{l2bi}
\Big\| \prod_{k=1,2} \mm_{\de, \la}(D)f_k\Big\|_{L^{q/2}(\R^d)} \le C \de\la^{-1} \prod_{k=1,2} \|f_k\|_{L^2(\R^d)}.
\end{equation}
The estimate \eqref{l2bi} holds if $\mm_{\de, \la}$ and $\de\la^{-1}$  are replaced with $\mm_\de$ and $\de$, respectively.
\end{cor}
\begin{proof}
Recalling Remark \ref{simple_cutoff} and \eqref{Adl} and  changing variables $(\xi',\xi_d)\to(\xi', \xi_d+\psi(\xi'))$, we see that for $k=1,2,$
\[
|\mm_{\de,\la}(D)f_k (x)| \lesssim  \int  \Big| \int e^{i (x'\cdot \xi'+x_d\psi(\xi'))} A_{\de,\la}(\xi) \wh{f_k} (\xi', \xi_d+\psi(\xi')) d\xi' \Big| \Big| \wt{\be} \Big( \frac{\xi_d}{\de\la} \Big) \Big| d\xi_d,
\]
where $\wt{\be}\in C^\infty_0 \big( (-4,-1/8) \cup (1/8,4) \big)$ satisfies $\wt{\be}=1$ on the $\xi_d$-support of $\be(\wt{m}(\xi)\xi_d)$.  Freezing $\xi_d$ we apply  the bilinear extension estimate \eqref{birest} to $ h_k(\xi')=A_{\de,\la}(\xi', \xi_d)  \wh{f_k} (\xi',\xi_d +\psi(\xi'))$, $k=1,2$. By the condition \eqref{trsv}, $\dist(\supp h_1, \supp h_2)\ge a_\circ$. Thus from Theorem \ref{bi_tao} and Minkowski's inequality we see that the left side of \eqref{l2bi} is bounded by
\begin{align}
	& \iint \Big\| \prod_{k =1,2} \int e^{i( x' \cdot \xi' +x_d \psi(\xi'))} A_{\de,\la}(\xi', \ta_k)  \wh{f_k} (\xi',\tau_k +\psi(\xi')) d\xi' \Big\|_{L^{\frac q2}(dx)}
		 \Big |\wt{\be}\Big(\frac{\tau_1}{\de\la}\Big)		\wt{\be} \Big(\frac{\tau_2}{\de\la} \Big) \Big| d\tau_1d\tau_2 \nonumber\\
	\label{l2bi1}
	&\le C  \iint \prod_{k =1,2} \big\|{ A_{\de,\la}(\xi', \ta_k)  \wh{f_k} (\xi' , \tau_k+\psi (\xi') ) } \big\|_{L^2(\R^{d-1};\,  d\xi')} 
		\Big |\wt{\be}\Big(\frac{\tau_1}{\de\la}\Big)		\wt{\be} \Big(\frac{\tau_2}{\de\la} \Big) \Big|  d\tau_1 d\tau_2 ,
\end{align}
where $C$ is independent of $\de$, $\la$ and $\psi$.  Since $\wt{m}\in {\bf Mul}(N,Cb)$ for some $C>0$, from \eqref{diff_uniform1} in Lemma \ref{diff_uniform}, we note that  $|A_{\de,\la}|\lesssim \lambda^{-1}$.  By the Cauchy--Schwarz inequality and the change of variables $\tau_k \to \tau_k -\psi(\,\cdot\,)$,  we see \eqref{l2bi1} is bounded by 
\begin{align*}
C\de \la \Big( \prod_{k=1,2} \int \la^{-2} |\wh{f_k}(\xi)|^2 d\xi \Big)^{1/2}.
\end{align*}
The inequality \eqref{l2bi} follows from Plancherel's identity.  The estimate for $\mm_\de(D)$ can be obtained in exactly the same way. 
\end{proof}

Before closing  this subsection, we state  a result which is necessary to prove Proposition \ref{dethick} in the next section. Trivially, by the Cauchy--Schwarz inequality and Proposition \ref{cstomul} it follows that
\begin{equation}\label{fake_bilinear}
\Big\| \prod_{k=1,2} \mm_{\de, \la}(D) f_k \Big\|_{L^{q/2}(\R^d)} \le C \la^{-2} (\de\la)^{ \frac {2d}{q} - d+1} \prod_{k=1,2} \|f_k\|_{L^q(\R^d)}
\end{equation}
whenever $p_*<q\le\infty$. Under the additional transversality condition \eqref{trsv} we have \eqref{l2bi}. Since  \eqref{fake_bilinear} holds regardless of  \eqref{l2bi}, we may interpolate this with \eqref{l2bi} while assuming \eqref{trsv}. This yields the following.
\begin{cor}\label{tcsbil}
Let $0<\de\le \la\lesssim 1$, $b>0$, $a_\circ\in(0,1/2]$ and suppose $2\le p\le q\le \infty$ and 
\begin{equation}\label{below_the_line}
\frac1{q}-\frac d{2(d+2)} < \frac{ d/[2(d+2)]- 1/p_*}{1/2-1/p_*} \Big( \frac1{p}-\frac12 \Big).\footnote{When $d=2$ this is $1/q<1/4$. When $d\ge3$ this is equivalent to saying that $(1/p,1/q)$ lies strictly below the line passing through the points $P_*$ and $P_\circ$. See Figure \ref{figthm2} and Figure \ref{figthm}. }
\end{equation}
Then, there exist  large $N$, small $\epsilon$, and  $C>0$ such that 
\begin{align}
\label{bimult1}
	\Big\| \prod_{k=1,2} \mm_{\de}(D) f_k \Big\|_{L^{q/2}(\R^d)} &\le C \de^{1-d+\frac{2d}p} \prod_{k=1,2}\|f_k\|_{L^p(\R^d)},\\
\label{bimult}
	\Big\| \prod_{k=1,2} \mm_{\de, \la}(D) f_k \Big\|_{L^{q/2}(\R^d)} &\le C\la^{-2} (\de\la)^{1-d+\frac{2d}p} \prod_{k=1,2}\|f_k\|_{L^p(\R^d)},
\end{align}
for  $\psi\in{\bf Ell}(N,\epsilon)$,  $m\in\mnb$, and $f_1$ and $f_2$ satisfying the separation \eqref{trsv}. 
\end{cor}

\subsection{Bochner--Riesz operator of negative order} \label{REBR}
If $(1/p,1/q)\in \{B, B'\} \cup [A',B')$, then the (restricted) weak type estimates stated in Theorem \ref{thm} can be obtained as consequences of the well-known estimates for the restriction-extension operator $f \to \F^{-1}\big( \wh f d\sigma \big) $ which is defined by
\[	\F^{-1}\big( \wh f d\sigma \big) (x)= \frac{1}{(2\pi)^d}\int_{\mathbb{S}^{d-1}} \wh f (\theta) e^{ix\cdot\theta}d\sigma(\theta) ,	\]
where $d\sigma$ is the surface measure on the unit sphere $\mathbb S^{d-1}$. In fact, this is a special case of order $-1$ of the classical Bochner--Riesz operator 
\begin{equation}\label{nbr}
    \mathcal R^{\alpha} f=  \F^{-1}\Big(\,\frac{(1-|\xi|^2)_+^{\al}}{\Gamma(1+\alpha)}\, \wh f(\xi)\Big) 
\end{equation}
which is defined by analytic continuation when $\al\le-1$.  Here $\Gamma$ is the gamma function.  For $d\ge2$ and $\al\in (0,\frac{d+1}2]$ let us set 
\[  P_\al(d):=\Big( \frac{d-1}{2d} +\frac\alpha{d},\, 0\Big), \quad   Q_\al(d):= \Big( \frac{d-1}{2d} + \frac \alpha{d}, \,\frac{d-1}{2d}-\frac{\alpha(d-1)}{d(d+1)} \Big),   \]
and
\begin{equation}\label{region_pentagon}
\mathcal P_\al(d) := \Big\{(x,y)\in I^2: x-y\ge \frac{2\al}{d+1}, ~ x>\frac{d-1}{2d} +\frac{\al}d , ~ y<\frac{d+1}{2d}-\frac\al d \Big\}. 
\end{equation} 
The following has been conjectured.

\begin{conj}\label{conj2}
Let $d\ge2$ and $0<\alpha < \frac{d+1}2$. $\mathcal R^{-\alpha}$ is bounded from $L^p(\R^d)$ to $L^q(\R^d)$ if and only if $(1/p, 1/q)\in \mathcal P_\alpha(d)$. 
\end{conj}

This problem was studied by several authors \cite{Bo, CS88, bak, Gu99, CKLS}. The complete characterization  of the necessity part is due to B\"orjeson  \cite{Bo}.  Estimates for $\mathcal R^{-\al}$ with $\alpha> 1/2$ and $(1/p, 1/q)\in \mathcal P_\alpha(d)\setminus  (Q_\alpha(d), Q_\alpha'(d))$ were obtained by Sogge \cite{sogge-1986}.  Partial results regarding the critical estimate with $(1/p, 1/q)\in (Q_\alpha(d), Q_\alpha'(d))$ were obtained by Bak, McMichael and Oberlin \cite{BMO}.  When $d=2$, the conjecture was solved by Bak \cite{bak}. The restricted weak type estimates at $Q_\alpha(d)$ and $Q_\alpha'(d)$ were proven by Guti\'errez \cite{Gu99} for $\alpha>0$ when $d\ge2$, and for  $\alpha>1/2$  when $d\ge 3$. The conjecture was verified by Cho, Kim, Lee and Shim \cite{CKLS}  for  $\alpha>\frac{(d-2)(d+1)}{2(d-1)(d+2)}$ and weaker endpoint estimates were also obtained.   

From Proposition \ref{dethick} and typical dyadic decomposition we can improve the current state of the boundedness of $\mathcal R^{-\al}$.
\begin{thm}\label{boc_rie_neg}
If $d\ge 3$ and $\alpha>\frac{d+1}{2}(\frac1{p_\circ}-\frac1{q_\circ})$  (that is to say,  $\alpha> \frac{(d+1)(d-1)}{2(d^2+4d-1)}$  if $d$ is odd and $\alpha >\frac{(d+1)(d-2)}{2(d^2+3d-2)}$  if $d$ is even),  then  Conjecture \ref{conj2} is true. Moreover, $\mathcal R^{-\alpha}$ is of restricted weak type $(p,q)$ when $(1/p, 1/q)\in \{Q_\al(d), Q_\al'(d)\}$, and of weak type $(p,q)$ if $(1/p, 1/q)\in (Q_\al'(d), P_\al'(d)]$. 
\end{thm}
\begin{proof}
By Proposition \ref{dethick} we may replace the condition $\frac{2d+4}{d}<q$ in \cite[Proposition 2.4]{CKLS} with $q_\circ<q$. Now the rest of the proof is identical with that of \cite[Theorem 1.1]{CKLS}.
\end{proof}
Especially, when $\alpha=1$, the result gives  the following characterization of $L^p$--$L^q$ boundedness for the restriction-extension operator,  which we need later. Recalling  \eqref{pentagon}, we note that $\mathcal P_1=\mathcal P$.
\begin{thm}[Restriction-extension estimates for the sphere] \label{rest_ext_sphere}   Let $d\ge 2$. The estimate 
\begin{equation}\label{rest_ext}
\Big\|\int_{\mathbb S^{d-1}} \wh f (\theta) e^{ix\cdot\theta} d\s(\theta) \Big\|_{L^q(\R^d)} \lesssim \|f\|_{L^p(\R^d)}
\end{equation}
holds if and only if $(1/p,1/q)\in \mathcal P$. Furthermore, for the critical $p,q$ such that $(1/p,1/q)=B$ or $B'$, the restricted weak type estimate holds instead of \eqref{rest_ext}. If $(1/p,1/q)\in (B',E']$, the weak type estimate holds (see Figure \ref{figthm2} and Figure \ref{figthm}).
\end{thm}

Finally, we record here the following real interpolation technique (see \cite{B, CSWaWr, Lee}), which will be needed several times in the succeeding sections. Here $\| \cdot\|_{r,s}$ denotes the norm of the Lorentz space $L^{r,s}$. 
\begin{lem}[\cite{Lee}]\label{intpl}
Let $\epsilon_1, \epsilon_2>0$, $1\le p_1^i,  p_2^i <\infty$, $1\le i \le k$, $1\le q_1, q_2 <\infty$. For every $j\in\Z$ let $T_j$ be $k$-linear operators satisfying $\|T_j(f_1, \cdots, f_k)\|_{q_1} \le M_1 2^{\epsilon_1 j}\prod_{i=1}^k \|f_i\|_{p_1^i}$ and $\|T_j(f_1, \cdots, f_k)\|_{q_2} \le M_2 2^{-\epsilon_2 j}\prod_{i=1}^k \|f_i\|_{p_2^i}$.  Then, for $\theta$, $q$ and $p_i$ defined by $\theta=\frac{\epsilon_2}{\epsilon_1+\epsilon_2}$, $\frac1q=\frac\theta{q_1}+\frac{1-\theta}{q_2}$ and $\frac1{p^i}=\frac\theta{p_1^i}+\frac{1-\theta}{p_2^i}$, the following hold: \\
$(\mathrm I)$
	$\| \sum_j T_j(f_1,\cdots, f_k)\|_{q,\infty} \le C M_0^\theta M_1^{1-\theta} \prod_{i=1}^k \|f_i\|_{p^i,1},$ \\[3pt]
$(\mathrm I\!\mathrm I)$
	$\| \sum_j T_j(f_1,\cdots, f_k)\|_{q} \le C M_0^\theta M_1^{1-\theta} \prod_{i=1}^k \|f_i\|_{p^i,1}$  \, if \, $q_1=q_2=q$, \\[3pt]
$(\mathrm I\!\mathrm I\!\mathrm I)$
	$\| \sum_j T_j(f_1,\cdots, f_k)\|_{q,\infty} \le C M_0^\theta M_1^{1-\theta} \prod_{i=1}^k \|f_i\|_{p^i}$ \, if \, $p^i_1=p_2^i=p^i$ for every $i$.
\end{lem}

\section{Proof of Proposition \ref{dethick}}\label{proof_bil}
In order to deduce the linear estimates \eqref{dela_0} and \eqref{dela} from the bilinear estimates in Corollary \ref{tcsbil} we basically follow the strategy  in \cite{Lee1, CKLS} with some modifications. As before, we may only prove \eqref{dela}. The estimate \eqref{dela_0} can be obtained by the same argument. 

Let us put $Q=\mI^{d-1}$ and for every integer $j\ge 0$ let $\mathcal D(j)$ be the collection of the closed dyadic cubes of size $2^{-j}$ in $Q$, that is,
\[	\mathcal D(j) := \Big \{ \prod_{k=1}^{d-1} [n_k 2^{-j}, (n_k+1)2^{-j}] : \,  n_k\in \Z, \, -2^j\le n_k\le 2^j-1 \Big \}.	\]
For  convenience let us denote by  $Q^j_k$ the members of $\mathcal D(j)$. 

For every $j\ge 1$ we define a relation $\sim$ on the dyadic cubes  contained in $\mathcal D(j)$ as follows. For $Q^j_{k_1},\, Q^j_{k_2}\in\mathcal D(j)$ we write $Q^j_{k_1}\sim Q^j_{k_2}$ if $Q^j_{k_1} \cap Q^j_{k_2}=\emptyset$, but there are parent cubes in $\mathcal D(j-1)$ which have nonempty intersection. It is easy to see that  $2^{-j}\le \dist(Q^j_{k_1}, Q^j_{k_2}) \lesssim 2^{-j}$ if $k_1\neq k_2$. By a kind of  Whitney decomposition of $Q\times Q$ away from its diagonal $\Lambda_Q=\{(\xi',\xi') : \xi' \in Q\}$,
\[	Q\times Q\setminus \Lambda_Q =\bigcup_{j\ge 1} \bigcup_{Q^j_{k_1}\sim Q^j_{k_2}} Q^j_{k_1}\times Q^j_{k_2},	\]
hence
\begin{equation}\label{wdcmp}
\sum_{j\ge 1} \sum_{Q^j_{k_1}\sim Q^j_{k_2}} \chi_{Q^j_{k_1}}\chi_{Q^j_{k_2}} =1
\end{equation}
almost everywhere in $Q\times Q$ (\cite{TVV, Lee1, CKLS}). For  $Q_k^j\in\mathcal D(j)$ we define $f^j_k$ by 
\[	\wh{f^j_k} (\xi) = \chi_{Q^j_k}(\xi') \wh f(\xi).	\] 
As mentioned before (Remark \ref{simple_cutoff}), with $\chi_0$ supported near the origin  we may assume $\widehat f$ is supported in  $2^{-5}\mI^d$.  Then, by \eqref{wdcmp} we can write
\[	(\mm_{\de, \la} (D)f)^2=\sum_{j\ge 6}T_j(f,f):=\sum_{j\ge 6} \sum_{Q^j_{k_1}\sim Q^j_{k_2}} \prod_{i=1,2}\mm_{\de, \la} (D)f^j_{k_i}. 	\]
We now try to obtain sharp estimates for the bilinear operators $\{T_j: j\ge  6\}$.  We separately consider  the cases $2^{2j}\lesssim 1/\de\la$ and $2^{2j}\gtrsim 1/\de\la$. 

\begin{lem}\label{smj}
Let $p,q$ satisfy $2\le p < q \le 4$ and \eqref{below_the_line}, and suppose that $2^{2j}\de\la < 1/10$. Then, there  are $N$ and $\epsilon$ which are independent of such $p$, $q$, $j$, $\de$ and $\la$, such that  
\begin{equation}\label{smjeq}
\normo{T_j(f_1,f_2)}_{q/2} \le C 2^{2j(\frac{d+1}q -(d-1)(1-\frac 1p))} \la^{-2} (\de\la)^{1-d+\frac {2d}p} \|f_1\|_p\|f_2\|_p 
\end{equation}
for $\psi\in{\bf Ell}(N,\epsilon)$ and $m\in \mnb$. Here the constant $C$ is independent of $j$, $\de$, $\la$, $m$ and $\psi$.
\end{lem}

\begin{lem}\label{lgj}
Suppose $2\le p < q$, $\frac{2d}{d-1} \le q \le 4$ and $2^{2j}\de\la \ge 1/10$. Then there are  $N$ and  $\epsilon$, independent of such $p$, $q$, $j$, $\de$ and $\la$, such that
\begin{equation}\label{lgjeq}
\normo{T_j(f_1,f_2)}_{q/2} \le C 2^{-2j(d-1)(\frac 1p-\frac 1q)}\la^{-2} (\de\la)^{2(\frac 1p-\frac 1q)} \|f_1\|_p \|f_2\|_p
\end{equation}
for $\psi\in{\bf Ell}(N,\epsilon)$ and $m\in \mnb$. The constant $C$ is independent of $j$, $\de$, $\la$, $m$ and $\psi$.
\end{lem}

Assuming Lemma \ref{smj} and Lemma \ref{lgj}  for the moment,  we  prove Proposition \ref{dethick}. 

\begin{proof}[Proof of  Proposition \ref{dethick}] Choose $N$ and $\epsilon>0$ so that both Lemma \ref{smjeq} and Lemma \ref{lgjeq} hold.  For $p,q$ such that $\frac{d+1}q=(d-1)(1-\frac1p)$ and $q_\circ<q<\frac{2(d+1)}{d-1}$, applying  $(\mathrm I)$ in Lemma \ref{intpl} with $k=2$ to the estimate \eqref{smjeq} we get 
\begin{equation}\label{smjsum}
\Big\|\sum_{2^{2j}\de\la <\frac1{10}} T_j(f_1,f_2)\Big\|_{q/2,\infty} \le C \la^{-2} (\de\la)^{1-d+\frac {2d}p} \prod_{i=1,2}\|f_i\|_{p,1}.
\end{equation}
On the other hand, when $2\le p <q$ and  $\frac{2d}{d-1} \le q \le 4$,  direct summation of \eqref{lgjeq} over $j$ with $2^{2j}\de\la\ge 1$ gives
\begin{equation}\label{lgjsum}
\Big\| \sum_{2^{2j}\de\la \ge\frac1{10}} T_j(f_1,f_2)\Big\|_{q/2}  \le C \la^{-2} (\de\la)^{(d+1)(\frac1p-\frac1q)} \prod_{i=1,2}\|f_i\|_{p}.
\end{equation}
Combining \eqref{smjsum} and \eqref{lgjsum} we obtain the following restricted weak type estimate
\begin{equation}\label{rest_wk}
\normo{(\mm_{\de, \la}(D)f)^2}_{q/2,\infty} \le C \la^{-2} (\de\la)^{1-d+\frac{2d}p}\|f\|_{p,1}^2
\end{equation}
for $\psi\in{\bf Ell}(N,\epsilon)$ and $m\in \mnb$  whenever $\frac1q=\frac{d-1}{d+1}(1-\frac1p)$ and $\frac{d-1}{2(d+1)}<\frac1q<\frac1{q_\circ}$. On the same range of $p,q$ we can upgrade  the restricted weak type estimates \eqref{rest_wk} to strong type bounds by using (real) interpolation between those estimates. The case $(p,q) = (2,\frac{2(d+1)}{d-1})$  follows directly from the Stein--Tomas restriction estimate in the similar argument as in the proof of Corollary \ref{L2bi}. This completes the proof of Proposition \ref{dethick}.
\end{proof}

Before we proceed to show Lemma \ref{smj} and Lemma \ref{lgj}  we recall the following lemma which is a slight modification of \cite[Lemma 3.5]{CKLS}.  Since the proof of \cite[Lemma 3.5]{CKLS} works without modification, we state it without proof. 

\begin{lem}\label{piece}
Let $2\le p <q\le 4$. Suppose that there is a constant $L$, independent of all pairs $(Q_{k_1}^j,Q_{k_2}^j)$ with $Q^j_{k_1}\sim Q^j_{k_2}$, such that 
\begin{equation}\label{1piece}
\Big\| \prod_{i=1,2} \mm_{\de, \la}(D)(f_i)^j_{k_i}\Big\|_{q/2} \le L\prod_{i=1,2}\|(f_i)^j_{k_i}\|_p.
\end{equation}
Then there is a constant $ C$ independent of $j,\, \de,\, \la,$ such that
\begin{equation}\label{summation}
\normo{T_j(f_1,f_2)}_{q/2} \le C L\|f_1\|_p\|f_2\|_p.
\end{equation}
\end{lem}

\begin{proof}[Proof of Lemma \ref{smj}]
By Lemma \ref{piece} it is sufficient to show that, for $Q^j_{k_1}\sim Q^j_{k_2}$ and $p,q$ satisfying $2\le p<q\le4$ and \eqref{below_the_line}, 
\begin{equation}\label{smj0}
\Big\|\prod_{i=1,2} \mm_{\de, \la}(D)(f_i)^j_{k_i} \Big\|_{q/2} \le C 2^{2j(\frac{d+1}q -(d-1)(1-\frac 1p))} \la^{-2} (\de\la)^{1-d+\frac {2d}p} \prod_{i=1,2} \|(f_i)^j_{k_i}\|_p
\end{equation}
with $C$ independent of $j$, $k_1$, $k_2$, $\de$, $\la$, $\psi\in{\bf Ell}(N,\epsilon)$ and $m\in \mnb$. For  given $k_1$, $k_2$ with $Q^j_{k_1}\sim Q^j_{k_2}$ let $R(j,k_1,k_2)$ be the smallest closed $(d-1)$-dimensional rectangle containing $Q^j_{k_1}\cup Q^j_{k_2}$ and let $c\in \mathbb \mI^{d-1}$ be the center of $R(j,k_1,k_2)$ and set 
\[	\rho := 2^{1-j}.	\]
Now we perform the change of variables $\xi\to L_{c,\rho}(\xi)$  in the frequency side. See \eqref{affine} and \eqref{multi-mod}. By setting 
\[	\widehat g_i(\xi)=\rho^{d+1}\chi_{Q^j_{k_i}}(\rho\xi'+c) 
	\wh f_i (L_{c,\rho}(\xi)), \quad i=1, 2,	\] 
	one can easily see that
\begin{equation}\label{smj1} 
	|\mm_{\de, \la}(D)(f_i)^j_{k_i}(x)| = \big|\big[(\mm_{\de, \la}\circ L_{c,\rho})(D) g_i\big] (\rho x'+ \rho x_d \nabla\psi(c), \rho^2 x_d)\big|. 
\end{equation}
Thus we have
\begin{equation}\label{smj2}  
	\Big\|\prod_{i=1,2} \mm_{\de, \la}(D)(f_i)^j_{k_i} \Big\|_{q/2}= \rho^{-\frac{2(d+1)}q}  \Big\|  \prod_{i=1,2} ( \mm_{\de, \la}\circ L_{c,\rho})(D) g_i \Big\|_{q/2}.	
\end{equation}
We now notice that $\wh{g_i}$ is supported in  $\wt Q_i\times\mI$, where $\wt Q_i$'s are cubes in $\mI^{d-1}$ of sidelength $1/2$, and $\dist (\wt Q_1, \, \wt Q_2)\ge 1/2$.   We now recall \eqref{multi-mod} and that  $\mm_{\de, \la}\circ L_{c,\rho}$ can be regarded as a multiplier $\mm_{\de', \la}$ given by  putting $\de'=\rho^{-2}\de $, $m= m\circ L_{c,\rho}$, and $\psi=\psi_{c,\rho}$ (see Remark \ref{simple_cutoff}).  Since  $\psi_{c,\rho}\in {\bf Ell}(N,\mf c\epsilon)$  by  Lemma \ref{approximation} and   $m\circ L_{c,\rho}\in {\bf Mul}({ N}, Cb)$ for some $C>0$,  we can apply Corollary \ref{tcsbil}  to the right hand side of \eqref{smj2} with $\de$ replaced by $\rho^{-2}\de$,  to get 
\[	\Big\|\prod_{i=1,2} \mm_{\de, \la}(D)(f_i)^j_{k_i} \Big\|_{q/2}\lesssim \rho^{-\frac{2(d+1)}q} \la^{-2} (\rho^{-2}\de\la)^{1-d+\frac{2d}p}   \prod_{i=1,2}  \|g_i \|_p.	\]
It is easy to see that  $\|g_i \|_{p}= \rho^{\frac{d+1}{p}} \|(f_i)^j_{k_i}\|_p$. Hence, we get \eqref{smj0}. 
\end{proof}

\begin{proof}[Proof of Lemma \ref{lgj}] 
In order to show \eqref{lgjeq} by Lemma \ref{piece} it is sufficient to show 
\begin{equation}\label{bibi}
	\Big\| \prod_{i=1,2} \mm_{\de, \la}(D)(f_i)^j_{k_i} \Big\|_{q/2} \lesssim \la^{-2} (2^{-j(d-1)}\de\la)^{2(\frac1p-\frac1q)}\prod_{i=1,2}\|(f_i)^j_{k_i}\|_p\,.
\end{equation}
Let $\wt\chi\in C_0^\infty(\R^{d-1})$ be a smooth cutoff function supported in $\mI^{d-1}$ such that $\wt\chi = 1$ on $(2^{-1}\mI)^{d-1}$, and $c^j_{k_i}$ the center of $Q^j_{k_i}$. Then  $\wt\chi_{Q^j_{k_i}}(\xi')=\wt \chi (2^j(\xi'-c^j_{k_i}))$ is supported in $2Q^j_{k_i}$ and equal to $1$ on $Q^j_{k_i}$. Thus $\wt\chi_{Q^j_{k_i}} \widehat f_{k_i}^j=\widehat f_{k_i}^j$  and we may write $\mm_{\de, \la}(D)f^j_{k_i} = K^j_{k_i} * f^j_{k_i} $, where
\[	K^j_{k_i}(x)=  (2\pi)^{-d}\int e^{i\xi\cdot x} \mm_{\de, \la}(\xi) \wt\chi_{Q^j_{k_i}}(\xi') d\xi.	\]
For notational convenience let us set $c=c_{k_i}^j ,$ $ \rho=2^{-j}$.  By changing variables $\xi \to L_{c,\rho}\xi$ we have
\[	K^j_{k_i}(x)= (2\pi)^{-d} e^{i(x'\cdot c +x_d \psi(c))} \rho^{d+1}  \int e^{i(\rho(x'+x_d\nabla\psi(c))\cdot\xi' +\rho^{2} x_d \xi_d)}  \mathfrak M_{\de,\la} (L_{c,\rho}\xi)  \wt\chi(\xi') d\xi , 	\]
As before we regard $\mf M_{\de,\la} \circ L_{c,\rho}$ as a multiplier $\mm_{\de, \la}$ given by  $\rho^{-2}\de\to \de $, $m\circ L_{c,\rho}\to m$, and $\psi_{c,\rho}\to \psi$ (see Remark \ref{simple_cutoff}).  From \eqref{multi-mod}, \eqref{diff} and Lemma \ref{diff_uniform} it easily follows that  $|\p^\alpha_\xi \mf M_{\de,\la} (L_{c,\rho} \xi)|\lesssim \lambda^{-1}$ uniformly in $\psi\in \Ell$ and $m\in \mnb$ whenever $|\al|\le N$.  Since $(\mf M_{\de,\la}\circ L_{c,\rho}) \wt\chi$  is supported in $\mI^d$, it is clear that $\big|\F^{-1}\big((\mf M_{\de,\la}\circ L_{c,\rho} )\wt\chi\big)(x) \big|\lesssim \lambda^{-1}(1+|x|)^{-M}$ for any $M\le N$.  Thus  $\big\|\F^{-1}\big( (\mf M_{\de,\la}\circ L_{c,\rho}) \wt\chi)\big\|_1\lesssim \lambda^{-1}$, and trivially we also have  $\|K^j_{k_i}\|_\infty \lesssim \la^{-1} \rho^{d-1}\de\la. $  Thus it follows that $\|K^j_{k_i}\|_r \lesssim \la^{-1} (\rho^{d-1}\de\la)^{1-\frac 1r}$  for $1\le r\le \infty$. Since $p\le q$,  from Young's  inequality we see that
\[	\|\mm_{\de, \la}(D)f^j_{k_i}\|_q\lesssim \la^{-1} (\rho^{d-1}\de\la)^{\frac1p-\frac1q}\|f^j_{k_i}\|_p.	\]
Hence, by H\"older's inequality we get the desired estimate \eqref{bibi}. 
\end{proof}

\section{Resolvent estimates: Proof of Theorem \ref{thm}}\label{proof_of_main_thm}
\subsection{Reduction}\label{reduction}
For $z\in \zs$ let  us set 
\[	m(\xi, z)=(|\xi|^2-z)^{-1}.	\]
For every multi-index $\alpha$, it is easy to see that 
\begin{equation}\label{m_diff}
	| \partial_\xi^\alpha m(\xi, z)| \le C_\alpha \frac{\max \{|\xi|^{|\alpha|}, 1\}}{||\xi|^2-z|^{|\alpha|+1}},
\end{equation}
where the constant $C_\alpha$ is independent of $z\in \zs$. We decompose $m(\xi,z)$ into singular and regular parts. Let us fix a small number $\de_\circ>0$ and choose a function $\rho_0\in C^\infty_0(\R^d)$ such that $\rho_0(\xi)=1$ if $1-\de_\circ \le |\xi| \le 1+\de_\circ$ and $\rho_0(\xi)=0$ if $|\xi|\le 1-2\de_\circ$ or $|\xi|\ge 1+2\de_\circ$.  Setting 
\[	\rho_1:=(1-\rho_0)\chi_{B_d(0,1)}, \quad \rho_2:=(1-\rho_0)\chi_{\R^d\setminus B_d(0,1)}; \quad m_j(\xi, z):=m(\xi,z)\rho_j(\xi), \quad j=0,\, 1,\,2,	\]
we have  
\[	m(\xi,z) = \sum_{j=0}^2 m_j(\xi, z),	\] 
Since both $m_1$ and $m_2$ are zero on the annulus $1-\de_\circ\le|\xi|\le 1+\de_\circ$ it is easy to check that $||\xi|^2-z|\ge 2\de_\circ \pm\de_\circ^2$ on $\supp m_1\cup\supp m_2$. From this and \eqref{m_diff} it follows that $m_1$ and $m_2$ are uniformly bounded in $C^\infty(\R^d)$ for all $z\in \zs $. More precisely, for all $z\in\zs$, we have 
\begin{align}
\label{ufrm}  | \partial^{\alpha}_{\xi} m_1(\xi, z) |& \le C_{\alpha, \de_\circ}, \\
\label{ufrm1}  | \partial^{\alpha}_{\xi} m_2(\xi, z) | & \le C_{\alpha, \de_\circ} |\xi|^{-|\alpha|-2}. 
\end{align}
Since $\rho_1$ is a compactly supported smooth function, $m_1(D,z)$ are bounded from $L^p(\R^d)$ to $L^q(\R^d)$ for $1\le p\le q\le \infty $. Moreover, the bounds are independent of $z\in \zs$ because of  \eqref{ufrm}. On the other hand, the operators  $m_2(D, z)$ are uniformly bounded from $L^p(\R^d)$ to $L^q(\R^d)$ when $(1/p,1/q)\in \mathcal R_0(d)$, which can be seen in a similar manner as in the proof of Proposition \ref{eresb}  because of \eqref{ufrm1}. Hence it remains to deal with the operators $m_0(D,z)$. 

Let $\theta_\circ$ be a small number and set $\mathbb S^1(\theta_\circ):=\{e^{i\theta}\in\mathbb S^1: \theta\in[\theta_\circ, 2\pi-\theta_\circ] \}$.  By  \eqref{m_diff} we have,  for any $\alpha$ and $z\in \mathbb S^1(\theta_\circ)$, $| \partial^{\alpha}_{\xi} m_0(\xi, z) | \le C_{\alpha, \de_\circ, \theta_\circ}.$  Hence similar argument shows that $m_0(D,z)$ are bounded from $L^p(\R^d)$ to $L^q(\R^d)$ uniformly in $z\in \mathbb S^1(\theta_\circ)$ whenever $1\le p\le q\le \infty$. As a result, we conclude that the uniform estimate
\begin{equation}\label{ufrm2}
	\|(-\Delta-z)^{-1}\|_{p\to q} \le C, \quad \forall z\in \mathbb S^1(\theta_\circ)
\end{equation}
holds if $(1/p,1/q)\in\mathcal R_0(d)$.  

For the rest of this section, we focus on obtaining sharp bounds for $m_0(D,z)$ when $0<|\im z| \ll \re z < 1$,  which is the main part of  obtaining the estimate \eqref{con1}. By scaling $\xi \to (\re z)^{1/2} \xi$ it is harmless to assume that  $z=1+i\de$ and $0<|\de| <\theta_\circ$. Now we are reduced to showing that 
\begin{equation}\label{main_part}
\Big\| \F^{-1}\Big( \frac{\wt\chi (|\xi|) \wh f(\xi)}{|\xi|^s-1-i\de}\Big)\Big\|_q \lesssim |\de|^{-\gamma_{p,q}} \|f\|_p
\end{equation}
with $s=2$ and $\wt\chi\in C^\infty_0\big( (1-2\de_\circ,1+2\de_\circ) \big)$ for a  small  $\de_\circ>0$. 

The estimate \eqref{main_part} actually holds on a range (of $p, q$) which is wider than $\mathcal R_0$.   All the required estimates for the proof of Theorem \ref{thm} are contained in the following proposition, which completes the proof of Theorem \ref{thm}.  Though we need only deal with the case $s=2$, we prove \eqref{main_part} with  $s\neq 0$ for later use. Before stating the estimates we remind the reader of the definitions of $\gamma_{p,q}$, $B$, $B'$,  $\mathcal P$, $\mathcal T$ and $\mathcal Q$ (see \eqref{BBprime}--\eqref{quadrangle}) and $P_\ast$, $P_\circ$ (\eqref{pstar}). We also recall that $D=(\frac{d-1}{2d}, \frac{d-1}{2d})$ and $E'=(1,\frac{d-1}{2d})$. See Figure \ref{fig_added}. 
\begin{figure}
\captionsetup{type=figure,font=footnotesize}
\centering
\begin{tikzpicture} [scale=0.6]\scriptsize
	\path [fill=lightgray] (0,0)--(15/4,15/4)--(50/11, 40/11)--(6,8/3)--(6,0)--(0,0);
	\path [fill=lightgray] (10,4)--(10-8/3,4)--(10-40/11,10-50/11)--(10-15/4,10-15/4)--(10,10);
	\draw [EDR] (6,8/3)--(10-8/3,4)--(10,4)--(10,0)--(6,0);
	\draw [EDR1] (50/11, 40/11)--(5,5)--(10-40/11,10-50/11)--(10-8/3,4)--(6,8/3);
	\draw [<->] (0,10.7)node[above]{$y$}--(0,0) node[below]{$(0,0)$}--(10.7,0) node[right]{$x$};
	\draw (0,10) --(10,10)--(10,0) node[below]{$(1,0)$};
	\draw (4,4)node[above]{$D$}--(6,8/3)node[above]{$B$}--(10-8/3,4)node[left]{$B'$}--(6,6)node[above]{$D'$};
	\draw (0,0)--(4,4);
	\draw [dash pattern={on 2pt off 1pt}] (4,4)--(6,6);
	\draw (6,6)--(10,10);
	\draw [dash pattern={on 2pt off 1pt}] (6, 8/3)--(6, 2)--(6,0)node[below]{$E$};
	\draw [dash pattern={on 2pt off 1pt}] (10-8/3, 4)--(8, 4)--(10,4)node[right]{$E'$};
	\draw [dash pattern={on 2pt off 1pt}] (50/11, 40/11)--(5,5)node[above]{$H$}--(10-40/11,10-50/11)--(10-15/4,10-15/4);
	\draw (5.6, 4.4) node{$\wt{\mathcal R}_2$};
	\draw (4, 2) node{$\mathcal Q\setminus[P_\ast,P_\circ,D]$};
	\draw (8.17, 6) node{$\big(\mathcal Q\setminus[P_\ast,P_\circ,D]\big)'$};
	\draw (8, 2) node{$\mathcal P$};
	\draw [fill] (15/4,15/4) circle [radius=0.03];
	\draw [fill] (50/11, 40/11) circle [radius=0.03] node[below] {$P_\circ$};
	\draw [dash pattern={on 2pt off 1pt}] (15/4,15/4)node[left]{$P_*$}--(50/11, 40/11);
\end{tikzpicture}\caption{Proposition \ref{main_prop} when $d\ge3$.}\label{fig_added}
\end{figure}

\begin{prop}\label{main_prop}
Let $d\ge2$, $s\neq0 $, and   $(1/p, 1/q)\neq (1,0)$.  Suppose that $(1/p, 1/q) \in  \mathcal P \cup  \wt{\mathcal R}_2   \cup \big(\mathcal Q\setminus [P_\ast, P_\circ, D]\big) \cup \big(\mathcal Q\setminus [P_\ast, P_\circ, D]\big)'$.  Then the estimate \eqref{main_part} is true {provided that $\de_\circ$ is small.} If  $(1/p, 1/q)\in \{B, B'\}$,  then the restricted weak type $(p,q)$ estimate holds. If $(1/p,1/q)\in (B',E']$ the weak type $(p,q)$ estimate holds. 
\end{prop} 

In what follows we consider the cases $(1/p, 1/q)\in  \mathcal P \cup \{B,  B'\}\cup (B',E'] \setminus\{(1,0)\}$ and $(1/p, 1/q) \in \wt{\mathcal R}_2   \cup \big(\mathcal Q\setminus [P_\ast, P_\circ, D]\big) \cup \big(\mathcal Q\setminus [P_\ast, P_\circ, D]\big)'$, separately. 

\subsection{Proof of Proposition \ref{main_prop} when $(1/p, 1/q)\in \mathcal P\cup \{B, B'\}\cup (B',E']$ and $(1/p, 1/q)\neq (1,0)$}
It is enough to show the following: 
\begin{gather}
\label{rest_weak_type} 
	\Big\| \F^{-1}\Big( \frac{\wt\chi (|\xi|) \wh f(\xi)}{|\xi|^s-1-i\de}\Big)\Big\|_{\frac {2d}{d-1},\infty}
	\lesssim	\|f\|_{\frac{2d(d+1)}{d^2+4d-1},1},  \\
\label{weak_type}
	\Big\| \F^{-1}\Big( \frac{\wt\chi (|\xi|) \wh f(\xi)}{|\xi|^s-1-i\de}\Big)\Big\|_{\frac {2d}{d-1}, \infty}
	\lesssim	\|f\|_p, \quad 1\le p<\frac{2d(d+1)}{d^2+4d-1}.
\end{gather}
The estimates in  \eqref{weak_type} are the weak type $(p,q)$  estimates for $(1/p,1/q)\in (B',E']$, and  \eqref{rest_weak_type} is the restricted weak type $(p,q)$ estimate  with  $(1/p, 1/q)=B'$.  By duality, (real) interpolation, and Young's inequality (note that the multiplier has compact support), it is easy to see that the estimate \eqref{main_part} for $(1/p, 1/q)\in \mathcal P\setminus\{(0,1)\}$ follows from \eqref{rest_weak_type} and \eqref{weak_type}.  Indeed, note that $\gamma_{p,q}=0$ when $(1/p, 1/q)\in [(1,0), E, B, B', E']$.

We prove \eqref{rest_weak_type} and \eqref{weak_type} by making use of  Theorem \ref{rest_ext_sphere} as in \cite{JKL, KRS}.  Both arguments  to show \eqref{rest_weak_type} and \eqref{weak_type} are not much different from each other except for using different estimates in Theorem \ref{rest_ext_sphere}.

\begin{proof}[Proof of \eqref{rest_weak_type}] 
Let  us  fix $(1/p_0,1/q_0)=B'$ and write   
\[	\frac{\wt\chi (|\xi|)}{|\xi|^s-1-i\de}= \mathfrak R(\xi)+i\mathfrak I(\xi):= \frac{(|\xi|^s-1)\wt\chi (|\xi|)}{(|\xi|^s-1)^2+\de^2} + i \frac{\de \wt\chi (|\xi|)}{(|\xi|^s-1)^2+\de^2}.	\]  
Then \eqref{rest_weak_type} follows if we show  that both the operators $\mathfrak R(D)$ and $\mathfrak I(D)$ are of restricted weak type $(p_0, q_0)$.  The  desired estimate for $\mathfrak I(D)$ is easier than that for $\mathfrak R(D)$. Writing in the spherical coordinates, application of Minkowski's inequality and Theorem \ref{rest_ext_sphere} gives
\begin{align*}
\| \mathfrak I (D) f\|_{q_0,\infty} 
	&\lesssim \int_{1-2\de_\circ}^{1+2\de_\circ} \frac{|\de|}{(\rho^s-1)^2+\de^2} \Big\|\int_{\mathbb S^{d-1}}\wh f (\rho\theta) e^{i\rho x\cdot \theta} d\s(\theta) \Big\|_{q_0,\infty} d\rho \\
	&\lesssim \|f\|_{p_0,1} \int \frac{|\de|}{t^2+\de^2} dt \lesssim \|f\|_{p_0,1}.
\end{align*}

For the real part, we decompose the multiplier  $\mathfrak R(\xi)$ as in \cite[Section 4]{JKL}. Let $\phi\in \mathcal S(\R)$ be such that $\supp \wh \phi \subset [-2,-1/2]\cup[1/2,2]$, $\sum_{j=-\infty}^\infty 2^{-j}t \phi(2^{-j}t)=1$ whenever $t\in\R\setminus\{0\}$, and we set $\wt\phi (t) = t \phi (t)$.\footnote{For a proof of existence of such $\phi$ we refer the reader to \cite[Lemma 2.2]{JKL}. Also, see \cite[Lemma 2.1]{CKLS}.} Let us define
\begin{gather*}
A_j(\xi):=\mathfrak R(\xi)\, \wt\phi(2^{-j} (|\xi|^s-1)),	 \quad	B_j(\xi):=\Big(\mathfrak R(\xi)-\frac{\wt\chi (|\xi|)}{|\xi|^s-1} \Big) \wt\phi(2^{-j}(|\xi|^s -1)), \\ 
	C_j(\xi):=\frac{\wt\chi (|\xi|)}{|\xi|^s-1}\wt\phi (2^{-j}(|\xi|^s -1))
\end{gather*}
for each $j\in\Z$, and break the multiplier into
\[	\mathfrak R(\xi)= \sum_{2^j<|\de|} A_j(\xi) +\sum_{2^j\ge |\de|} B_j(\xi) + \sum_{2^j\ge |\de|} C_j(\xi).	\]
Again, by using the spherical coordinate, Minkowski's inequality, and Theorem \ref{rest_ext_sphere}, we see that
\begin{align*}
\Big\| \F^{-1} \Big( \sum_{2^j<|\de|} A_j(\xi) \wh f(\xi)\Big)\Big\|_{q_0,\infty} 
	& \lesssim \|f\|_{p_0,1}  \sum_{2^j<|\de|} \int_{1-2\de_\circ}^{1+2\de_\circ} \frac{|\rho^s-1||\wt\phi (2^{-j}(\rho^s-1))|}{(\rho^s-1)^2+\de^2} d\rho \\
	& \lesssim \|f\|_{p_0,1} \int_{1-2\de_\circ}^{1+2\de_\circ} \frac{|\de|}{(\rho^s-1)^2+\de^2} d\rho \lesssim  \|f\|_{p_0,1}
\end{align*}
since $\sum_{2^j<|\de|} |t \wt\phi(2^{-j}t)|\lesssim \sum_{2^j<|\de|} 2^j \lesssim |\de|$. Similarly we have
\begin{align*}
\Big\| \F^{-1} \Big( \sum_{2^j\ge |\de|} B_j(\xi) \wh f(\xi)\Big)\Big\|_{q_0,\infty} 
	& \lesssim \|f\|_{p_0,1}  \sum_{2^j\ge |\de|} \int_{1-2\de_\circ}^{1+2\de_\circ} \frac{\de^2 2^{-j} |\phi (2^{-j}(\rho^s-1))|}{(\rho^s-1)^2+\de^2} d\rho \\
	& \lesssim \|f\|_{p_0,1} \int_{1-2\de_\circ}^{1+2\de_\circ} \frac{|\de|}{(\rho^s-1)^2+\de^2} d\rho \lesssim  \|f\|_{p_0,1}.
\end{align*}

To estimate the multiplier operator given by $C_j$ we need the following. 
\begin{lem}\label{local_special} 
Let {$s\ne 0$} and $\la>0$.  Suppose $\phi\in\mathcal S(\R)$ with $\supp\wh\phi \subset [-2,-1/2]\cup[1/2,2]$. Then, for $1\le p,q\le\infty$ satisfying $q\ge 2$ and $\frac 1q \ge \frac{d+1}{d-1}(1-\frac1p)$,
\begin{equation}\label{loc_sp}
	\left\| \F^{-1} \left(\phi \big(\la^{-1}(|\xi|^s-1) \big)\wt\chi(|\xi|)\wh f(\xi) \right) \right\|_{q}\lesssim \la^{\frac{d+1}{2}-\frac dq} \|f\|_p,
\end{equation}
where $\wt\chi\in C^\infty_0\big( (1-\de_\circ, 1+\de_\circ) \big)$ for some small $\de_\circ>0$. 
\end{lem}
Assuming this lemma for the moment let us continue.  Since $C_j(\xi)=\wt\chi(|\xi|) 2^{-j} \phi (2^{-j}(|\xi|^s-1))$, by Lemma \ref{local_special}  we have
\begin{equation} \label{rs_est}
\big\| C_j(D) f \big\|_\s
	=  2^{-j} \Big\| \F^{-1} \Big( \phi \big( 2^{-j} (|\xi|^s-1) \big) \wt\chi(|\xi|) \wh f(\xi)\Big) \Big\|_\s 
	\lesssim  2^{j(\frac{d-1}2-\frac d\s)} \|f\|_r
\end{equation}
for $2\le \s\le\infty$, $\frac1\s\ge\frac{d+1}{d-1}(1-\frac1r)$. Application of $(\mathrm I)$ in Lemma \ref{intpl} yields 
\[	\Big\| \F^{-1} \Big( \sum_{2^j\ge |\de|} C_j(\xi) \wh f(\xi)\Big)\Big\|_{q_0,\infty} \lesssim \|f\|_{p_0,1}.	\]
Therefore, the proof of \eqref{rest_weak_type} is completed.   
\end{proof}

\begin{proof}[Proof of \eqref{weak_type}]
We may follow  the same lines of argument as in the proof of  \eqref{rest_weak_type}  by replacing  the $L^{p_0,1}$--$L^{q_0,\infty}$ estimate for the restriction-extension operator with the $L^p$--$L^{q,\infty}$ estimate for the same operator with $(1/p,1/q)\in (B', E']$ in Theorem \ref{rest_ext_sphere}. The only difference occurs when we attempt to prove 
\[	\Big\| \F^{-1} \Big( \sum_{2^j\ge |\de|} C_j(\xi)  \wh f(\xi)\Big)\Big\|_{q,\infty} \lesssim \|f\|_{p}.	\]
However, this can be obtained again by \eqref{rs_est} and using the last statement $(\mathrm I\!\mathrm I\!\mathrm I)$ in Lemma \ref{intpl} since we can fix $p$ while $q$ is allowed to be chosen to satisfy the assumption in Lemma \ref{intpl}. This observation first appeared in Bak \cite{bak}. Also, see \cite{CKLS}. 
\end{proof}

Now, we prove Lemma \ref{local_special}. 
\vspace{-10pt}
\begin{proof}[Proof of Lemma \ref{local_special}]
We may assume that $\la\le 1/100$. Otherwise, for every $M\ge 0$, the multiplier in \eqref{loc_sp} is smooth and uniformly bounded in $C^M_0(\R^d)$, hence \eqref{loc_sp} is trivial. By interpolation and Young's inequality, it is sufficient to show \eqref{loc_sp} for $(p,q)=(\frac{2(d+1)}{d+3},2)$, and for $(p,q)=(1,\infty)$. When $(p,q)=(\frac{2(d+1)}{d+3},2)$ using Plancherel's identity and the Stein--Tomas restriction theorem (\cite{St-beijing, Tom}) we have
\begin{align*}
\Big\| \F^{-1} \Big(\phi \big(\la^{-1}(|\xi|^s-1) \big)\wt\chi(|\xi|)\wh f(\xi) \Big) \Big\|_2
	&\approx\Big( \int \big|\phi \big(\la^{-1}(\rho^s-1) \big) \wt\chi(\rho)\big|^2 \Big| \int_{\mathbb S^{d-1}} \wh f(\rho\theta) d\s (\theta) \Big|^2 \rho^{d-1} d\rho \Big)^\frac12 \\
	&\lesssim \Big( \int_{1-\de_\circ}^{1+\de_\circ} \big|\phi \big(\la^{-1}(\rho^s-1) \big)\big|^2 d\rho\Big)^\frac12 \|f\|_\frac{2(d+1)}{d+3} \lesssim \la^\frac12  \|f\|_\frac{2(d+1)}{d+3}.
\end{align*}
Thus, it remains to show \eqref{loc_sp} when $(p,q)=(1,\infty)$. The related kernel is given by
\[	K(x)=(2\pi)^{-d-1}\iint \int_{\mathbb S^{d-1}} e^{i(\rho x\cdot\theta +r\la^{-1}(\rho^s-1))} \chi(\rho) \wh\phi(r) d\s(\theta) \, dr d\rho,	\] 
where $\chi(\rho):=\wt\chi(\rho)\rho^{d-1}$, and it suffices to show that
\begin{equation}\label{kernel_est}
|K(x)|\lesssim \la^{\frac{d+1}2}.
\end{equation}
We separately consider the three cases $|x|\le |s|\la^{-1}/100$, $|x|\ge 100|s|\la^{-1}$, and $|x|\approx \la^{-1}$. For the first case, since $\supp\wh\phi\subset [-2,-1/2]\cup[1/2,2]$, we have $|\frac d{d\rho} (\rho x\cdot\theta +r\la^{-1}(\rho^s-1) ) | \gtrsim \la^{-1}$. Hence integration by parts gives $|K(x)|\lesssim \la^{M}$ for any $M\ge0$. For the rest of cases, we recall $\int_{\mathbb S^{d-1}} e^{ix\cdot \theta} d\s (\theta)=c_d |x|^{-\frac{d-2}2} J_{\frac{d-2}{2}}(|x|)$ and use the  asymptotic expansion of the Bessel function $J_\nu$ (\cite{OM, St-book}). Thus we have 
\begin{align}	
K(x)	&= \sum_{\pm} \sum_{j=0}^M c_{j,\pm} |x|^{-\frac{d-1}2 -j}\iint e^{i(r\la^{-1}(\rho^s-1)\pm \rho|x|)} \chi_{j,\pm}(\rho) \wh \phi(r) d\rho dr +O (|x|^{-M-\frac{d+1}2} ) \label{kernel1}	\\
	&= (2\pi)^{-1}\sum_{\pm} \sum_{j=0}^M c_{j,\pm}|x|^{-\frac{d-1}2 -j}\int e^{\pm i\rho|x|}  \chi_{j,\pm}(\rho) \phi\big(\la^{-1}(\rho^s -1)\big) d\rho +O (|x|^{-M-\frac{d+1}2} ),	\label{kernel2}
\end{align}
for $M\ge d$ and $\chi_{j,\pm} \in C^\infty_0\big( (1-\de_\circ, 1+\de_\circ) \big)$. When $|x|\ge 100|s|\la^{-1}$ we use \eqref{kernel1}. Since $\wh\phi(r)\neq 0$ only if $|r|\approx 1$, we have $|\frac d{d\rho} (r\la^{-1}(\rho^s-1) \pm \rho|x|) | \gtrsim |x|$, so $|K(x)|\lesssim |x|^{-M}$ for any $M\ge0$ by integration by parts. When $|x|\approx \la^{-1}$ taking the absolute value of the integrands in \eqref{kernel2} we get $|K(x)|\lesssim \la^\frac{d+1}{2}$. Therefore, \eqref{kernel_est} follows.
\end{proof}

\begin{rem}\label{diff_Gu}
In \cite[pp. 16--22]{Gu} the estimate $\|m_0(D, 1+i\de)\|_{p\to q}\lesssim 1$ was obtained by decomposing the kernel $K_3(x):=\F^{-1}\big(m_0(\, \cdot \, , 1+i\de)\big)(x)$ and using oscillatory integral estimate. Instead we work in the Fourier transform side by decomposing the multiplier $m_0(\xi, 1+i\de)$ dyadically away from the unit sphere $|\xi|=1$ carrying singularity.
\end{rem}

\subsection{Proof of  Proposition \ref{main_prop} when $(1/p, 1/q) \in  \wt{\mathcal R}_2   \cup \big(\mathcal Q\setminus [P_\ast, P_\circ, D]\big) \cup \big(\mathcal Q\setminus [P_\ast, P_\circ, D]\big)'$} \label{pf_of_main_prop}
Since we already have  the estimates \eqref{rest_weak_type}, \eqref{weak_type},  and \eqref{main_part} with $p=q=2$, in view of interpolation and duality, it is sufficient to show \eqref{main_part} for $p,q$ satisfying $ (1/p,1/q)\in (P_\circ, B) \cup [(0,0), P_\ast)$. 

For the purpose we may assume $|\delta|$ is small enough. Thus,  by  finite decomposition, rotation, and discarding harmless smooth part of the multiplier,  we may assume that the multiplier is supported near $(0,\cdots,0,-1)\in\R^d$.  We write $\xi=(\eta, \tau)\in\R^{d-1}\times \R $ and  
\[	|\xi|^s-1= \frac{(\tau+\sqrt{1-|\eta|^2}) (\tau-\sqrt{1-|\eta|^2}) ((\ta^2+|\eta|^2)^\frac s2 -1)}{\ta^2+|\eta|^2-1} .\]
Let us set 
\[\psi(\eta)=1-\sqrt{1-|\eta|^2},  \quad  m(\eta, \tau)=\frac{1}{s} \frac{(\tau+\psi(\eta)-2)(((\ta-1)^2+|\eta|^2)^\frac s2 -1)}{(\ta-1)^2+|\eta|^2-1}.  \] 
It is easy to see that  $m(\eta, \tau)=-1+O(|\tau|+|\eta|^2)$ for $|\tau|, |\eta|\ll 1$. In particular, for the case $s=2$ (corresponding to the Laplacian resolvent), the function $m(\eta, \tau)$ takes simpler form and it is easy to see that there is no singularity. After change of variables $\tau\to \tau-1$,  we may further assume that the multiplier is of the form
\[	\mf M_\delta (\eta, \tau) = \frac {\chi_0(\eta, \tau)}{s( \tau-\psi (\eta) )m(\eta,\ta)-{i\de}},	\]
where $\chi_0$ is a smooth function supported on a  small neighborhood of the origin in $\R^d$.  By further harmless affine transformations {(see \eqref{affine} and Lemma \ref{approximation})}, we may assume that $\psi \in{\bf Ell}(N, \epsilon)$ for a large $N\ge 10d$ and a small $\epsilon>0$, and $m\in {\bf Mul}(N, b)$ for some $b>0$, so that both Proposition \ref{dethick} and Proposition \ref{cstomul} are valid.  Thus  $\mf M_\delta$ takes the form
\[	\mf M_\delta(\eta, \tau) = \frac1{|\de|} \varphi\Big( \frac{m(\eta,\ta)(\ta-\psi(\eta))}{|\de|} \Big)\chi_0(\eta, \tau)	\]
for $\varphi (t) = (st\pm i)^{-1}$, which clearly satisfies the condition \eqref{diff}.  We break ${\mf M_\delta}$ as follows:
\begin{align}\label{break_dyadic}
\mf M_\delta(\eta, \ta) &= \frac1{|\de|} \varphi\Big( \frac{m(\eta,\ta)(\ta-\psi(\eta))}{|\de|} \Big) \be_0 \Big( \frac{m(\eta,\ta)(\ta-\psi(\eta))}{|\de|}\Big) \chi_0(\eta, \tau)  \\
\nonumber&\qquad + \frac1{|\de|}\sum_{j=1}^{\log\frac1{|\de|}} \varphi\Big( \frac{m(\eta,\ta)(\ta-\psi(\eta))}{|\de|} \Big) \be \Big( \frac{m(\eta,\ta)(\ta-\psi(\eta))}{2^{j-1}|\de|}\Big)\chi_0(\eta, \tau).
\end{align}

Since $(1/p,1/q)\in (P_\circ, B) \cup [(0,0), P_\ast)$, we note that $\gamma_{p,q}=\frac{d+1}2-\frac dp$, and that $(p,q)$ are the pairs given by $\frac{2d}{d+1}< p <p_\circ$ and $\frac 1q = \frac{d-1}{d+1}(1-\frac 1p)$, or $p_\ast<p=q\le \infty$. We apply Proposition \ref{dethick} with $\de$ and $\la$ replaced with $|\de|$ and $2^{j-1}$, respectively, to each of the multiplier operators  which are given by the functions on the right hand side of  \eqref{break_dyadic}. This yields, for $2\le p<p_\circ$ and $\frac 1q=\frac{d-1}{d+1}(1-\frac 1p)$, 
\[	\|{\mf M_\delta}(D)\|_{p\to q} 
	\lesssim \frac1{|\de|} \Big( |\de|^{\frac{1-d}2+\frac dp} + \sum_j  2^{-j} (2^j |\de|)^{\frac{1-d}2+\frac dp}\Big) 
	\lesssim |\de|^{\frac dp -\frac{d+1}2}. \]
Now interpolation between these  estimate and  \eqref{rest_weak_type} gives the desired estimate \eqref{main_part} for  $p,q$ satisfying $ (1/p,1/q)\in (P_\circ, B)$.   The remaining cases $(1/p,1/q)\in [(0,0), P_\ast)$ can be handled similarly by making use of Proposition \ref{cstomul}. Repeating the same argument, we get $\|{\mf M_\delta}(D)\|_{p\to p} \lesssim |\de|^{\frac dp -\frac{d+1}2}$ for  $p_\ast<p\le \infty$.  \qed

\subsection{Description of $\mathcal Z_{p,q}(\ell)$ }\label{drawing_figures}
The case $(1/p,1/q)\in\wt{\mathcal R}_3'$ can be deduced from the case $(1/p, 1/q)\in\wt{\mathcal R}_3$ by duality, hence we may consider the case  $(1/p,1/q)\in \mathcal R_1\cup\big(\bigcup_{i=2}^3\wt{\mathcal R}_i\big)$ only.  For $d\ge2$ and $(1/p,1/q)\in \mathcal R_1\cup\big(\bigcup_{i=2}^3\wt{\mathcal R}_i\big)$, we  set 
\[	\omega_{p,q}=\omega_{p,q}(d):=1-\frac{d}{2}\Big(\frac1p-\frac1q\Big),	\] 
which lies in $[0,1]$. Since ${\mathlarger \kappa}_{p,q}(z) = |z|^{-\omega_{p,q}} (\dist(z,[0,\infty))/|z|)^{-\gamma_{p,q}}$,  we get
\begin{align}\label{zpq}
	\mathcal Z_{p,q}(\ell)	&=\big\{z\in\C\setminus\{0\}: \re z\le 0,~ \ell |z|^{\omega_{p,q}}\ge 1\big\} \\
\nonumber			&\qquad \cup \big\{z\in\C\setminus[0,\infty): \re z>0, ~ \ell |\im z|^{\gamma_{p,q}} \ge | z |^{\gamma_{p,q}-\omega_{p,q}} \big\}
\end{align} 
for $\ell>0$ and $(1/p,1/q)\in \mathcal R_1\cup\big(\bigcup_{i=2}^3\wt{\mathcal R}_i\big)$. The shape of $\mathcal Z_{p,q}(\ell)$ is mainly determined by the value of $\gamma_{p,q}$  and  $\gamma_{p,q}-\omega_{p,q}$.\footnote{If $\gamma_{p,q}>0$ and  $\re z\ge |\im z|\gg 1$, then the region is roughly determined by $|\im z|\gtrsim (\re z)^{\frac{\gamma_{p,q}-\omega_{p,q}}{\gamma_{p,q}}}$. Likewise, if $\gamma_{p,q}>0$ and  $0<\re z< |\im z|$,  $|\im z|\gtrsim \ell^{-1/\omega_{p,q}} $ for   $z\in \mathcal Z_{p,q}(\ell)$.} When $\omega_{p,q}>0$, the value $\ell$ does not have particular role in determining the overall shape of $\mathcal Z_{p,q}(\ell)$. However, if $\omega_{p,q}=0$  the profile of $\mathcal Z_{p,q}(\ell)$ depends not only on $p,q,d,$ but also on $\ell$.  In what follows we handle these two cases separately. 

\begin{figure}
\captionsetup{type=figure,font=footnotesize}
\centering
	\begin{subfigure}[b]{0.3\textwidth}
		\centering
		\includegraphics[width=\textwidth]{complement_circular.pdf}
		\caption{\scriptsize $\frac{d+3}{2(d+1)}\le\frac1p<\frac{d+2}{2d}$.}
		\label{complement_circular}
	\end{subfigure}
	\hfill
	\begin{subfigure}[b]{0.3\textwidth}
		\centering
		\includegraphics[width=\textwidth]{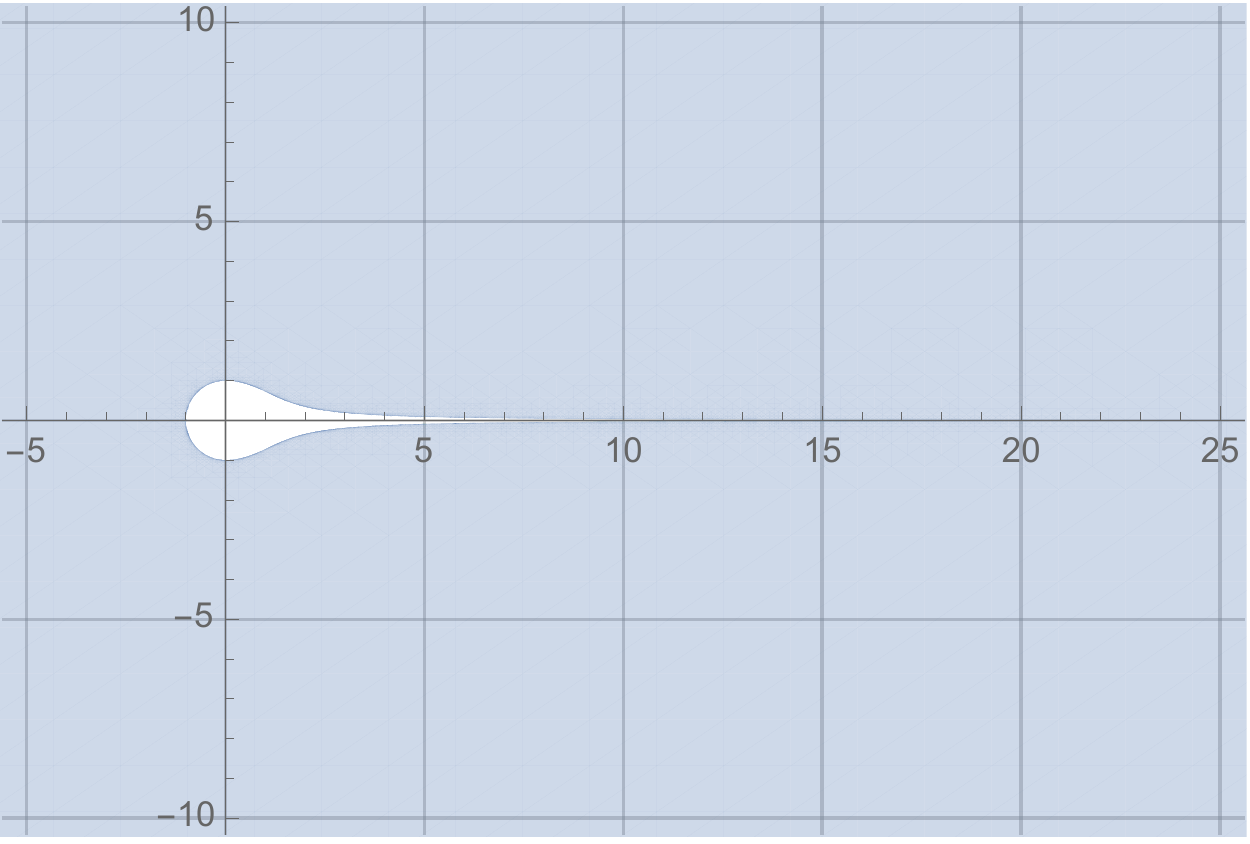}
		\caption{\scriptsize $\frac12<\frac1p<\frac{d+3}{2(d+1)}$.}
		\label{dual_p40div26}
	\end{subfigure}		
	\hfill
	\begin{subfigure}[b]{0.3\textwidth}
		\centering
		\includegraphics[width=\textwidth]{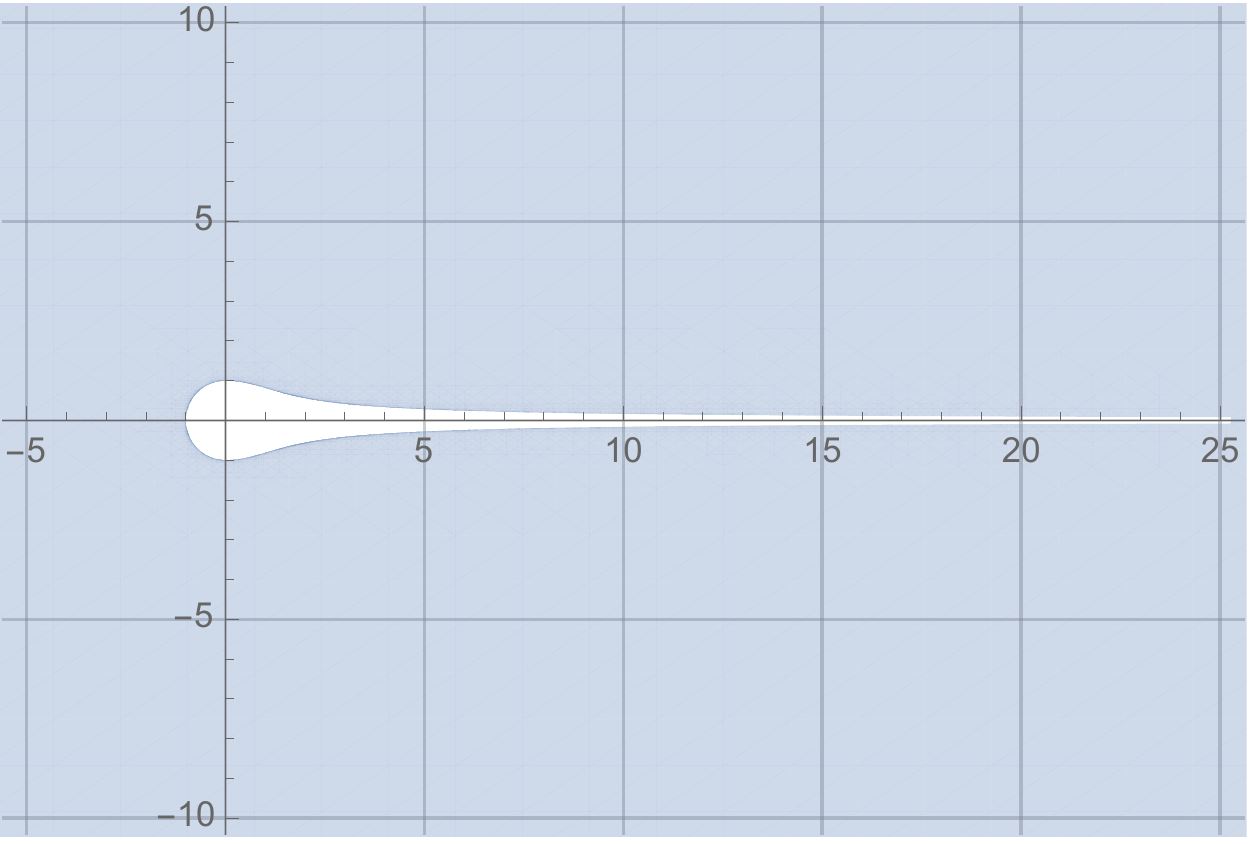}
		\caption{\scriptsize $\frac12<\frac1p<\frac{d+3}{2(d+1)}$.}
		\label{dual_p80div51}
	\end{subfigure}		
	
	\bigskip 
	
	\begin{subfigure}[b]{0.3\textwidth}
		\centering
		\includegraphics[width=\textwidth]{dual_p40div25.pdf}
		\caption{\scriptsize $\frac12<\frac1p<\frac{d+3}{2(d+1)}$.}
		\label{dual_p40div25}
	\end{subfigure}
	\hfill
	\begin{subfigure}[b]{0.3\textwidth}
		\centering
		\includegraphics[width=\textwidth]{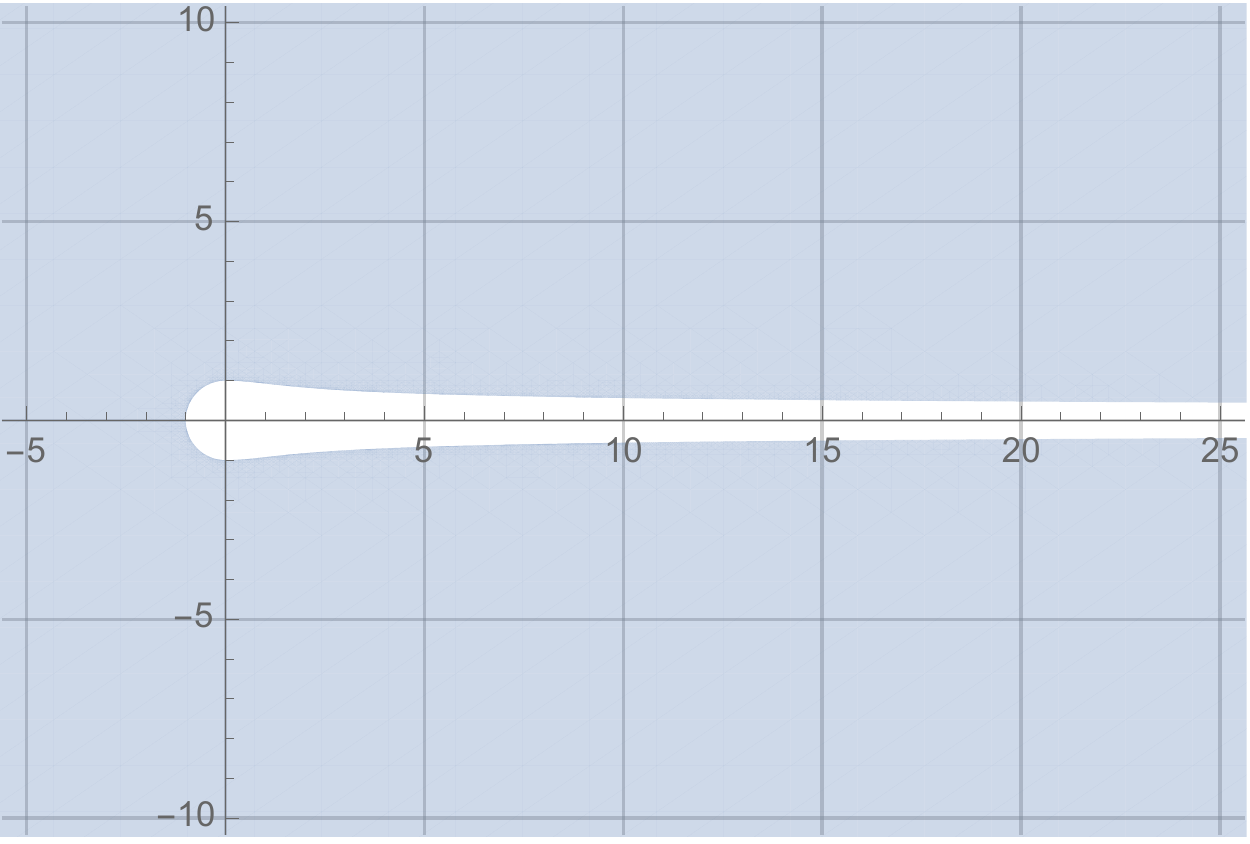}
		\caption{\scriptsize $\frac12<\frac1p<\frac{d+3}{2(d+1)}$.}
		\label{dual_p40div24}
	\end{subfigure}
	\hfill
	\begin{subfigure}[b]{0.3\textwidth}
		\centering
		\includegraphics[width=\textwidth]{broad10.pdf}
		\caption{\scriptsize $(\frac1p,\frac1q)=(\frac12,\frac12)$.}
		\label{broad10(uniform)}
	\end{subfigure}		
\caption{$\mathcal Z_{p,q}(1)$ ({\bf(a)} $\to$ {\bf(b)} $\to$ {\bf(c)} $\to$ {\bf(d)} $\to$ {\bf(e)} $\to$ {\bf(f)}) as $1/p$ decreases while $(\frac1p, \frac1q)\in \big[H, (\frac{d+2}{2d},\frac{d-2}{2d}) \big)$.} \label{fig_line_of_duality}
\end{figure}

\subsubsection*{The case $\omega_{p,q}>0$}  We further subdivide  this case into  the cases $(1/p,1/q)\in {\mathcal R}_1$, $(1/p, 1/q)\in \wt{\mathcal R}_2$, and  $(1/p, 1/q)\in \wt{\mathcal R}_3$. 

\vspace{-8pt}

\begin{itemize}
	\item $(1/p,1/q)\in {\mathcal R}_1(2)$, or $(1/p,1/q)\in {\mathcal R}_1(d)\setminus(A,A')$ if $d\ge3$: Then $\omega_{p,q}\in(0,\frac1{d+1}]$ and $\ga_{p,q}=0$. 	Hence $\mathcal Z_{p,q}(\ell) = \{z\in\C\setminus[0,\infty): |z| \ge \ell^{-1/\omega_{p,q}} \}$.  See  Figure \ref{complement_circular}. 
\smallskip
	\item $(1/p, 1/q)\in \wt{\mathcal R}_2$: Then $\omega_{p,q} \in ( \frac1{d+1},1]$, $\gamma_{p,q}=1-\frac{d+1}2(\frac1p-\frac1q)\in(0,1]$, and $\gamma_{p,q}-\omega_{p,q}= -\frac12(\frac1p-\frac1q)\le 0$.  If $(1/p, 1/q)\in \wt{\mathcal R}_2\setminus \{H\}$, 	since $\gamma_{p,q}-\omega_{p,q}<0$,  $\mathcal Z_{p,q}(\ell)$ is the complement of a neighborhood of $[0,\infty)$ which shrinks along the positive real line as $\re z\to \infty$. See Figure \ref{dual_p40div26}, Figure \ref{dual_p80div51}, Figure 	\ref{dual_p40div25} and Figure \ref{dual_p40div24}.   Also, {$\mathcal Z_{2,2}(\ell)=\{z: \re z\le 0,\, |z| \ge 1/\ell\}\cup \{z: \re z>0, \, |\im z|\ge 1/\ell\}$} (see Figure \ref{broad10(uniform)}). 
\smallskip
	\item $(1/p, 1/q)\in \wt{\mathcal R}_3$: In this case, $\omega_{p,q}\in(0, 1]$, $\gamma_{p,q}=\frac{d+1}2-\frac dp>0$ and $\gamma_{p,q}-\omega_{p,q}=\frac d2(\frac{d-1}d-(\frac1p+\frac1q))$. So, we divide $\wt{\mathcal R}_3$ into the three sets 	$\wt{\mathcal R}_{3,+}$, $\wt{\mathcal R}_{3,0}$, and $\wt{\mathcal R}_{3,-}$.\footnote{We recall from the introduction that $\wt{\mathcal R}_{3,\pm}=\{(\frac1p,\frac1q)\colon \pm(\gamma_{p,q}-\omega_{p,q})<0\}$ and $\wt{\mathcal R}_{3,0}=\{(\frac1p,\frac1q)\colon \gamma_{p,q}-\omega_{p,q}=0\}$.} 
	\begin{itemize}
		\item[$\dagger$] 
		$(1/p, 1/q)\in\wt{\mathcal R}_{3,+}$: Since $\gamma_{p,q}-\omega_{p,q}<0$,  $\mathcal Z_{p,q}(\ell)$ is the complement of a neighborhood of $[0,\infty)$ which shrinks along positive real line as $\re z\to \infty$. See Figure \ref{p2_q3}. 
	\item[$\dagger$]  
		$(1/p, 1/q)\in\wt{\mathcal R}_{3,0}$: Since $\gamma_{p,q}-\omega_{p,q}=0$,  $\mathcal Z_{p,q}(\ell)$ is the complement of the $\ell^{-1/\ga_{p,q}}$-neighborhood of $[0,\infty)$. See Figure \ref{p2_q10div3}. 
	\item[$\dagger$] 
		$(1/p, 1/q)\in\wt{\mathcal R}_{3,-}$:  In this case $\gamma_{p,q}-\omega_{p,q}>0$. 
		Hence $\mathcal Z_{p,q}(\ell)$ is the complement of a neighborhood of $[0,\infty)$ whose boundary asymptotically satisfies $|\im z|\approx (\re z)^{1-\omega_{p,q}/\gamma_{p,q}}$ when $\re z$ is large.  See Figure \ref{p2_q9pt9}, Figure \ref{p2_q20div3} and Figure \ref{p2_q5}. 
	\end{itemize}
\end{itemize}

\vspace{-8pt}
\begin{figure}
\captionsetup{type=figure,font=footnotesize}
\centering
	\begin{subfigure}[b]{0.3\textwidth}
		\centering
		\includegraphics[width=\textwidth]{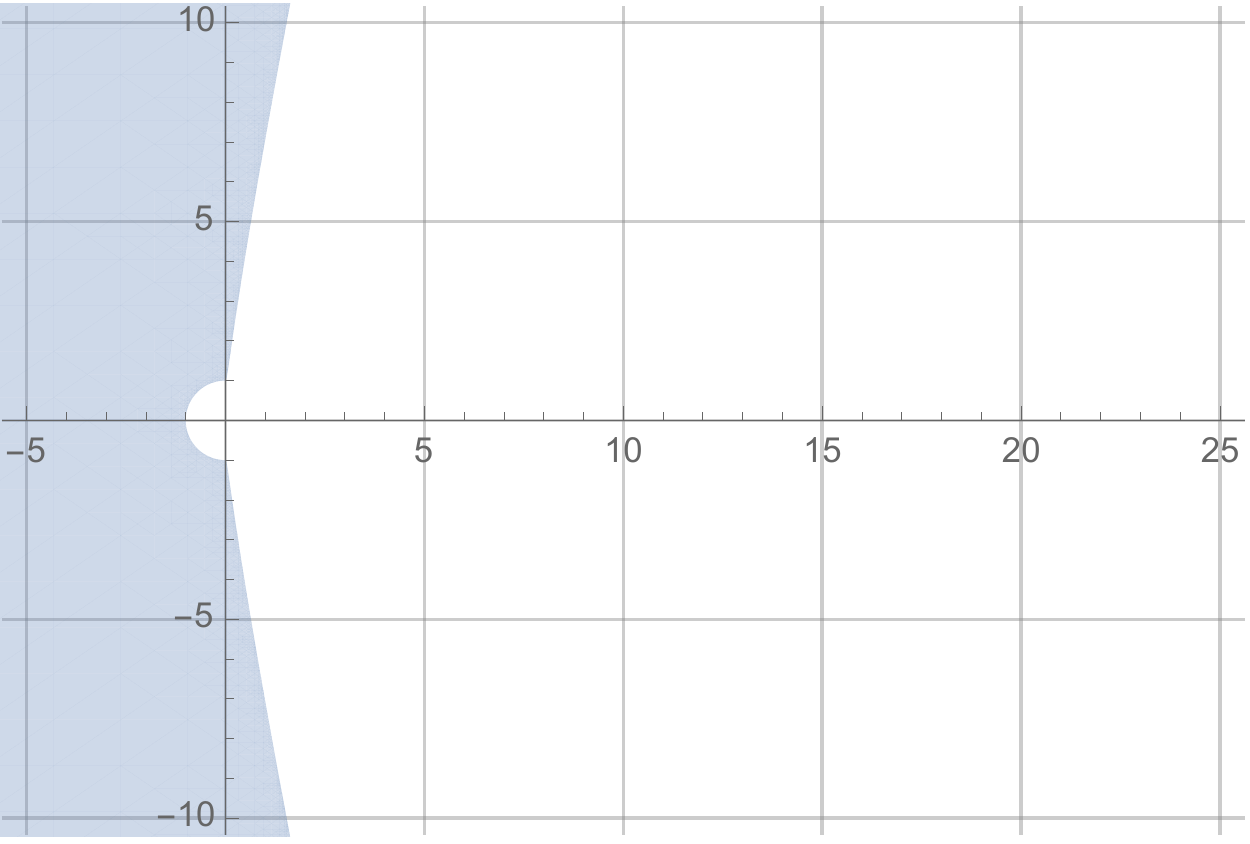}
		\caption{\scriptsize $0<\frac1q-\min\{0,\frac{d-4}{2d}\}\ll 1$.}
		\label{p2_q9pt9}
	\end{subfigure}
	\hfill 
	\begin{subfigure}[b]{0.3\textwidth}
		\centering
		\includegraphics[width=\textwidth]{broad6.pdf}
		\caption{\scriptsize $\min\{0,\frac{d-4}{2d}\}<\frac1q<\frac{d-1}d-\frac12$.}
		\label{p2_q20div3}
	\end{subfigure}		
	\hfill 	
	\begin{subfigure}[b]{0.3\textwidth}
		\centering
		\includegraphics[width=\textwidth]{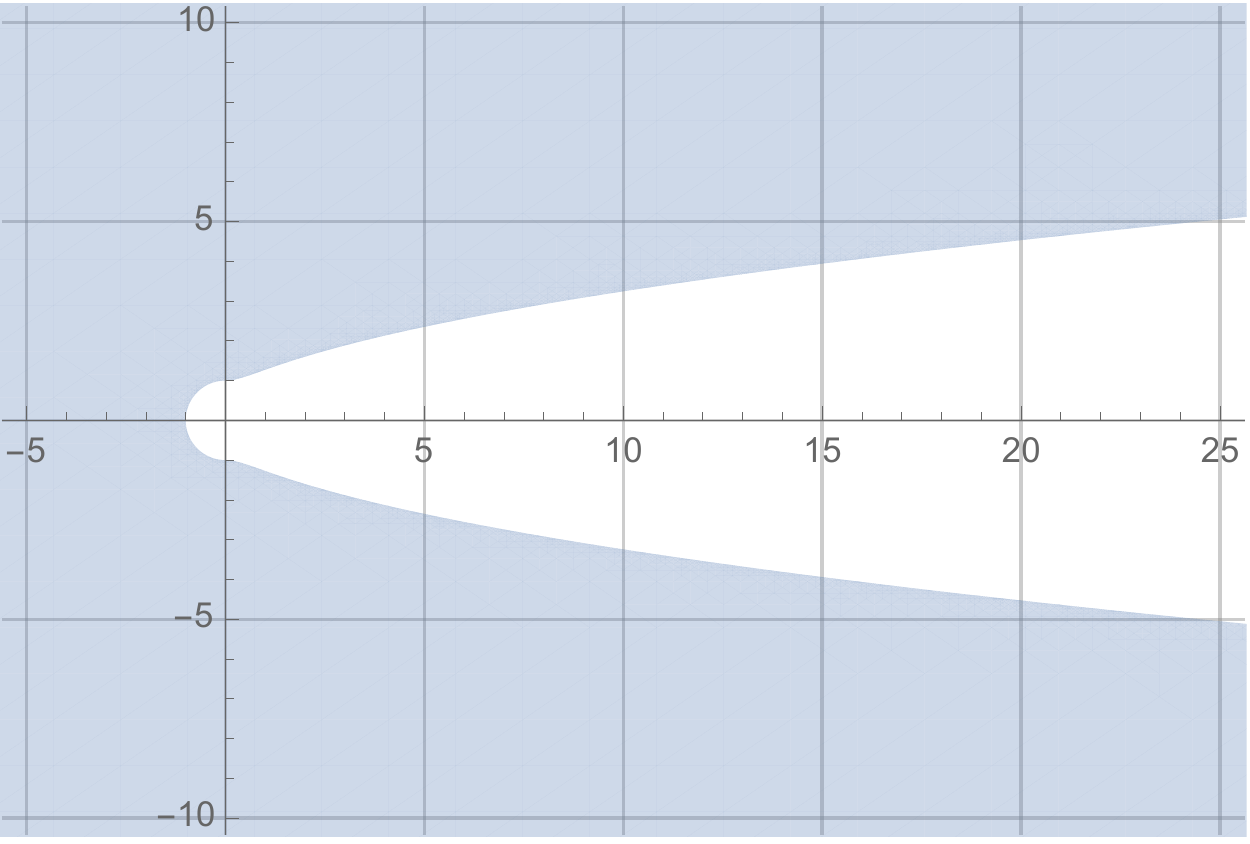}
		\caption{\scriptsize $\min\{0,\frac{d-4}{2d}\}<\frac1q<\frac{d-1}d-\frac12$.}
		\label{p2_q5}
	\end{subfigure}
	
	\bigskip
	
	\begin{subfigure}[b]{0.3\textwidth}
		\centering
		\includegraphics[width=\textwidth]{broad10.pdf}
		\caption{\scriptsize $(\frac1p,\frac1q)=(\frac12,\frac{d-1}d-\frac12)$.}
		\label{p2_q10div3}
	\end{subfigure}		
	\hfill  
	\begin{subfigure}[b]{0.3\textwidth}
		\centering
		\includegraphics[width=\textwidth]{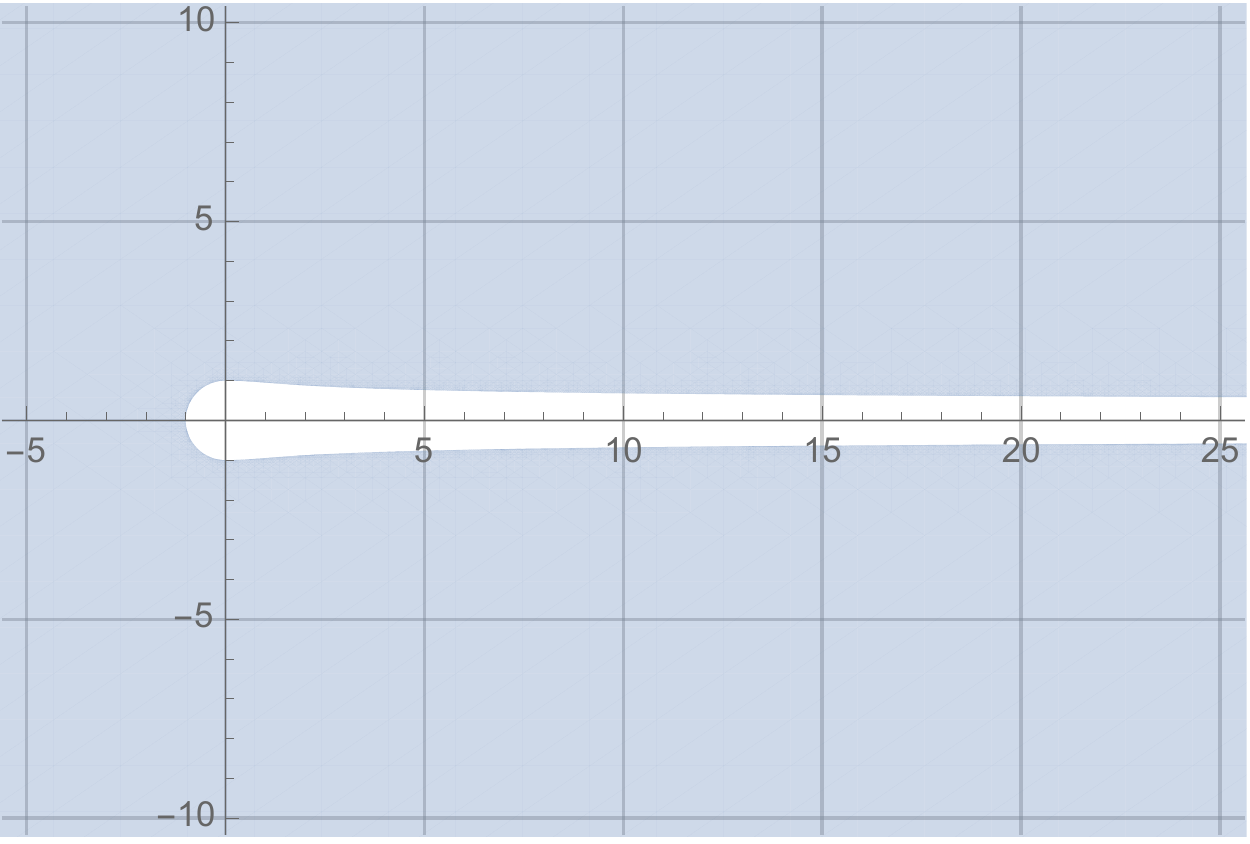}
		\caption{\scriptsize $\frac{d-1}d-\frac12<\frac1q<\frac12$.}
		\label{p2_q3}
	\end{subfigure}
	\hfill 
	\begin{subfigure}[b]{0.3\textwidth}
		\centering
		\includegraphics[width=\textwidth]{broad10.pdf}
		\caption{\scriptsize $(\frac1p,\frac1q)=(\frac12,\frac12)$.}
		\label{p2_q2}
	\end{subfigure}		
\caption{$\mathcal Z_{2,q}(1)$ ({\bf(a)} $\to$ {\bf(b)} $\to$ {\bf(c)} $\to$ {\bf(d)} $\to$ {\bf(e)} $\to$ {\bf(f)})  as $\frac1q\in \big(\min\{0,\frac{d-4}{2d}\}, \frac12 \big]$ increases.}\label{fig_shapes_p2}
\end{figure}

\subsubsection*{The case $\omega_{p,q}=0$}  
In this case the shape of $\mathcal Z_{p,q}(\ell)$  depends on the value of $\ell$ as well. 
\begin{itemize}
\vspace{-8pt}
	\item
	Let $(1/p, 1/q)\in (A,A')$.  Since $\omega_{p,q}=\ga_{p,q}=0$, we have $\mathcal Z_{p,q}(\ell)=\emptyset$ if $\ell<1$, and $\mathcal Z_{p,q}(\ell)=\C\setminus[0,\infty)$ if $\ell\ge1$. 
\smallskip
	\item
	Let $d\ge4$ and  $(\frac1p, \frac1q)\in((\frac 2d, 0),A)$. In this case $\omega_{p,q}=0$ and $\ga_{p,q}=\frac{d+1}2-\frac{d}{p}\in(0,\frac{d-3}2)$.
	\begin{itemize}
		\item[$\dagger$] 
		If $\ell<1$, $\mathcal Z_{p,q}(\ell)=\emptyset$. 
		\item[$\dagger$] 
		If $\ell=1$,  there is  a rigid dichotomy of $\mathcal Z_{p,q}(1)$ between the case of uniform bound and the other case; $\mathcal Z_{p,q}(1)=\C\setminus[0,\infty)$ if $(1/p,1/q)\in(A,A')$, but $\mathcal Z_{p,q}(1)=\{z: \re z\le 0\}\setminus\{0\}$ if $(1/p,1/q)\notin[A,A']$.
		\item[$\dagger$]
		When $\ell>1$, there is also a kind of  dichotomy  although it is not so rigid as in the former case with $\ell=1$. Indeed, if $\ell>1$ and $(1/p,1/q)\in(A,A')$, then $\mathcal Z_{p,q}(\ell)=\C\setminus[0,\infty)$. Otherwise, $\mathcal Z_{p,q}(\ell)$ is the complement (in $\C$) of a (planar) cone of which axis is the positive real line $[0,\infty)$, and apex is the origin.  It is interesting to note that, as $(1/p,1/q)$ moves from (near) $A$ to (near) $(2/d,0)$  along the line $1/p-1/q=2/d$, the apex angle gets larger from $0$ to $2\arctan\big(1/\sqrt{\ell^\frac4{d-3} -1} \big)$. See Figure \ref{fig_Sobolev_line2}. 
	\end{itemize}
\end{itemize}

\begin{figure}
\captionsetup{type=figure,font=footnotesize}
\centering
	\begin{subfigure}[b]{0.3\textwidth}
		\includegraphics[width=\textwidth]{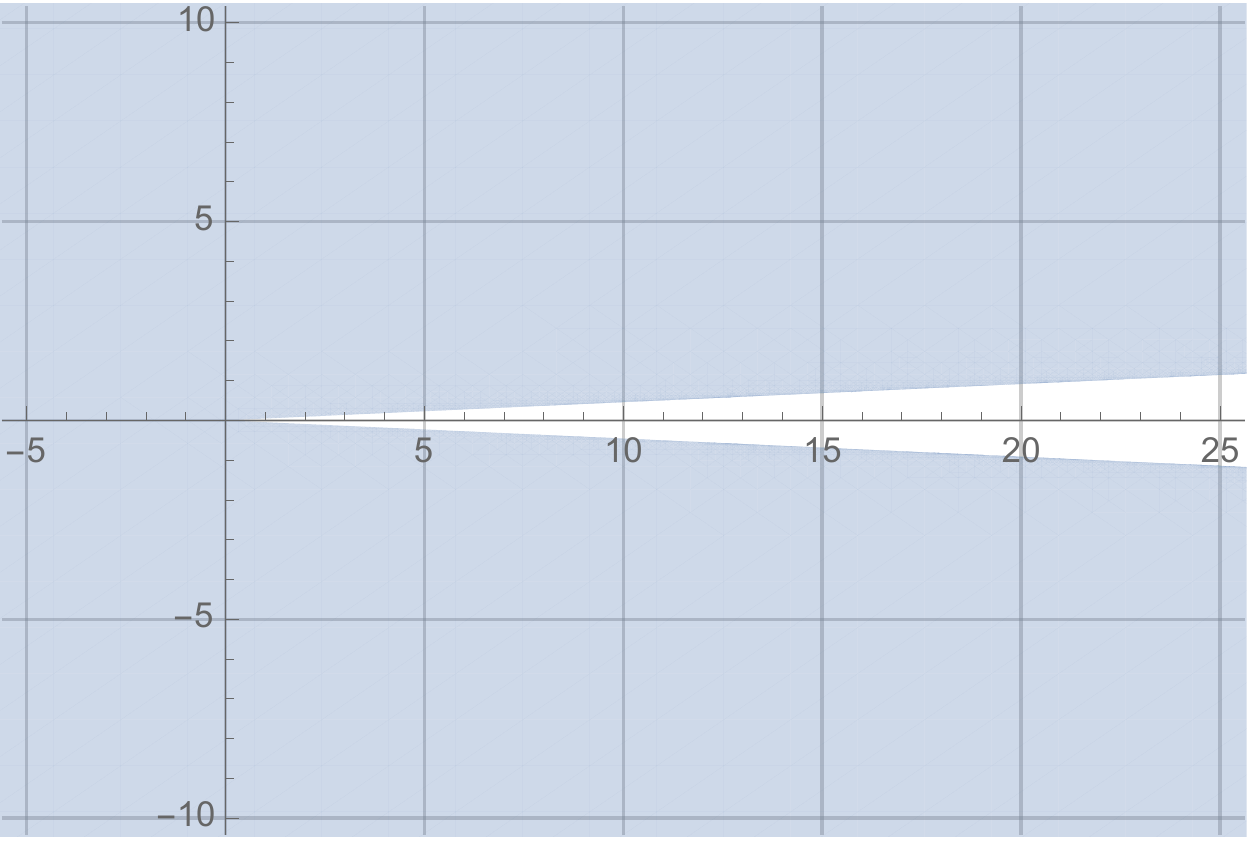}
		\caption{\scriptsize $0<\frac{d+1}{2d}-\frac1p\ll 1$.}
		\label{l>1p1_684}
	\end{subfigure}
	\hfill 
	\begin{subfigure}[b]{0.3\textwidth}
		\includegraphics[width=\textwidth]{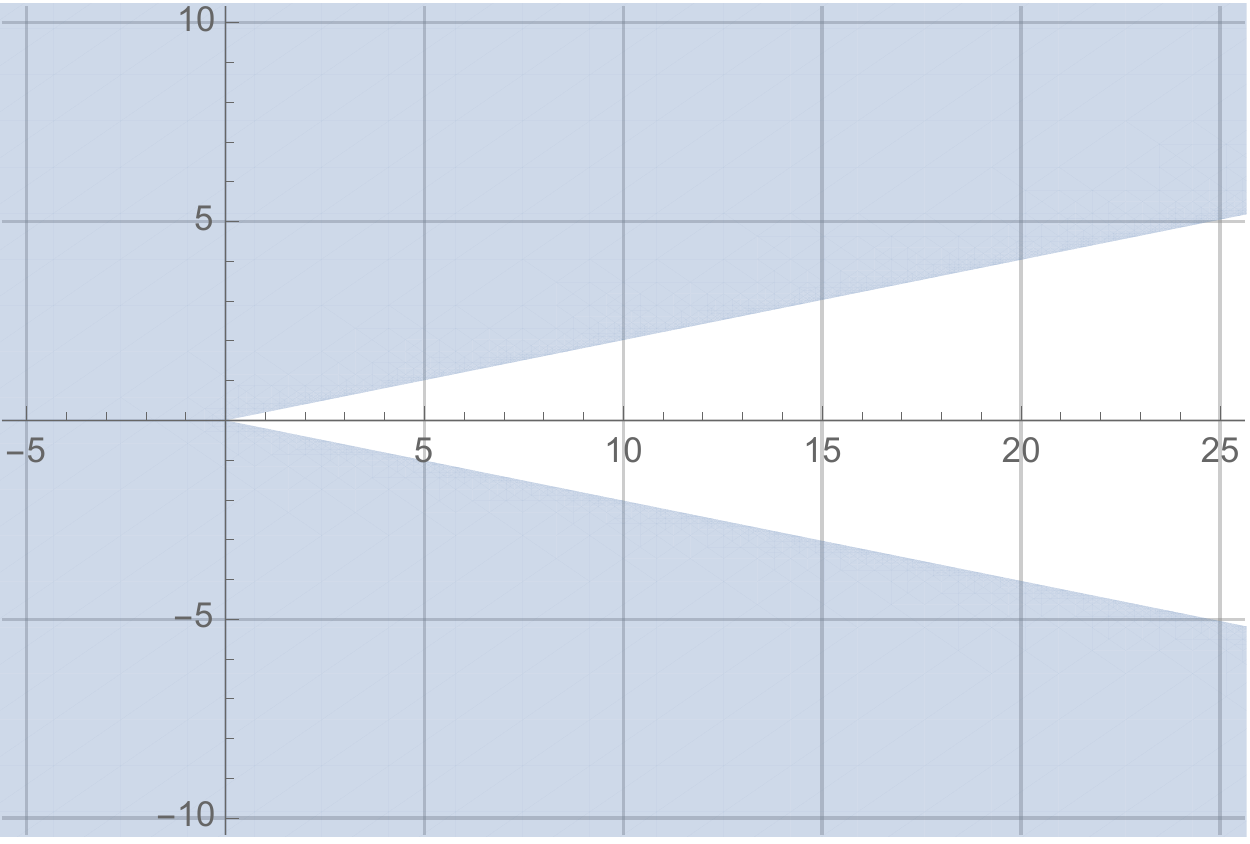}
		\caption{\scriptsize $\frac2d<\frac1p<\frac{d+1}{2d}$.}
		\label{l>1p1_7}
	\end{subfigure}	 
	\hfill 
	\begin{subfigure}[b]{0.3\textwidth}
		\includegraphics[width=\textwidth]{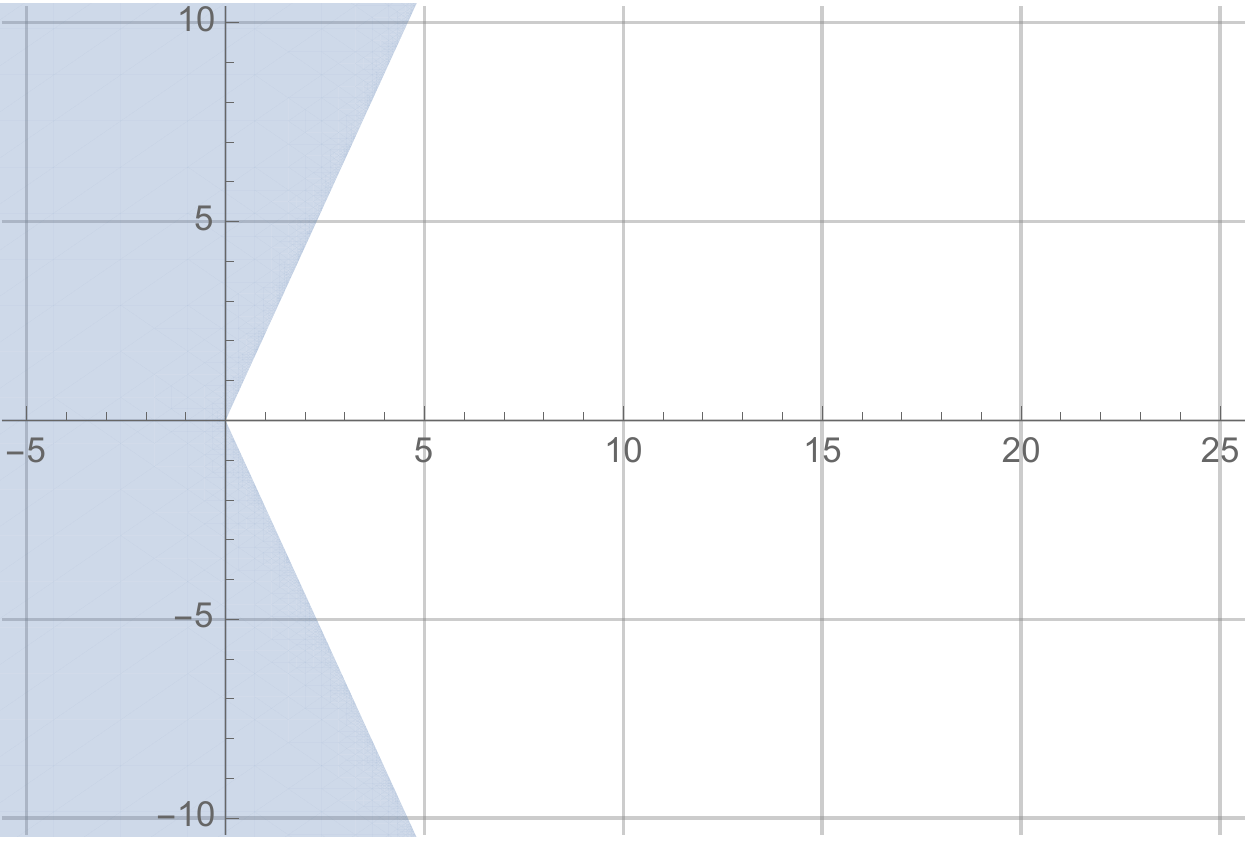}
		\caption{\scriptsize $0<\frac1p-\frac2d\ll1$.}
		\label{l>1p2_499999}
	\end{subfigure}	
\caption{$\mathcal Z_{p,q}(\ell)$ with $\ell>1$ ({\bf(a)} $\to$ {\bf(b)} $\to$ {\bf(c)}) as $\frac1p$ decreases in the interval $(\frac2d, \frac{d+1}{2d})$ along the line $\frac1p-\frac1q=\frac2d$ and $d\ge4$.}\label{fig_Sobolev_line2}
\end{figure}

\section{Lower bounds for $\|(-\Delta-z)^{-1}\|_{p\to q}$: Proposition \ref{shp}}\label{shp1}
In this section we obtain lower bounds for $\|(-\Delta-z)^{-1}\|_{p\to q}$, which prove  Proposition \ref{shp}.  Before doing this we provide proof of Proposition \ref{eresb} which is simpler. 

\subsection{Proof of Proposition \ref{eresb}}  
We first show the sufficiency part. By \eqref{mult} it is sufficient to show the estimate  
\begin{equation}\label{resolz1}
	\Big\| \F^{-1} \Big( \frac{\wh f(\xi)}{|\xi|^2-z}  \Big) \Big\|_q \le C_z\|f\|_p
\end{equation}
holds whenever $(1/p, 1/q)\in \mathcal R_0$.  Since $z\neq 0$, thanks to the scaling property \eqref{sc} we may assume that $z\in\mathbb S^1\setminus \{1\}$. Let us break the multiplier 
\[	(|\xi|^2-z)^{-1}=\beta_0(|\xi|)(|\xi|^2-z)^{-1}+\sum_{j\ge 0}\be(2^{-j}|\xi|)(|\xi|^2-z)^{-1}. 	\]
It is clear that $\beta_0(|\xi|)(|\xi|^2-z)^{-1}$ is smooth and compactly supported in the open ball $B_d(0,3/4)$. Hence
\begin{equation}\label{beta0}
	\Big\| \mathcal F^{-1} \Big({\beta_0(|\xi|)}({|\xi|^2-z})^{-1} \widehat f (\xi) \Big)\Big\|_q\le C\|f\|_p 
\end{equation}
for $p,q$ satisfying $1\le p\le q\le \infty$. By scaling it is easy to see $\big\| \mathcal F^{-1} \big( \frac{\be(2^{-j}|\xi|)}{|\xi|^2-z} \big) \big\|_{r}\lesssim 2^{(d-2-\frac dr)j} $ for $1\le r\le \infty$. Thus, Young's inequality gives 
\[	\Big\| \F^{-1}\Big(\frac{{\be(2^{-j}|\xi|)}\wh{f}(\xi)}{|\xi|^2-z} \Big) \Big\|_q \lesssim 2^{j(\frac dp-\frac dq-2)} \|f\|_p	\]
whenever $1\le p\le q\le \infty$. Thus summation along $j$  and combining the resulting estimate with \eqref{beta0} give all desired estimates \eqref{resolz1} except the case $1/p-1/q=2/d$.  In order to obtain the estimate \eqref{resolz1} for  $p,q$ with $1/p-1/q=2/d$  in the case  $d\ge 3$, we use $(\mathrm I\!\mathrm I\!\mathrm I)$ in  Lemma \ref{intpl} with $k=1$ to get 
\[	\Big\|\sum_{j\ge 1}\mathcal F^{-1}\Big({\be(2^{-j}|\xi|)}({|\xi|^2-z})^{-1} \widehat f(\xi) \Big)\Big\|_{q,\infty} \lesssim \|f\|_{p}	\]
provided that $1\le p, q< \infty$ and $1/p-1/q=2/d$.  Real interpolation between those estimates together with \eqref{beta0} gives  the desired estimate \eqref{resolz1} for all $p,q$ satisfying $(1/p, 1/q)\in \mathcal R_0$. 

Now we consider the necessity part of Proposition \ref{eresb}.  {In the case $d\ge 3$} we need to show \eqref{resolz1} holds only if 
\begin{equation}\label{nec_con}
	\frac 1p- \frac 1q \ge 0, \quad  \frac 1p- \frac 1q \le \frac 2d, \quad  (p, q)\neq \Big(1,\, \frac d{d-2} \Big), \quad (p, q)\neq \Big(\frac d2,\, \infty\Big).  
\end{equation}
The first condition is obvious since the multiplier operator is translation invariant (\cite{Ho60}).  For the second condition we notice that, for large $j$, the estimate \eqref{resolz1} remains valid with $C$ independent of $j$ if $\widehat f(\xi)$ is replaced with $\beta(2^{-j}|\xi|) \widehat f (\xi)$. Then re-scaling the estimate gives 
\begin{equation}\label{resolz2}
	\Big\| \F^{-1} \Big( \frac{\beta(|\xi|)\wh f(\xi)}{|\xi|^2-2^{-2j}z} \Big) \Big\|_q\le C_z 2^{j(2+ \frac dq- \frac dp)}\|f\|_p 
\end{equation}
for $f\in\mathcal S(\R^d).$ Now fix a nonzero Schwartz function $f$ such that $\|\F^{-1}(|\xi|^{-2}\beta(|\xi|)\wh f(\xi))\|_q>0$.  If we take limit $j\to+\infty$ the left side of \eqref{resolz2} converges to $\|\F^{-1}(|\xi|^{-2}\beta(|\xi|)\wh f(\xi))\|_q$, but  the quantity $2^{j(2+d/q-d/p)}$ on the right side converges to zero when $1/p-  1/q >  2/d$. Therefore any estimate of the form \eqref{resolz1} is possible only if  $1/p-1/q \le 2/d$.  This gives the second condition in \eqref{nec_con}.  Finally, to show the last two conditions,  by duality we need only show one of them.  Suppose \eqref{resolz1} is true with $(p,q)=(1, d/(d-2))$. Then by \eqref{resolz1} and \eqref{beta0} we see that 
\[	\Big\| \mathcal F^{-1} \Big((1-\be_0(|\xi|))(|\xi|^2-z)^{-1} \wh f (\xi) \Big)\Big\|_\frac{d}{d-2}\le C\|f\|_1.	 \] 
We now note that $1-\be_0(|\xi|) = (1-\be_0(4|\xi|)) (1-\be_0(|\xi|)) $.   By Mikhlin's multiplier theorem (applied to the multiplier operator given by $(1-\be_0(4|\xi|))(|\xi|^2-z)|\xi|^{-2}$) we see that
\[	\Big\| \mathcal F^{-1} \Big((1-{\beta_0(|\xi|)})  |\xi|^{-2} \wh f (\xi)\Big)\Big\|_\frac{d}{d-2}\le C\|f\|_1.	 \] 
By scaling this also implies, for all $\epsilon>0$, 
\[ \Big\| \mathcal F^{-1} \Big((1-{\beta_0(\epsilon |\xi|)})  |\xi|^{-2} \wh f (\xi) \Big)\Big\|_\frac{d}{d-2}\le C\|f\|_1. \] 
Letting $\epsilon\to 0$ gives $ \| \mathcal F^{-1} (  |\xi|^{-2} \widehat f (\xi) )\|_\frac{d}{d-2}\le C\|f\|_1$ which is obviously not true.  Therefore we conclude the estimate \eqref{resolz1} cannot be true with $p=1$ and $q=d/(d-2)$. 

When $d=2$ the above argument works for $p,q$ satisfying $1/p-1/q<2/d$, but not for $p,q$ with $1/p-1/q=2/d$, that is, $p=1$, $q=\infty$ because Mikhlin's theorem does not hold with $q=\infty$. So we still need to show the failure of \eqref{resolz1} with $p=1$, $q=\infty$. But the failure can be shown in a more straightforward manner.  We need to prove that for any fixed $z\in\C\setminus [0,\infty)$ there does not exist a constant $C>0$ such that
\begin{equation}\label{2d1_infty}
	\|(-\De -z)^{-1}f\|_{L^\infty (\R^2)} \le C \|f\|_{L^1(\R^2)}, \quad \forall f\in L^1(\R^2).
\end{equation}
To show this let us assume \eqref{2d1_infty} and recall from \cite[p. 202]{Tes} that 
\[	(-\De-z)^{-1} f (x)= (2\pi)^{-1} \int _{\R^2} K_0(\sqrt{-z}\, |x-y |) f(y) dy,	\]
where $K_\nu(w)$ is the modified Bessel function of the second kind (see \cite{OM,Tes}). It is well-known (\cite[p. 252]{OM}) that 
\begin{equation}\label{bessel_asymp2}
	\lim_{w\to 0} \frac{K_0(w)}{-\ln w} =1.
\end{equation}
Let us choose a $\phi\in C^\infty_0(B_2(0,1))$ such that $\phi\ge 0$ and $\|\phi\|_1=1$, and set $\phi_\epsilon (x)= \epsilon^{-2}\phi(\epsilon^{-1}x)$, $\epsilon>0$. Testing \eqref{2d1_infty} with $f=\phi_\epsilon$ and letting $\epsilon \to 0$ yield $\sup_{x\in\R^2} \big| K_0(\sqrt{-z}\,|x|) \big| \lesssim 1$.  This contradicts the asymptotic \eqref{bessel_asymp2}, thus we conclude \eqref{2d1_infty} fails. \qed

\subsection {Proof of  Proposition \ref{shp}} \label{pf_lower_bound}
For a bounded function $m(\xi)$ which is symmetric under reflection $\xi\to -\xi$ it is easy to see that the $L^p$--$L^q$ norms of the operators $m(D)$ and $\overline{m}(D)$ are equal, hence we have
\[
\|(m\pm\overline{m})(D)\|_{p\to q}\le 2\|m(D)\|_{p\to q}.
\]
Therefore, to prove the lower bounds \eqref{shp2}, we may work with  the imaginary part
\[
\im  \Big(\frac{1}{|\xi|^2-z} \Big) =\frac 1{2i}\Big( \frac{1}{|\xi|^2-z}-\frac{1}{|\xi|^2-\bar{z}} \Big)=\frac{\im z}{(|\xi |^2-\re z)^2+(\im z)^2}.
\]
The lower bounds in \eqref{shp2} is meaningful only when $0<|\im z| \ll \re z <1$. Hence we only need to consider  $z=1+i\de$ with $0<100\de<1$ in this section. The lower bound $\gtrsim 1$ (in the case of $\gamma_{p,q}=0$) is clear since the resolvent operators are nontrivial. Thus, recalling \eqref{def-gamma} it is sufficient to show that 
\[
\Big\|\F^{-1}\Big(\frac{\de \wh f(\xi)}{(|\xi|^2-1)^2+\de^2}\Big) \Big\|_q \gtrsim \max\big\{\de^{-1+\frac{d+1}2(\frac 1p-\frac 1q)}, \de^{\frac{d-1}2-\frac dq}\big\} \|f\|_p. 
\]
for  $f\in L^p(\R^d)$. The other lower bound $\de^{\frac dp-\frac{d+1}2}$ in \eqref{shp2} follows from the  lower bound $\de^{\frac{d-1}2-\frac dq}$ by duality.  Since we also consider the resolvent estimates for the fractional Laplacian in Section \ref{fract} we  prove this in a slightly  more general  form. The argument is based on the counterexamples related to the restriction conjecture: a Knapp type example and the Fourier transform of the surface measure on the sphere (see \cite[pp. 387--388]{St-book}).
\begin{lem}\label{lower} Let $1\le p,q\le \infty$, and for $r, s >0$ let $m_\de^s(r) :=\frac{\de}{(r^s-1)^2+\de^2}$. Then, if  $0<\de <c$ for a small $c>0$,
\begin{align}
\label{beebsh_fr} 
\| m_\de^s(|D|) \|_{p\to q}  &\gtrsim \de^{-1+\frac{d+1}2(\frac 1p-\frac 1q)}, \\
\label{r4sh_fr} 
\| m_\de^s(|D|) \|_{p\to q}  &\gtrsim \de^{\frac{d-1}2-\frac dq},
\end{align}
where the implicit constants depend only on $p$, $q$, $s$ and $d$.
\end{lem}
\begin{proof}[Proof of \eqref{beebsh_fr}] \label{knapp}
Let $c,k$ be positive constants to be chosen later, depending only on $d$ and $s$. If $|\xi_j|\le c\sqrt \de$, $j=1,\dots, d-1$ and $k\de/4\le \xi_d-1\le k\de$,  then
\[	0< |\xi|^s-1 \le \left(1+\left( (d-1)c^2+2k+ 10^{-2}{k^2} \right)\de \right)^{s/2}-1	\]
provided that $0<100\de\le 1$.  Let us set $\mu(t):=(1+t)^{s/2}$ and $M_s:= \max\{|\mu ''(t)|: 0\le t\le 1\}$.  Then, from Taylor's theorem it follows that 
\[	0< |\xi|^s-1 \le  s\left( (d-1)c^2+2k+10^{-2}{k^2} \right)\de	\]
if we choose $\de$ small enough, that is to say,  $0<\big( (d-1)c^2 +2k +10^{-2}{k^2} \big) \de\le \min\big\{1, s/M_s\big\}$. We now choose $c=1/ \sqrt{2(d-1)s}$ and $k$ as the positive solution of the quadratic equation $2k+10^{-2} k^2=1/2s$ so that, for $0< \de \le \min\big\{s, s^2/ M_s\big\}$, 
\[	0< |\xi|^s-1 \le \de.	\]  
Let us set $c_s :=\min\big\{ 10^{-2}, s, s^2/ M_s \big\}$, and choose $\phi,\, \psi \in C^\infty_0(\R)$ such that $\supp\phi\subset [-1, 1]$, $0\le \phi \le 1$, $\phi=1$ on $[-1/2, 1/2]$,  $\supp \psi\subset [1/4, 1]$, $0\le \psi\le 1$ and $\psi=1$ on $[1/2, 3/4]$. 
For every $\de\in(0,c_s)$ we define $f_\de\in\mathcal S (\R^d)$ by 
\[
\wh{f_\de}(\xi)= \psi\Big( \frac{\xi_d-1}{k\de} \Big) \prod_{j=1}^{d-1} \phi\Big( \frac{\xi_j}{c\sqrt \de}\Big).
\]
Since $| \int m_\de^s(|\xi|) \wh{ f_\de}(\xi) e^{ix\cdot\xi} d\xi| = | \int  m_\de^s(|\xi|) \wh {f_\de} (\xi) e^{i(x\cdot\xi-x_d)} d\xi|$, we have 
\[	\Big| \int m_\de^s(|\xi|) \wh{ f_\de}(\xi) e^{ix\cdot\xi} d\xi \Big| 
		\ge  \Big| \int m_\de^s(|\xi|) \wh {f_\de} (\xi)\cos(x\cdot\xi-x_d) d\xi \Big| -  \Big| \int m_\de^s(|\xi|) \wh {f_\de} (\xi)\sin(x\cdot\xi-x_d) d\xi \Big|.	\]
On $\supp\wh {f_\de}$,  $|\xi_j|\le c\sqrt \de$ for $j=1,\cdots, d-1$ and $k\de/4\le \xi_d-1\le k\de$. Thus $m_\de^s(|\xi|) \ge 1/2\de$ whenever $\xi\in \supp\wh{ f_\de}$ and $0<\de<c_s$.  
Also, if $\xi\in \supp \wh {f_\de}$ and  
\[x\in\mathcal A_\de:=\Big\{ x\in\R^d: |x_j|\le \frac{1}{200(d-1)c\sqrt \de},\, j=1,\cdots, d-1,\, |x_d|\le\frac{1}{200k\de} \Big\},\]
 then $|x\cdot\xi-x_d|\le 1/100$, hence
\[	\Big| \int_{\R^d} m_\de^s(|\xi|) \wh{ f_\de}(\xi) e^{ix\cdot\xi} d\xi \Big|
		\ge \frac1{2\de} \int \wh{ f_\de}(\xi) \Big( 1-\frac1{100} \Big) d\xi - \frac1{100\de} \int \wh {f_\de}(\xi) d\xi \approx \de^{-1} \|\wh {f_\de}\|_{1} \approx \de^\frac{d-1}{2}.	\]
Integration on the box $\mathcal A_\de$ yields
\[	\Big\| \int_{\R^d} m_\de^s(|\xi|) \wh {f_\de}(\xi) e^{ix\cdot\xi} d\xi\Big\|_{q} \gtrsim \de^\frac{d-1}{2} |\mathcal A_\de|^{1/q} \approx \de^{\frac{d-1}{2}-\frac{d+1}{2q}}.	\]
On the other hand it is easy to check that $\|f_\de\|_p \approx \de^{\frac{d+1}2-\frac{d+1}{2p}}$. Thus we obtain \eqref{beebsh_fr}. 
\end{proof}

\begin{proof}[Proof of \eqref{r4sh_fr}] 
Let $\phi$ be a non-negative smooth function on $\R$ such that  $\supp \phi\subset ( 1-2\varepsilon_\circ, 1+2\varepsilon_\circ) $ for some small $\varepsilon_\circ>0$ to be determined later depending on $s$. 
We take $f\in C^\infty_0(\R^d)$ so that  $\wh {f}(\xi)=\phi(|\xi|)$ and set 
\[	Q(x):=\int_{\R^d} m_\de^s(|\xi|) \wh{ f}(\xi) e^{ix\cdot\xi} d\xi.	\] 
By the spherical coordinate we write 
\[	Q(x)	=  (2\pi )^\frac d2 \int_{1-2\varepsilon_\circ}^{1+2\varepsilon_\circ} m_\de^s(r)\, \phi(r) r^{d-1} |rx|^{\frac{2-d}{2}} J_{\frac{d-2}2}(|rx|) dr,	\]
where $J_\nu $ denotes the Bessel function of order $\nu$. It is well-known (see \cite[p. 338]{St-book}) that for $\nu> -1/2$ 
\begin{equation}\label{bessel_asymp}
J_\nu(r)=\left(\frac{\pi r}2\right)^{-\frac 12}\cos \left( r-\frac {\pi \nu}{2} -\frac \pi 4\right) +R_\nu(r), \quad r>0,
\end{equation}
where $R_\nu$ satisfies $|R_\nu(r)|\le c_\nu r^{-3/2}$ if $r\ge 1$. 

By \eqref{bessel_asymp} with $\nu=\frac{d-2}2$ and the formula $\cos(u+v)=\cos u\cos v- \sin u\sin v$,  we write $Q$ as follows:
\[Q(x)=Q_1(x)-Q_2(x)+Q_3(x),\] 
where 
\begin{align}
\label{Q11}	Q_1(x) &:= 2(2\pi)^\frac{d-1}{2} |x|^\frac{1-d}{2}\cos\Big( |x|- \frac{\pi(d-1)}{4} \Big) \int_{1-2\varepsilon_\circ}^{1+2\varepsilon_\circ} m_\de^s(r)\, \phi(r) r^\frac{d-1}{2} \cos\left( (r-1)|x| \right) dr, \\
\nonumber	Q_2(x) &:=2(2\pi)^\frac{d-1}{2} |x|^\frac{1-d}{2}\sin\Big( |x|- \frac{\pi(d-1)}{4} \Big) \int_{1-2\varepsilon_\circ}^{1+2\varepsilon_\circ} m_\de^s(r)\,\phi(r) r^\frac{d-1}{2} \sin\left( (r-1)|x| \right) dr, \\
\nonumber	Q_3(x) &:= (2\pi)^\frac{d}{2} \int_{1-2\varepsilon_\circ}^{1+2\varepsilon_\circ} m_\de^s(r)\, \phi(r) r^{d-1} |rx|^\frac{2-d}{2} R_\frac{d-2}{2}(|rx|) dr.
\end{align}
We now  split the domain of the integral $\int_{1-2\varepsilon_\circ}^{1+2\varepsilon_\circ} m_\de^s(r) \phi(r) r^\frac{d-1}{2} \cos( (r-1)|x| ) dr$ of  $Q_1(x)$ into subintervals on which $|r-1|\lesssim \de$ and $|r-1|\gtrsim \de$, respectively.  To be precise let us set $k(t):=\int_{|\ta|\le t} \frac1{\ta^2+1}d\ta$ and fix a large $\la>0$ such that
\begin{equation}\label{klda}
k(s\lambda) \ge 100 (\pi-k(s\lambda)  ). 
\end{equation}
Clearly, such $\lambda$ exists since $\int_{-\infty}^\infty \frac 1{\ta^2+1} d\ta =\pi$. Let $\mu$ be a small number so that $\la\mu\le 10^{-2}$, and  let
\[	\mathcal A':=\Big \{x\in\R^d: \frac{\mu}{4\de} \le |x| \le \frac{\mu}{2\de}\Big \}.	\]
Put $\psi(r):=r^s-1$, $r>0$ and set $M_s :=\max \{ |\psi''(r)|:|r-1|\le 1/2 \}$.  By Taylor's theorem, $|\psi(r)-s(r-1)|\le {M_s}(r-1)^2/2$ when $|r-1|\le 1/2$. Thus, if $|\psi(r)|\le s\la \de$, then 
\[	|r-1|\le \frac{M_s}{2s}(r-1)^2 +  \frac{|\psi(r)|}s \le \frac{M_s\varepsilon_\circ}{s}|r-1| + \la \de, \quad \forall r\in\supp \phi \cap [1/2,3/2].	\]
Choosing $\varepsilon_\circ:=\min \{ s/(2M_s), 1/4 \}$ we have $|r-1|\le 2\la \de$ on $\supp \phi$ whenever $|\psi(r)|\le s\la \de$. Therefore, if $x\in\mathcal A'$ and $|\psi(r)|\le s\la \de$, then $|(r-1)x|\le \la\mu\le10^{-2}$, so $\cos((r-1)|x|)\ge 99/100$ on $\supp\phi$.

Now we break the integral part of  $Q_1(x)$ as the following:
\[	I_1(x)+I_2(x):=\Big( \int_{|\psi(r)|\le s\la \de}+\int_{|\psi(r)|> s\la \de} \Big) \frac{\de}{\psi(r)^2+\de^2} \,\phi(r) r^\frac{d-1}{2}\cos \left((r-1)|x|\right) dr.	\]
If $x\in\mathcal A'$, by the above choice of $\lambda$ and $\varepsilon_\circ$, we have 
\[	I_1(x) \ge \frac{99}{100 } \int_{|\psi(r)|\le s\la \de} \frac{\de}{\psi(r)^2+\de^2} \, \phi(r) r^{\frac{d-1}{2} -s+1} \frac{d\psi(r)}{s}.	\]
Hence, we choose $\phi\in C_0^\infty((1-2\varepsilon_\circ, 1+2\varepsilon_\circ))$ such that $\phi(r)r^{ \frac{d-1}{2} -s+1}=1$ if $|r-1| \le \varepsilon_\circ$, and $0\le \phi(r) r^{\frac{d-1}{2} -s+1} \le 2$ for all $r$. Thus it follows that, if $x\in\mathcal A'$ and $\de\le \varepsilon_\circ/(2\la)$, then
\[	I_1(x) \ge \frac{99}{100 s} \int_{|t|\le s\la \de} \frac{\de}{t^2+\de^2} dt = \frac{99}{100 s} k(s\la).	\]
On the other hand, by our choice of $\phi$  and \eqref{klda}
\[	|I_2(x)| \le \int_{|\psi(r)|> s\la \de} \frac{\de}{\psi(r)^2+\de^2} \, \phi(r) r^{\frac{d-1}{2}-s+1} \frac{d\psi(r)}{s} 
		\le \frac 2s \big(\pi- k(s\la)\big)\le \frac{k(s\la)}{50s}.	\]
Therefore we have, for $x\in\mathcal A'$ and $0< \de \le \varepsilon_\circ/(2\la)$, 
\begin{equation}\label{Q1}
I_1(x)+I_2(x)\ge I_1(x)-|I_2(x)| \ge \frac{97 k(s\la)}{100 s}.
\end{equation}

For each $n\in\N$ let us set
\[	\mathcal A_n:= \Big\{x\in \mathcal A' : \pi \Big( 2n+\frac{d-1}{4} \Big) - \frac{1}{100} \le |x| \le \pi \Big( 2n+\frac{d-1}{4} \Big)+\frac{1}{100} \Big\},	\]
which is nonempty only if $n\approx \de^{-1}$.  We also set 
\[	\mathcal A:=  \bigcup_{n\in \N}A_n.	\] 
From now on suppose that  $x\in\mathcal A$. It is clear that $\cos (|x|- \pi(d-1)/4 ) \ge 99/100$, and  $| \sin (|x|- \pi(d-1)/4)| \le 1/100$.  Hence, when  $0<\de \le \varepsilon_\circ/(2\la)$ it follows from  \eqref{Q11} and  \eqref{Q1} that 
\[	Q_1(x) \ge 2(2\pi)^{\frac{d-1}{2}}  |x|^{\frac{1-d}2} \cdot \frac{99}{100}\cdot \frac{97 k(s\la)}{100 s}.	\]
Similarly,  we get
\[	|Q_2(x)| \le 2(2\pi)^{\frac{d-1}{2}}  |x|^{\frac{1-d}2} \cdot \frac{2}{100s} \cdot \frac{101\, k(s\la)}{100}.	\]
Combining the estimates for $Q_1(x)$ and $|Q_2(x)|$, we have 
\[	Q_1(x)-Q_2(x)\gtrsim |x|^{\frac{1-d}2}  \approx \de^\frac {d-1}{2},	\]
where the implicit constants depend only on $d$ and $s$. On the other hand, it is straightforward to see that $Q_3(x)=O(\de^{\frac{d+1}2})$ {as} $ \de\to 0$.  Thus, with sufficiently small $\de$,  we have that 
\[	Q(x)=Q_1(x)-Q_2(x)+Q_3(x)\gtrsim   \de^\frac {d-1}{2}, \quad \forall x\in \mathcal A. \] 
Observe that we can choose a  constant $c\in(0,1/10)$, independent of all small $\de>0$, such that $|\mathcal A| \ge c|B(0,\mu(2\de)^{-1})|$. Therefore, we conclude that
\[	\normo{Q}_{L^q(\mathcal A)} \gtrsim \de^{\frac{d-1}{2}-\frac dq}	\]
for all sufficiently small $\de>0$, with the implicit constant depending only on $d,s$.  Meanwhile, $f=\F^{-1}(\phi(|\cdot|))\in L^p(\R^d)$ for any $p\in[1,\infty]$.  Thus, this completes the proof of \eqref{r4sh_fr}. 
\end{proof}

\section{Sharp resolvent estimates for the fractional Laplace operators}\label{fract}
The  resolvent estimates for $(-\Delta)^{\frac s2}$ can be obtained by making use of the argument we have used for  the resolvent estimates for $-\Delta$ (the case $s=2$). In technical aspect there is not  much difference, but it is worthwhile to record the result for the operator $((-\Delta)^{\frac s2}-z)^{-1}$.   We shall be brief, but include the statements of results and sketch their proofs.   In what follows we consider $s\in (0,d)$ though generalization to  $s\ge d$  is also possible. 

We begin with the following which can be shown by adapting the proof of Proposition \ref{eresb}, so we state it without proof. 
\begin{prop}\label{eresb_fr} 
Let $d\ge2$, $0<s<d$,  $1\le p, q \le \infty$ and let $z\in\C\setminus [0,\infty)$. Then, $\|((-\Delta)^{\frac s2}-z)^{-1}\|_{p\to q}<\infty$ if and only if $(1/p, 1/q)\in \mathcal R^s_0$ which is given by
\[	\mathcal R^s_0=\mathcal R^s_0(d) :=
		\Big\{(x,y) \in I^2 : 0\le x, y\le 1,~ 0\le x-y \le  \frac sd \Big\} \setminus \Big\{ \Big(1, \frac{d-s}d \Big) , \Big(\frac sd, 0 \Big) \Big\}.	\]
\end{prop}
We introduce some notations which  we need  to state our results. For $d\ge 2$ and $0<s< d$, let us define $\mathcal R^s_1=\mathcal R^s_1(d),$ $\mathcal R^s_2=\mathcal R^s_2(d)$, and $\mathcal R^s_3=\mathcal R^s_3(d)$ by 
\[     \mathcal R^s_1:= \mathcal P(d) \cap \mathcal R^s_0(d), \quad
     \mathcal R^s_2 := \big(\mathcal T(d)\cap \mathcal R^s_0(d)\big)\setminus \big( [D,H)\cup [D',H)\big), \quad 
	\mathcal R^s_3  := \mathcal Q(d) \cap \mathcal R^s_0(d).	\]
Note that  if $0<s<\frac{2d}{d+1}$, $\mathcal R^s_1 = \emptyset$ (see Figure \ref{fig3}), and  $\mathcal R^2_i= \mathcal R_i$ for $i=0, 1, 2, 3$. If $0<s< 2$, then $\mathcal R^s_i= \mathcal R_i\cap \mathcal R^s_0$ for every $i$. From the definition of $\gamma_{p,q}$ (recall \eqref{def-gamma}), it  follows that \eqref{gamma} holds with $\mathcal R_i$ replaced by ${\mathcal R}^s_i$, $i=1,2,3$. 

\begin{figure}
\captionsetup{type=figure,font=scriptsize}
\begin{minipage}[b]{0.45\textwidth}
\centering
\begin{tikzpicture} [scale=0.6]\scriptsize
	\path [fill=lightgray] (0,0)--(15/4,15/4)--(50/11, 40/11)--(5,5)--(10-40/11,10-50/11)--(10-15/4,10-15/4)--(10,10)--(10,7.5)--(2.5,0)--(0,0);
	\draw [<->] (0,10.7)node[above]{$y$}--(0,0) node[below]{$(0,0)$}--(10.7,0) node[right]{$x$};
	\draw (0,10) --(10,10)--(10,0) node[below]{$(1,0)$};
	\draw (4,4)node[above]{$D$}--(5.5,3);
	\draw (6,6)--(7,4.5);
	\draw (0,0)--(4,4);
	\draw [dash pattern={on 2pt off 1pt}]  (4,4)--(6,6);
	\draw [dash pattern={on 2pt off 1pt}] (4,0)node[below]{$\frac 2d$}--(10,6);	
	\draw [dash pattern={on 2pt off 1pt}] (6,0)node[below]{$E$}--(6, 8/3)node[above]{$B$}--(10-8/3, 4)node[left]{$B'$}--(10, 4)node[right]{$E'$};
	\draw [dash pattern={on 2pt off 1pt}] (6, 8/3)--(11/2,3);
	\draw [dash pattern={on 2pt off 1pt}] (10-8/3, 4)--(7, 10-11/2);
	\draw (6,6)--(10,10);
	\draw (2.5,0)node[below]{$\frac sd$}--(10,7.5);
	\draw [dash pattern={on 2pt off 1pt}] (15/4,15/4)node[left]{$P_*$}--(50/11, 40/11)node[below] {$P_\circ$}--(5,5)node[above]{$H$}--(10-40/11,10-50/11)--(10-15/4,10-15/4);
	\draw (5.6, 4.4) node{$\wt{\mathcal R}^s_2$};
	\draw (3.2, 2) node{$\wt{\mathcal R}^s_3$};
	\draw (8, 6.8) node{$\wt{\mathcal R}^{s\prime}_3$};
\end{tikzpicture}\caption{The case $d\ge3$, $0<s<\frac{2d}{d+1}$}\label{fig3}
\end{minipage}\hfill
\begin{minipage}[b]{0.46\textwidth}
\centering
\begin{tikzpicture} [scale=0.6]\scriptsize
	\path [fill=lightgray] (0,0)--(15/4,15/4)--(50/11, 40/11)--(5,5)--(10-40/11,10-50/11)--(10-15/4,10-15/4)--(10,10)--(10,4.4)--(5.6,0)--(0,0);
	\draw [<->] (0,10.7)node[above]{$y$}--(0,0) node[below]{$(0,0)$}--(10.7,0) node[right]{$x$};
	\draw (0,10) --(10,10)--(10,0) node[below]{$(1,0)$};
	\draw (4,4)node[above]{$D$}--(6,8/3)node[above]{$B$}--(10-8/3,4)node[left]{$B'$}--(6,6);
	\draw (0,0)--(4,4);
	\draw [dash pattern={on 2pt off 1pt}]  (4,4)--(6,6);
	\draw (6,6)--(10,10);
	\draw (5.6,0)node[below]{$\frac sd$}--(10,4.4);
	\draw [dash pattern={on 2pt off 1pt}] (4,0)node[below]{$\frac 2d$}--(10,6);
	\draw [dash pattern={on 2pt off 1pt}] (6, 8/3)--(6,0)node[below]{$E$};
	\draw [dash pattern={on 2pt off 1pt}] (10-8/3, 4)--(10, 4)node[right]{$E'$};
	\draw [dash pattern={on 2pt off 1pt}] (15/4,15/4)node[left]{$P_*$}--(50/11, 40/11)node[below] {$P_\circ$}--(5,5)node[above]{$H$}--(10-40/11,10-50/11)--(10-15/4,10-15/4);
	\draw (5.6, 4.4) node{$\wt{\mathcal R}^s_2$};
	\draw (30/8, 2) node{$\wt{\mathcal R}^s_3$};
	\draw (8, 50/8) node{$\wt{\mathcal R}^{s\prime}_3$};
	\draw (7.1, 2.9) node{$\mathcal R^s_1$};
\end{tikzpicture}\caption{The case $d\ge3$, $\frac{2d}{d+1}\le s< d$}\label{fig4}
\end{minipage}
\end{figure}

As before, by scaling we have $\| ((-\Delta)^{\frac s2} - z)^{-1}\|_{p\to q} = |z|^{-1+\frac ds(\frac 1p-\frac 1q)} \| ((-\Delta)^{\frac s2} - |z|^{-1}z)^{-1}\|_{p\to q}$ for  $z\in \C\setminus [0,\infty)$.  Using Lemma \ref{lower} we may repeat the argument in the proof of  Proposition \ref{shp} to get
\[	\|((-\Delta)^{\frac s2}-z)^{-1}\|_{p\to q} \gtrsim \dist(z,[0,\infty))^{-\gamma_{p,q}}, \quad   z\in \zs.	\]
Combining these two,  we get, for $z\in\C\setminus [0,\infty)$, 
\begin{equation}\label{lower_fr}	
\| ((-\Delta)^{\frac s2} - z)^{-1}\|_{p\to q} \gtrsim {\mathlarger \kappa}_{p,q}^s(z):= |z|^{-1+\frac ds(\frac 1p-\frac 1q)+\gamma_{p,q}} \dist(z,[0,\infty))^{-\gamma_{p,q}},	
\end{equation} 
and we may conjecture the following which is a natural extension of Conjecture \ref{main_conj}.
\begin{conj}\label{conj_fr}
Let $d\ge 2$, $0<s < d$ and let $(1/p, 1/q)\in \big( \bigcup_{i=1}^3\mathcal R_i^s \big) \cup \mathcal R_3^{s\prime}$.  There exists an absolute constant $C$, depending only on $p$, $q$, $d$ and $s$, such that,   for $z\in \C\setminus [0,\infty)$,
\begin{equation}\label{shp_resol_fr}	
C^{-1}  { \mathlarger \kappa}_{p,q}^s(z) \le \|((-\Delta)^{\frac s2}-z)^{-1}\|_{p\to q} \le C  { \mathlarger \kappa}_{p,q}^s(z).
\end{equation}
\end{conj}
When $d\ge3$ let us set $\wt{\mathcal R}^s_2 := \mathcal R_0^s(d)\cap \wt{\mathcal R}_2(d)$ and $\wt{\mathcal R}^s_3:={\mathcal R}^s_3(d)\setminus [D,P_\circ, P_\ast]$. We have the following. 

\begin{thm}\label{fractional} 
Let $z\in \C\setminus[0,\infty)$. If $d=2$, Conjecture \ref{conj_fr} is true. If $d\ge3$, the conjectured estimate  \eqref{shp_resol_fr} is true whenever $(1/p,1/q)\in  \mathcal R^s_1\cup \big( \bigcup_{i=2}^3 \wt{\mathcal R}^s_i\big) \cup \wt{\mathcal R}^{s\prime}_3$.  Furthermore, if $d\ge2$, for $p,q,s$ satisfying  $(1/p,1/q)\in\{B,B'\}$ and $\frac{2d}{d+1}\le s< d$ (see Figure \ref{fig4}), we have $\|((-\Delta)^{\frac s2}-z)^{-1}f\|_{q,\infty} \lesssim |z|^{-1+\frac{2d}{s(d+1)}} \|f\|_{p,1}$ and,  for $p,s$ satisfying $(\frac1p,\frac{d-1}{2d}) \in (B',E'] \cap \mathcal R^s_0$ and $\frac{2d}{d+1}< s< d$,  we have $\|((-\Delta)^{\frac s2}-z)^{-1}f\|_{\frac{2d}{d-1},\infty}\lesssim |z|^{-1+\frac ds(\frac1p-\frac1q)} \|f\|_p.$
\end{thm}
We  remark that Theorem \ref{thm} is a special case of  Theorem \ref{fractional}  when $d\ge 3$. If $\frac{2d}{d+1}\le s<d$ and $(1/p, 1/q)\in\mathcal R^s_1(d)$, Theorem \ref{fractional} covers the result by Huang, Yao and Zheng \cite[Theorem 1.4]{HYZ}.

\begin{proof}[Proof of Theorem \ref{fractional}] We basically follow the argument in Section \ref{proof_of_main_thm} with some modifications.  For $z\in \zs$ let  us set 
\[	m^s(\xi, z)=(|\xi|^s-z)^{-1}.	\]
Using the same functions  $\rho_0,$   $\rho_1, $ and   $\rho_2$ as in Section \ref{proof_of_main_thm} we break $m^s$ such that $m_j^s(\xi, z):=m(\xi,z)\rho_j(\xi),$ $ j=0, 1,2,$ and  $ m^s(\xi,z) = \sum_{j=0}^2 m_j^s(\xi, z).$  Since $|\partial^\alpha_\xi  m_2^s(\xi,z)|\lesssim |\xi|^{-s-|\alpha|}$ and $m_2^s$ is supported away from the orign, by the standard argument (for example, the proof of sufficiency part of Proposition \ref{eresb})  we see that $\|m_2^s(D, z) f\|_{q}\le C\|f\|_{p}$ if $(1/p,1/q)\in \mathcal R^s_0$.  Similarly,  since $m_1^s$ is supported in $B_d(0,1)$ and $|\partial^\alpha_\xi  m_1^s(\xi,z)|\lesssim |\xi|^{s-|\alpha|}$, we see that $\|m_1^s(D, z) f\|_{q}\le C\|f\|_{p}$ if $1\le p\le q\le \infty$.  Thus, we need only handle $m_0^s(D, z)$.  If $z \in \mathbb S^1(\theta_\circ)$,  $\partial^\alpha_\xi m_0^s({\xi}, z)$  is uniformly bounded.  So, we may assume $z \not\in \mathbb S^1(\theta_\circ)$ and we  are reduced to showing  \eqref{main_part} when $(1/p,1/q)\in  \mathcal R^s_1\cup \big( \bigcup_{i=2}^3 \wt{\mathcal R}^s_i\big) \cup \wt{\mathcal R}^{s\prime}_3$, and its (resticted) weak type variants when $(1/p,1/q)\in [B',E']\cap \mathcal R_0^s$. All of these estimates are contained  in Proposition \ref{main_prop}.
\end{proof}

\subsection{Region of spectral parameters for uniform estimate} 
Let $p$, $q$, $d$ and $s$ be as in Theorem \ref{fractional}, and let $\ell>0$. Making use of Theorem \ref{fractional} we can also describe the region
\[	\mathcal Z_{p,q}^s(\ell):= \big\{z\in \C\setminus [0,\infty) : { \mathlarger \kappa}_{p,q}^s(z) \le \ell  \big\}.	\] 
We  consider  three cases $0<s<\frac{2d}{d+1}$, $s=\frac{2d}{d+1}$, and $\frac{2d}{d+1}<s< d$, separately. Also, by duality it is sufficient to consider $p,q$ satisfying  $(1/p,1/q)\in  \mathcal R^s_1\cup \wt{\mathcal R}^s_2 \cup \wt{\mathcal R}^s_3$.  As before we set the homogeneity degree $\omega_{p,q}^s=\omega_{p,q}^s(d):=1-\frac ds (\frac1p-\frac1q ),$ which is in $[0,1]$ when $(1/p,1/q)\in \mathcal R_0^s$, and note that
\begin{align*}
\mathcal Z_{p,q}^s(\ell)
	& =\big\{z\in \C\setminus \{0\} : \re z\le 0, \, \ell |z|^{\omega_{p,q}^s} \ge 1 \big\}	\\
	& \qquad \cup \big\{z\in \C\setminus [0,\infty) : \re z>0, \, \ell |\im z|^{\gamma_{p,q}}\ge |z|^{\ga_{p,q}-\omega_{p,q}^s}\big\}.
\end{align*}
Since $\omega_{2,2}^s=1=\ga_{2,2}$ for any $s$ and $d$,  $\mathcal Z_{2,2}^s(\ell)$ is always the complement of the $\ell^{-1}$-neighborhood of $[0, \infty)$ (see Figure \ref{1nbd_removed} or Figure \ref{broad10(uniform)}). We also note that  $\mathcal Z_{p,q}^s(\ell)=\emptyset$ if $\omega_{p,q}^s=0$ and $\ell<1$.  In what follows we disregard the case $p=q=2$, and   the case $\ell<1$ whenever $\omega_{p,q}^s=0$. 

{\bf When $0<s<\frac{2d}{d+1}$.}  In this case, $\mathcal R_1^s=\emptyset$, and $\ga_{p,q}>0$ for all $p,q$ with $(1/p,1/q)\in \wt{\mathcal R}^s_2 \cup \wt{\mathcal R}^s_3$.  \emph{If  $\omega_{p,q}^s=0$}, $\mathcal Z_{p,q}^s(1)=\{z\in\C\setminus\{0\}: \re z\le 0\}$, and  $\mathcal Z_{p,q}^s(\ell)$ with $\ell>1$ is a complement of a planar cone such as in Figure \ref{fig_Sobolev_line2}. 
\emph{If $\omega_{p,q}^s>0$,}  one can easily check that $\ga_{p,q}-\omega_{p,q}^s>0$ for all $(1/p,1/q)\in \big(\wt{\mathcal R}_2^s\setminus\{H\}\big)\cup \wt{\mathcal R}_3^s$. Thus, $\mathcal Z_{p,q}^s(\ell)$ has profiles such as the regions in Figure \ref{p2_q9pt9}, Figure \ref{p2_q20div3}, or Figure \ref{p2_q5}.

{\bf When $s=\frac{2d}{d+1}$.} In this case, $\frac1p-\frac1q=\frac sd=\frac2{d+1}$ and $\mathcal R_1^s=(B, B')$. 
\emph{If $\omega_{p,q}^s=0$},     $\mathcal Z_{p,q}^s(\ell)= \C\setminus [0,\infty)$ for $(1/p,1/q)\in\mathcal R_1^s$ and $\ell\ge1$;  $\mathcal Z_{p,q}^s(1)$ is the left half-plane for $(1/p,1/q)\in \wt{\mathcal R}_3^s$;   $\mathcal Z_{p,q}^s(\ell)$ is a complement of a planar cone for $(1/p,1/q)\in\mathcal R_3^s$ and  $\ell>1$ (see Figure \ref{fig_Sobolev_line2});  {there is no $p,q$ satisfying $(1/p,1/q)\in \mathcal R_2^s$.} 
\emph{If $\omega_{p,q}^s>0$}, we consider the following two cases:

\vspace{-10pt}

\begin{itemize}
	\item 
	$(1/p,1/q)\in\wt{\mathcal R}_2^s\setminus\{H\}$: In this case, $\ga_{p,q}-\omega_{p,q}^s=0$ since $s=\frac{2d}{d+1}$. Hence, $\mathcal Z_{p,q}^s(\ell)$ is the complement of the $\ell^{-1/\ga_{p,q}}$-neighborhood of $[0,\infty)$.
	\smallskip
	\item 
	$(1/p,1/q)\in\wt{\mathcal R}_3^s$: Then,  $\ga_{p,q}-\omega_{p,q}^s= -\frac {d+1}2 \big(\frac1q-\frac{d-1}{d+1}(1-\frac1p) \big)>0$. Thus, the profiles of $\mathcal Z_{p,q}^s(\ell)$ take the forms of the regions in Figure   \ref{p2_q9pt9}, Figure \ref{p2_q20div3}, or Figure \ref{p2_q5}.
\end{itemize}

{\bf When $\frac{2d}{d+1}<s< d$.} The classification of the profiles of $\mathcal Z_{p,q}^s(\ell)$ is similar to that  in Section \ref{drawing_figures} where $s=2$.   Again we consider the cases $\omega_{p,q}^s=0$ and $\omega_{p,q}^s>0$, separately.  
\emph{If $\omega_{p,q}^s=0$,} there are only two cases   $(1/p,1/q)\in\mathcal R_1^s$ and $(1/p,1/q)\in \wt{\mathcal R}_3^s$. For the first  case
$\mathcal Z_{p,q}^s(\ell)=\C\setminus[0,\infty)$ when $\ell\ge1$, and for the latter $\mathcal Z_{p,q}^s(1)$ is the left half-plane and $\mathcal Z_{p,q}^s(\ell)$ with $\ell >1$ is the complement of a planar cone (Figure \ref{fig_Sobolev_line2}). \emph{If  $\omega_{p,q}^s>0$}, we consider the following cases: 
\vspace{-10pt}
\begin{itemize}
		\item $(1/p,1/q)\in \mathcal R_1^s$: Since $\ga_{p,q}=0$ and $0<\omega_{p,q}^s<\frac{s(d+1)-2d}{s(d+1)}$,   $\mathcal Z_{p,q}^s(\ell)=\{z\in\C\setminus[0,\infty): |z|\ge \ell^{-1/\omega_{p,q}}\}$. 
		\item  $(1/p,1/q)\in \wt{\mathcal R}_2^s\setminus\{H\}$:  Since $\ga_{p,q}-\omega_{p,q}^s=(\frac ds-\frac{d+1}2)(\frac1p-\frac1q) <0$, $\mathcal Z_{p,q}^s(\ell)$ is the complement of a neighborhood of $[0,\infty)$ which shrinks along the positive real line as $\re z\to \infty$. 
		\item  
	$(1/p,1/q)\in \wt{\mathcal R}_3^s$: Note that  $\gamma_{p,q}=\frac{d+1}2-\frac dp>0$, and $\gamma_{p,q}-\omega_{p,q}^s=\frac ds(\frac{1-s}p+\frac{s(d-1)}{2d}-\frac1q)$. We divide $\wt{\mathcal R}_3^s$ into  $\wt{\mathcal R}_{3,+}^s$, $\wt{\mathcal R}_{3,0}^s$, and $\wt{\mathcal R}_{3,-}^s$ which are given by
	\begin{align*}
		\wt{\mathcal R}_{3,\pm}^s &=\Big\{(x,y)\in \wt{\mathcal R}_3^s: \pm\Big(y-(1-s)x-\frac{s(d-1)}{2d} \Big)>0 \Big\},  \\
		\wt{\mathcal R}_{3,0}^s &=\Big\{(x,y)\in \wt{\mathcal R}_3^s : y=(1-s)x+\frac{s(d-1)}{2d} \Big\}.
	\end{align*}
	\begin{itemize}
		\item[$\dagger$]
		$(1/p, 1/q)\in\wt{\mathcal R}_{3,+}^s$: Since $\gamma_{p,q}-\omega_{p,q}^s<0$, $\mathcal Z_{p,q}(\ell)$ is the complement of a neighborhood of $[0,\infty)$ which shrinks along positive real line as $\re z\to \infty$. 
		\item[$\dagger$] 
		$(1/p, 1/q)\in\wt{\mathcal R}_{3,0}^s$: Then  $\gamma_{p,q}-\omega_{p,q}^s=0$ and  $\mathcal Z_{p,q}(\ell)$ is the complement of the $\ell^{-1/\ga_{p,q}^s}$-neighborhood of $[0,\infty)$. 
		\item[$\dagger$] 
		$(1/p, 1/q)\in\wt{\mathcal R}_{3,-}^s$:  Since $\gamma_{p,q}-\omega_{p,q}^s>0$,  
		$\mathcal Z_{p,q}(\ell)$ is the complement of a neighborhood of $[0,\infty)$ whose boundary asymptotically satisfies $|\im z|\approx (\re z)^{1-\omega_{p,q}^s/\gamma_{p,q}^s}$ when $\re z$ is large (Figure \ref{p2_q9pt9}, Figure \ref{p2_q20div3}, and Figure \ref{p2_q5}).
	\end{itemize}
\end{itemize}

\subsection{Location of the eigenvalues of $(-\Delta)^{\frac s2}+V$ }  
Finally, using Theorem \ref{fractional} we can obtain the following which describes location of eigenvalues of the fractional operator $(-\De)^{\frac s2}+V$ acting in $L^q(\R^d)$.  The proof is similar to that of Corollary \ref{location_eigen}.
\begin{cor} 
Let $d\ge2$, $0<s< d$, $(1/p,1/q)\in\mathcal R_1^s\cup \big(\bigcup_{i=2}^3 \wt{\mathcal R}_i^s \big)\cup \wt{\mathcal R}_3^{s,\prime}$, and let $C>0$ be the constant which appears in \eqref{shp_resol_fr}. Fix a positive number $\ell>0$ (we choose $\ell\ge1$ if $1/p-1/q=s/d$). Suppose that $\|V\|_{L^{\frac{pq}{q-p}}(\R^d)} \le t (C\ell )^{-1}$ for some  $t\in(0,1)$. Then, if $E\in\C\setminus[0,\infty)$ is an eigenvalue of $(-\Delta)^{\frac s2}+V$ acting in $L^q(\R^d)$, $E$ must lie in $\C\setminus \mathcal Z_{p,q}^s(\ell)$.
\end{cor}

\bibliographystyle{plain}

\end{document}